\documentclass[a4paper,11pt]{amsart}
\usepackage{amsthm, amsmath, amssymb,amsbsy,graphicx,mathrsfs,enumerate, bm}
\usepackage{amscd,wasysym}
\usepackage{tikz}
\usepackage[margin=3cm]{geometry}

\def\wid1{15.7cm}

\newcommand{\wb}{\circ}
\newcommand{\bb}{\bullet}

\newcommand\Tableau[2][\relax]{
   \ifx\relax#1\relax%
       \else 
         \foreach\box in {#1} { \filldraw[blue!30]\box+(-.5,-.5)rectangle++(.5,.5);
   }
       \fi
    \newcount\row\newcount\col
    \row=0
    \foreach \Row in {#2} {
       \col=1
       \foreach\k in \Row {
          \draw(\the\col,\the\row)+(-.5,-.5)rectangle++(.5,.5);
          \draw(\the\col,\the\row)node{\k};
          \global\advance\col by 1
       }
       \global\advance\row by -1
    }
}

\tikzset{dots/.style={thick,dotted}}

\numberwithin{table}{section}

\newlength\savedwidth

\usepackage{caption}
\DeclareCaptionLabelSeparator{period}{. }
\newtheorem{thm}{Theorem}[section]
\newtheorem{lem}[thm]{Lemma}
\newtheorem{cor}[thm]{Corollary}
\newtheorem{prop}[thm]{Proposition}

\newtheorem{defi}[thm]{Definition}

\theoremstyle{definition}
\newtheorem{remark}[thm]{Remark}
\newtheorem{exam}[thm]{Example}

\numberwithin{equation}{section}

\newcommand{\ls}[2]{{\vphantom{#2}}^{#1\!}{#2}}

\newcommand{\im}{\operatorname{im}}

\newcommand{\Aut}{\operatorname{Aut}}

\newcommand{\mZ}{\mathbb{Z}}

\newcommand{\mN}{\mathbb{N}}

\newcommand{\Hom}{\operatorname{Hom}}

\renewcommand{\rm}{\mathrm}

\newcommand{\End}{\operatorname{End}}

\newcommand{\thra}{\twoheadrightarrow}

\newcommand{\sm}{\setminus}

\newcommand{\lda}{\lambda}

\newcommand{\fr}{\mathfrak}

\newcommand{\scr}{\mathscr}

\newcommand{\ol}{\overline}

\newcommand{\vph}{\vphantom}

\renewcommand{\a}{\alpha}
\renewcommand{\b}{\beta}
\newcommand{\e}{\epsilon}
\newcommand{\g}{\gamma}
\renewcommand{\d}{\delta}
\renewcommand{\k}{\kappa}

\newcommand{\s}{\sigma}
\newcommand{\om}{\omega}
\newcommand{\Om}{\Omega}

\renewcommand{\fr}{\mathfrak}

\newcommand{\Bl}{\operatorname{Bl}}

\newcommand{\lan}{\langle}
\newcommand{\ran}{\rangle}

\newcommand{\Char}{\operatorname{char}}

\newcommand{\hght}{\operatorname{ht}}

\newcommand{\Lda}{\Lambda}

\newcommand{\Par}{\operatorname{\mathsf{Par}}}
\newcommand{\Std}{\operatorname{\mathsf{Std}}}

\newcommand{\cl}{\mathcal}

\newcommand{\wh}{\widehat}

\newcommand{\mtt}{\mathtt}

\newcommand{\qdim}{\operatorname{qdim}}

\newcommand{\res}{\operatorname{res}}
\newcommand{\vn}{\varnothing}

\newcommand{\ot}{\otimes}
\newcommand{\cont}{\operatorname{cont}}

\newcommand{\Shape}{\operatorname{\mathsf{Shape}}}
\newcommand{\Ab}{\operatorname{\mathsf{Ab}}}
\newcommand{\Asf}{\operatorname{\mathsf{A}}}
\newcommand{\cmod}{\text{-}\mathrm{mod}}
\newcommand{\noarrow}{\:\notslash\:}
\newcommand{\BK}{\mathsf{BK}}
\newcommand{\END}{\operatorname{END}}
\newcommand{\veps}{\varepsilon}

\newcommand{\rot}{\mathrm{rot}}

\newcommand{\wt}{\widetilde}

\newcommand{\bi}{\text{\boldmath$i$}}
\newcommand{\bj}{\text{\boldmath$j$}}
\newcommand{\bk}{\text{\boldmath$k$}}

\newcommand{\isoto}[1][]{\mathop{\xrightarrow[#1]{\rule{0pt}{.9ex}{\raisebox{-.6ex}[0ex][-.7ex]{$\mspace{4mu}\sim\mspace{3mu}$}}}}}
\newcommand{\vphi}{\varphi}

\allowdisplaybreaks

\title[RoCK blocks]{RoCK blocks, wreath products and KLR algebras}
\author{Anton Evseev}\thanks{The author was supported by the EPSRC grant EP/L027283/1.}
\subjclass[2010]{Primary~20C08; Secondary~20C30}

\address{School of Mathematics, University of Birmingham, Edgbaston, Birmingham B15 2TT, UK}
\email{a.evseev@bham.ac.uk}

\begin{document}

\begin{abstract}
We consider RoCK (or Rouquier) blocks of symmetric groups and Hecke algebras at roots of unity. We prove a conjecture of Turner asserting that a certain idempotent truncation of a RoCK block of weight $d$ of a symmetric group $\mathfrak S_n$ defined over a field $F$ of characteristic $e$ is Morita equivalent to the principal block of the wreath product $\fr S_e \wr \fr S_d$. This generalises a theorem of Chuang and Kessar that applies to RoCK blocks with abelian defect groups. Our proof relies crucially on an isomorphism between $F\mathfrak S_n$ and a cyclotomic Khovanov--Lauda--Rouquier algebra, and the Morita equivalence we produce is that of graded algebras. We also prove the analogous result for an  Iwahori--Hecke algebra at a root of unity defined over an arbitrary field. 
\end{abstract}

\maketitle

\section{Introduction}\label{sec:intro}

\subsection{The main result}\label{subsec:main}
Let $\xi$ be a fixed element of an arbitrary field $F$. We assume that 
there exists an integer $e\ge 2$ such that $1+\xi+\cdots +\xi^{e-1} = 0$ and let $e$ be the smallest such integer (the \emph{quantum characteristic} of $\xi$). We fix $e$, $F$ and $\xi$ throughout the paper. 

For an integral domain $\cl O$, an invertible element $\xi\in \cl O$ and an integer $n\ge 0$,
the \emph{Iwahori--Hecke algebra}
$\cl H_n(\cl O,\xi)$ is the $\cl O$-algebra defined by the generators 
$T_1,\ldots, T_{n-1}$ subject to the relations 
\begin{align}
\label{H1} (T_r-\xi) (T_r+1) &= 0    \quad && \text{for } 1\le r<n, \\
\label{H2} T_r T_{r+1} T_r &= T_{r+1} T_r T_{r+1} \quad && \text{for } 1\le r <n-1, \\
\label{H3} T_r T_s &= T_s T_r  \quad && \text{for } 1\le r,s< n \text{ such that } |r-s|>1. 
\end{align}
Throughout, we write $\cl H_n = \cl H_n (F,\xi)$. 
The algebra $\cl H_n$ is cellular, and hence $F$ is necessarily a splitting field for this algebra (see e.g.~\cite[Theorem 3.20]{Mathas1999}). 

 It is well known that the blocks of $\cl H_n$ are parameterised by the set 
\begin{equation}\label{Blen}
 \Bl_e (n) = \{ (\rho, d) \in \Par \times \mN \mid \rho 
\text{ is an } e\text{-core and } |\rho| + ed=n \}, 
\end{equation}
where $\Par$ is the set of all partitions.
We write $b_{\rho,d}$ for the block idempotent of $\cl H_n$ 
corresponding to $(\rho,d)\in \Bl_e (n)$, and $\cl H_{\rho,d} =b_{\rho,d} \cl H_n $ denotes the corresponding block (see Section~\ref{sec:rock} for details). 
Representation theory of \emph{RoCK} (or \emph{Rouquier}) blocks of $\cl H_n $ (see Definition~\ref{def:rock}) is much more tractable than that of blocks $\cl H_{\rho,d} $ in general. By a fundamental result of 
Chuang and Rouquier~\cite[Section 7]{ChuangRouquier2008}, 
for any $d\ge 0$ and any two $e$-cores $\rho^{(1)}$ and $\rho^{(2)}$, the algebras $\cl H_{\rho^{(1)}, d}$ and 
$\cl H_{\rho^{(2)}, d}$ are derived equivalent. 
Consequently, in order to understand the structure of an arbitrary block 
$\cl H_{\rho,d}$ up to derived equivalence, it suffices to give a description of the structure of each RoCK block up to derived equivalence. 
If $\xi =1$, then $e=\Char F$ is necessarily prime and $\cl H_n  \cong F\fr S_n$, where $\fr S_n$ denotes the symmetric group on $n$ letters. 
Chuang and Kessar~\cite{ChuangKessar2002} proved that, when $\xi=1$ and $d<\Char F=e$, 
a RoCK block $\cl H_{\rho,d} $ is Morita equivalent to 
the wreath product $\cl H_{\vn,1}  \wr \fr S_d$. Note that here $\cl H_{\vn,1} $ is the principal block of $F\fr S_e$ and that the result of~\cite{ChuangKessar2002} applies precisely to RoCK blocks of symmetric groups with abelian defect. In fact, the aforementioned theorems of Chuang--Rouquier and Chuang--Kessar are stronger, as they hold with $F$ replaced by an appropriate discrete valuation ring. 

When $d\ge \Char F$, the Morita equivalence of Chuang--Kessar 
no longer holds, as a RoCK block 
$\cl H_{\rho,d} $ has more isomorphism classes of simple modules than
$\cl H_{\vn,1}  \wr \fr S_d$. Nevertheless, Turner~\cite{Turner2009} conjectured in general (for $\xi=1$) that $\cl H_{\vn,1}  \wr \fr S_d$ is Morita equivalent to a certain idempotent truncation of a RoCK block. More precisely, for any integers $0\le m\le n$, view $\cl H_m$ as a subalgebra of $\cl H_n$ via the embedding $T_j \mapsto T_j$ for $1\le j<m$. For any $e$-core $\rho$ and $d\ge 0$, define 
\begin{equation}\label{frhod}
 f_{\rho,d} = b_{\rho,0} b_{\rho,1} \cdots b_{\rho,d} \in \cl H_{|\rho|+de} .
\end{equation}
Clearly, the factors in this product commute pairwise, so $f_{\rho,d}$ is an idempotent. The main result of this paper is the following theorem, which settles affirmatively~\cite[Conjecture 82]{Turner2009} (stated in \emph{loc.~cit.} for the case $\xi=1$). 

\begin{thm}\label{thm:main1}
 Let $\cl H_{\rho,d} $ be a RoCK block and $f=f_{\rho,d}$. 
Then we have an algebra isomorphism $f \cl H_{\rho,d}  f \cong \cl H_{\rho,0}  \otimes_F (\cl H_{\vn,1}  \wr \fr S_d)$. Hence, the algebra $f \cl H_{\rho,d}  f$ is Morita equivalent to
$\cl H_{\vn,1}  \wr \fr S_d$. 
\end{thm}

The second statement follows from the first one because $\cl H_{\rho,0} $ is 
a split simple algebra. 

\begin{remark}
 The formula defining the idempotent appearing in~\cite[Conjecture 82]{Turner2009} is different to~\eqref{frhod}, but the resulting idempotent is 
equal to $f_{\rho,d}$: see Proposition~\ref{prop:falt}.
\end{remark}

While Theorem~\ref{thm:main1} is stated purely in the language of representation theory of symmetric groups and Hecke algebras at roots of unity, the proof given in this paper 
relies crucially on the fact that $\cl H_{\rho,d} $ is isomorphic to a certain cyclotomic Khovanov--Lauda--Rouquier (KLR) algebra. A consequence of this fact is that each of the algebras $\cl H_{\rho,d} $, $f_{\rho,d} \cl H_{\rho,d}  f_{\rho,d}$ and $\cl H_{\vn,1}  \wr \fr S_d$  has
a natural $\mZ$-grading. Moreover, 
$\cl H_{\vn,1}  \wr \fr S_d$ is nonnegatively graded: this observation plays an important role in the proof. 
The isomorphism and the Morita equivalence in Theorem~\ref{thm:main1} are those of graded algebras (see Theorem~\ref{thm:main2} for a more precise statement).

In order to explain the meaning of Theorem~\ref{thm:main1} in more detail, we recall certain well-known general facts on idempotent truncation (see e.g.~\cite[Section 6.2]{Green1980}).
 Let $A$ be an algebra over a field $k$ and $\veps\in A$ be an idempotent. Let $A\cmod$ be the category of left modules over $A$. 
Then we have an exact functor $\cl F\colon A\cmod \to \veps A\veps\cmod$ defined as follows: for any $A$-module $V$, set $\cl F(V) = \veps V$, and for any morphism $\phi\colon V\to W$ of $A$-modules, set $\cl F(\phi) = \phi|_{\veps V}$. 
Further, the image $\cl F(D)$ of any simple $A$-module $D$ is either simple or $0$, and, if $\{ D_{\lda} \mid \lda \in \Lda\}$ is a complete and irredundant set of representatives of isomorphism classes of simple $A$-modules,
then $\{ \veps D_{\lda} \mid \lda \in \Lda,\, \veps D_{\lda} \ne 0\}$ is a complete and irredundant set of representatives of isomorphism classes of simple $\veps A\veps$-modules. Informally, $\veps A\veps\cmod$ captures the part of the structure of $A\cmod$ that corresponds to the simple modules $D\in A\cmod$ such that $\veps D\ne 0$. In particular, if $\veps D\ne 0$ for all simple $A$-modules $D$, then $A$ is Morita equivalent to $\veps A\veps$. 

When $A=\cl H_{\rho,d} $ is a RoCK block and $\veps = f_{\rho,d}$, it is the case that $\veps D\ne 0$ for all simple $A$-modules $D$ if and only if $d<\Char F$ or $\Char F=0$: see Proposition~\ref{prop:Mor}.
Thus, when $d<\Char F$ or $\Char F=0$, Theorem~\ref{thm:main1} 
yields a Morita equivalence between the RoCK block $\cl H_{\rho,d}$ and 
$\cl H_{\vn,1} \wr \fr S_d$. This equivalence was proved by Chuang and Miyachi~\cite[Theorem 18]{ChuangMiyachi2010} under the assumption that either $\Char F=0$ or  $\xi$ belongs to the prime subfield of $F$ 
by using the aforementioned result of Chuang--Kessar (for $\xi=1$) and similar 
results obtained independently by Miyachi~\cite{Miyachi2001} 
and Turner~\cite{Turner2002} for RoCK blocks of finite general linear groups (for $\xi \ne 1$). 
The proof given below is different from the arguments in the above papers. The isomorphism in Theorem~\ref{thm:main1} is constructed uniformly for all cases and is quite explicit, once the statement of the theorem is translated into the language of KLR algebras.

In the case when $\xi=1$ and $e=2$, Theorem~\ref{thm:main1} was proved by Turner (see~\cite[Theorem 84]{Turner2009}) using a Brauer morphism.
Independently of the present work, 
the same result was proved for $e=2$ and arbitrary $\xi$ by 
Konishi~\cite{Konishi2015}.

In the case when $\Char F=0$, the decomposition matrix 
of a RoCK block $\cl H_{\rho,d} $ was determined by Chuang and Tan~\cite[Theorem 1.1]{ChuangTan2002} and, independently, by 
Leclerc and Miyachi~\cite[Corollary 10]{LeclercMiyachi2002}. 
After a certain relabelling, this matrix may be seen to be identical to the decomposition matrix of 
$\cl H_{\vn,1}  \wr \fr S_d$. 
When $\xi=1$, the decomposition matrix of a RoCK block $\cl H_{\rho,d} $ was determined by Turner~\cite[Theorem 132]{Turner2009} and was shown to coincide with that
of $\cl H_{\vn,1}  \wr \fr S_d$ by Paget~\cite[Theorem 3.4]{Paget2005}, again, after a certain relabelling.
These results are closely related to Theorem~\ref{thm:main1} but are not directly implied by it, as we do not describe explicitly the $\cl H_{\vn,1} \wr \fr S_d$-modules which are the images of Specht and simple modules of $\cl H_{\rho,d} $ under the composition of the Morita equivalence of Theorem~\ref{thm:main1} and the functor $\cl F$.

 In addition to conjecturing the statement of Theorem~\ref{thm:main1}, Turner~\cite{Turner2009} has constructed two remarkable algebras that he conjectured to be Morita equivalent respectively to the whole RoCK block $\cl H_{\rho,d} $ and to a RoCK block of a $\xi$-Schur algebra (see~\cite[Conjectures 165 and 178]{Turner2009} respectively). 
 After the present paper was submitted, 
 the first of these conjectures was proved 
 in~\cite{EvseevKleshchev2016} using results contained here.



\subsection{Outline of the paper}
Section~\ref{sec:rock} contains the definition of a RoCK block. In Section~\ref{sec:klr}, we recall the definition of KLR algebras, state some of their standard properties and 
state a graded version of Theorem~\ref{thm:main1} as Theorem~\ref{thm:main2}. The proof of Theorem~\ref{thm:main2} occupies Sections~\ref{sec:comb}--\ref{sec:surj}. A detailed outline of the proof is given in~\S\ref{subsec:outline}, after the required notation is introduced. 

In Section~\ref{sec:concl}, we prove two simple results that have already been referred to above and clarify Theorem~\ref{thm:main1}. In Section~\ref{sec:alt}, we give two alternative descriptions of the images in $f_{\rho,d} \cl H_{\rho,d} f_{\rho,d}$ of elementary transpositions of $\fr S_d$ under 
the isomorphism of Theorem~\ref{thm:main1}, specifically,
\begin{enumerate}[(i)]
\item an explicit formula for those images in terms of generators of the relevant KLR algebra (given without proof), see~Equation~\eqref{tauexpl}; 
\item a formula in terms of the generators $T_r$ of $\cl H_{|\rho|+de}$ and the grading on 
$f_{\rho,d} \cl H_{\rho,d} f_{\rho,d}$, see Proposition~\ref{prop:Xires}; in the case when $\xi=1$, we give a simple description of the whole isomorphism in these terms, not just of its restriction to $\fr S_d$: see Theorem~\ref{thm:Hecke}.  
\end{enumerate}
These results provide different viewpoints on the isomorphism and may be useful for determining the images of simple and Specht modules of $\cl H_{\rho,d}$ under the Morita equivalence of Theorem~\ref{thm:main1} composed with the functor from $\cl H_{\rho,d}\cmod$
to $f_{\rho,d} \cl H_{\rho,d}f_{\rho,d}\cmod$ described above.

\subsection{General notation}
The symbol $\mN$ denotes the set of nonnegative integers. 
For integers $n>r>0$, we denote by $s_r$ the elementary transposition $(r,r+1) \in \fr S_n$.
If $\cl O$ is a commutative ring, then $\cl O^{\times}$ denotes the set of all invertible elements of $\cl O$. The centre of an algebra $A$ is denoted by $Z(A)$. 
Throughout, subalgebras of $F$-algebras are not assumed to contain the identity element and $F$-algebra homomorphisms are not assumed to preserve the identity unless they are described as ``unital''. 

By a \emph{graded} vector space (algebra, module) we mean a $\mZ$-graded one.
If $V$ is a graded vector space, then $V_{\{n\}}$ denotes its $n$-th homogeneous component, so that $V=\bigoplus_{n\in \mZ} V_{\{n\}}$.  If $v\in V$, we write $v_{\{n\}}$ for the 
$n$-th component of $v$, so that $v_{\{n\}}\in V_{\{n\}}$ and $v=\sum_{n\in \mZ} v_{\{n\}}$.
For a subset $S\subset \mZ$, we set 
$V_S = \bigoplus_{n\in S} V_{\{n\}}\le V$ and $v_{S} = \sum_{n\in S} v_{\{n\}}$. We abbreviate $V_{<0}$ for $V_{\mZ_{<0}}$, etc.

If the graded vector space $V$ is finite-dimensional, then its graded dimension is defined as
$\qdim V = \sum_{n\in \mZ} (\dim V_{\{n\}}) q^n \in \mZ[q,q^{-1}]$.  
If $A$ is a graded algebra and $m\in \mZ$, then the graded algebra
$A\lan m\ran$ is defined to be the same algebra as $A$ with the grading given by
$A\lan m\ran_{\{n\}} = A_{\{n-m\}}$ for all $n\in\mZ$. 

If $U,V$ are $F$-vector spaces, we write $U\ot V$ for $U\ot_F V$. 
If $X$ is a subset of an $F$-vector space, then $FX$ denotes the $F$-span of $X$.
If $X$ and $Y$ are vector subspaces of an algebra $A$, we write 
$XY=F\{ xy \mid x\in X, y\in Y\}$. 
For a symbol $x$, we often use the notation $x^m$ as an abbreviation for $x,\ldots,x$ ($m$ entries) or $(x,\ldots,x)$ (as appropriate). Also, we write $x^{\ot m} = x\ot \cdots \ot x$.

\subsection*{Acknowledgement}
I am grateful to Alexander Kleshchev for helpful comments that led to improvements in the paper and, in particular, for pointing out the relevance of ``R-matrices'' for KLR algebras constructed in~\cite{KangKashiwaraKim2013}: this resulted in a considerable simplification of the proof of Theorem~\ref{thm:main1}, which previously relied on lengthy explicit computations in KLR algebras.

\section{RoCK blocks}\label{sec:rock}

If $\lda = (\lda_1,\ldots,\lda_r)$ is a partition (so that $\lda_1\ge \cdots \ge \lda_r>0$ are integers), we write $\ell(\lda)=r$ and $|\lda|=\sum_{j=1}^r \lda_j$. 
Let $\cl O$ be an integral domain and $t\in \cl O^{\times}$. 
For any partition $\lda$ of an integer $n$, let $S^{\lda, \cl O,t}$ be the Specht $\cl H_n (\cl O,t)$-module defined as in~\cite[Section 3.2]{Mathas1999}.
We write $S^{\lda} = S^{\lda,F,\xi}$. 
Note that $S^{\lda}$ is the dual of the ``Specht module'' associated with $\lda$ constructed in~\cite{DipperJames1986}.

For the definition of an $e$-core and $e$-weight of a partition, we refer the reader e.g.~to~\cite[Chapter 2]{JamesKerber1981}. 
If $n\ge 0$ and $(\rho,d)\in \Bl_e (n)$ (cf.~\eqref{Blen}), 
then we define $b_{\rho,d}\in \cl H_{n}$
to be the unique block idempotent of $\cl H_n$ such that $b_{\rho,d} S^{\lda}=S^{\lda}$ for all partitions $\lda$ of $n$ with $e$-core $\rho$ and $e$-weight $d$ (see~\cite[Theorem 4.13]{DipperJames1987}). 
The corresponding block algebra is $\cl H_{\rho,d}:= b_{\rho,d} \cl H_n$.

We introduce notation related to abacus representations of partitions (see~\cite[Section 2.7]{JamesKerber1981}). 
Let $\lda$ be a partition and $l\ge 1$, $N\ge \ell(\lda)$ be integers.
We set $\lda_r=0$ for every integer $r>\ell(\lda)$.  
The abacus display $\Ab_N^{l} (\lda)$ is the subset of $\mN \times \{ 0,\ldots,l-1\}$ defined by the property that 
$(t,i)\in \Ab_N(\lda)$ if and only if $l t+i \in \{\lda_1+N-1,\lda_2+N-2,\ldots,\lda_{N}\}$, whenever $t\ge 0$ and $0\le i<l$. 
The set $\mN \times \{ 0,\ldots,l-1\}$ is visualised as a table with infinitely many rows and $l$ columns. The columns $0,\ldots,l-1$ are drawn from left to right, and the 
entries $(0,i),(1,i),(2,i),\ldots$ of each column are drawn from the top down. 
The columns $\mN \times \{ i\}$ are referred to as \emph{runners}, and the elements of 
$\Ab_N^{l} (\lda)$ are referred to as \emph{beads}. 
In particular, the number of beads of $\Ab_N^{l} (\lda)$ on runner $i$ is defined as
$|\Ab_N^{l} (\lda) \cap (\mN \times \{i\})|$. 
We will write $\Ab_N (\lda) = \Ab^e_N (\lda)$, and we view $\Ab^1_N (\lda)$ 
as a subset of $\mN$ in the obvious way.

\begin{defi}\label{def:rock} \cite[Definition 52]{Turner2009}
 Let $\rho$ be an $e$-core. We say that $\rho$ is a \emph{Rouquier core} for an integer 
$d\ge 1$ if there exists an integer $N\ge \ell (\rho)$ such that for all 
$i=0,\ldots,e-2$, the abacus display 
$\Ab_N (\rho)$ has at least $d-1$ more beads on runner $i+1$ than on runner $i$. 
In this case, the block 
$\cl H_{\rho,d} $ is said to be a \emph{RoCK block}.  
\end{defi}

If $\rho$ is a Rouquier core for some $d\ge 1$ and $N\ge \ell(\rho)$ is such that there is an abacus display as above, then
 $\kappa=-N+e\mZ \in \mZ/e\mZ$ will be called a \emph{residue} of $\rho$. It is easy to show that $\rho$ has only one residue; in particular, this fact is a consequence of Lemma~\ref{lem:CK}. 

\begin{exam}\label{ex:rock1}
Let $e=3$ and $\rho = (8,6,4,2,2,1,1)$. 
The abacus display of $\rho$ for $N=7$ is 
\[
\begin{matrix}
\wb & \bb & \bb \\
\wb & \bb & \bb \\
\wb & \wb & \bb  \\
\wb & \wb & \bb \\
\wb & \wb & \bb \\
\wb & \wb & \wb \\
\wb & \wb & \wb \\
 & \vdots &  
\end{matrix}
\]
where the elements of the set $\Ab_7 (\rho)$ are represented by $\bb$.
Thus,
$\cl H_{\rho,d}$ is a RoCK block of residue $2+3\mZ$ for each $d=1,2,3$.  
\end{exam}

\section{KLR presentation of $\cl H_{n}$}\label{sec:klr}

\subsection{The KLR algebra and the Brundan--Kleshchev isomorphism}\label{subsec:klr1}

KLR algebras (also called quiver Hecke algebras) were introduced by Khovanov and Lauda~\cite{KhovanovLauda2009} and independently by Rouquier~\cite{Rouquier2008}. We follow the presentation given in~\cite{BrundanKleshchev2009a}. 

Let $I=\mZ/e\mZ = \{0,1,\ldots,e-1\}$. If $\bi\in I^n=I\times \cdots \times I$ for some $n\ge 0$, 
then $\bi_r$  denotes the $r$-th entry of $\bi$ 
(for $1\le r\le n$). Usually, we will write $i_r$ instead of $\bi_r$ and will use a similar convention for other bold symbols. 
If $\bi\in I^n$ and $\bj\in I^m$, we denote by $\bi \bj$
the concatenation $(i_1,\ldots,i_n,j_1,\ldots,j_m)$ of $\bi$ and $\bj$.

Consider the quiver $\Gamma$ with vertex set $I$, a directed edge from $i$ to $i+1$ for each $i\in I$ and no other edges. Write $i\to j$ if there is an edge from $i$ to $j$ but not from $j$ to $i$, $i\leftrightarrows j$ if there are edges between $i$ and $j$ in both directions,
and $i \noarrow j$ if $j\ne i,i\pm 1$. 
Let $\mtt{C}=(\mtt{c}_{ij})_{i,j\in I}$ be the corresponding generalized Cartan matrix (of type $A_{e-1}^{(1)}$): 
\[
 \mtt c_{ij} = 
\begin{cases}
 2 & \text{if } i=j, \\
 0 & \text{if } i\noarrow j \\
 -1 & \text{if } i \to j \text{ or } i \leftarrow j, \\
 -2 & \text{if } i \leftrightarrows j. 
 \end{cases}
\]

For $i,j\in I$, define polynomials $L_{ij} \in F[y,y']$ by 
\[
L_{ij} (y,y') =
\begin{cases}
 0 & \text{if } i=j \\
 1 & \text{if } i \noarrow j, \\
 y' - y  & \text{if } i \to j, \\
 y - y' & \text{if } i\leftarrow j, \\
 (y'-y)(y-y') & \text{if } i \rightleftarrows j.
\end{cases}
\]

The symmetric group $\fr S_n$ acts on $I^n$ as follows:
$w  (i_1,\ldots, i_n) = (i_{w^{-1} (1)}, \ldots, i_{w^{-1} (n)})$ for $w\in \fr S_n$. 
The \emph{KLR algebra} $R_n$ is the $F$-algebra generated by the set
\begin{equation}\label{eq:klrgens}
 \{ e(\bi) \mid \bi \in I^n \} \cup \{ y_1,\ldots, y_n \} \cup 
\{ \psi_1,\ldots,\psi_{n-1} \}
\end{equation}
subject to the relations
\begin{align}
  \label{rel:eid} e(\bi) e(\bj) &= \d_{\bi, \bj} e(\bi), \displaybreak[1] \\
 \sum_{\bi \in I^n} e(\bi) &=1, \displaybreak[1] \\
  y_r e(\bi) &= e(\bi) y_r, \displaybreak[1] \\
 \label{rel:epsi} \psi_r e(\bi) &= e(s_r \bi) \psi_r, \displaybreak[1] \\
 \label{rel:ycomm} y_r y_s &= y_s y_r, \\
 \label{rel:psicomm} \psi_r \psi_s &= \psi_s \psi_r \quad \hspace{-0.7mm}
  \text{if } |r-s|>1, \\
\label{rel:ypsi0}  \psi_r y_s &= y_s \psi_r \quad \text{if } s\ne r,r+1, \displaybreak[1] \\
\label{rel:ypsi1}  \psi_r y_{r+1} e(\bi) &= 
\begin{cases}
 (y_r \psi_r +1) e(\bi) & \text{if } i_r = i_{r+1}, \\
  y_r \psi_r e(\bi) & \text{if } i_r \ne i_{r+1}; \\
\end{cases} \displaybreak[1] \\
\label{rel:ypsi2} y_{r+1} \psi_r e(\bi) &= 
 \begin{cases}
  (\psi_r y_r +1) e(\bi) & \text{if } i_r = i_{r+1}, \\
   \psi_r y_r e(\bi) & \text{if } i_r \ne i_{r+1};
 \end{cases} \displaybreak[1] \\
 \psi_r^2 e(\bi) &= L_{i_r,i_{r+1}} (y_r, y_{r+1}) e(\bi), \displaybreak[1] \label{rel:psisq}\\
\psi_r \psi_{r+1} \psi_r e(\bi) &=
\begin{cases}
 (\psi_{r+1} \psi_r \psi_{r+1} +1) e(\bi) & \text{if } i_{r+2} = i_r \to i_{r+1}, \\
 (\psi_{r+1} \psi_r \psi_{r+1} -1) e(\bi) & \text{if } i_{r+2} = i_r \leftarrow i_{r+1}, \\ 
 (\psi_{r+1} \psi_r \psi_{r+1} - 2y_{r+1} +y_r +y_{r+2}) e(\bi) & \text{if } i_{r+2} = i_r \rightleftarrows i_{r+1}, \\
 \psi_{r+1} \psi_r \psi_{r+1} e(\bi) & \text{otherwise} \label{rel:psibr}
\end{cases}
\end{align}
for all $\bi,\bj\in I^n$ and all admissible $r$ and $s$. 

Let $(\fr h, \Pi, \Pi^{\vee})$ be a realization of the Cartan matrix $\mtt C$ (see~\cite[\S 1.1]{Kac1990}), with simple roots $\{ \a_i \mid i\in I\}$, 
simple coroots
$\{ \a_i^{\vee} \mid i\in I\}$ and fundamental dominant weights 
$\{ \Lda_i \mid i\in I\}$ satisfying 
$\lan \Lda_i, \a_j^{\vee} \ran = \d_{ij}$ for $i,j\in I$. 
Let $P_+ = \bigoplus_{i\in I} \mN \Lda_i$ and $Q_+ = \bigoplus_{i\in I} \mN \a_i$.
If $\a=\sum_{i\in I} n_i \a_i\in Q_+$, write 
$\hght(\a) = \sum_{i\in I} n_i$.


Let $\Lda\in P_+$. 
The \emph{cyclotomic KLR algebra} $R^{\Lda}_{n}$ is defined as the quotient of $R_n$ by the $2$-sided ideal generated by the set
$\{ y_1^{\lan \Lda,\a_{i_1}^{\vee}\ran} e(\bi) \mid \bi\in I^n \}$. 
It follows from the above relations that the algebras $R_n$ and $R^{\Lda}_{n}$ are both graded by the following rules: $\deg(e(\bi))=0$, 
$\deg(\psi_r e(\bi))= - \mtt c_{i_r,i_{r+1}}$ and 
$\deg(y_t) = 2$ whenever $\bi \in I^n$, $1\le r<n$ and $1\le t\le n$ (see~\cite{BrundanKleshchev2009a}). 

\begin{remark}\label{rem:KLdiag}
Elements of $R_n$ may be represented as linear combinations of diagrams described by Khovanov and Lauda~\cite{KhovanovLauda2009}. While diagrams are not used explicitly in our proof, the reader may find it helpful to translate some of the assertions below into the language of diagrams. 
\end{remark}

For each $i\in I$, define $\hat i\in F$ by 
\begin{equation}\label{bari}
 \hat i = \begin{cases} i & \text{if } \xi=1, \\
             \xi^{i} & \text{if } \xi \ne 1.
          \end{cases}
\end{equation}
Here, and in the sequel, if $\xi=1$, then $i$ is identified with an element of $F$ via the embedding $I=\mZ/e\mZ\to F$; and if $\xi\ne 1$, then $\xi^i=\xi^{i'} \in F$, where $i'\in \mZ$ is any representative of the coset $i$. 
Let $H^{\Lda}_{n}=H^{\Lda}_{n} (\xi)$ be the \emph{cyclotomic Hecke algebra} with parameter $\xi$; see~\cite{BrundanKleshchev2009a}.
That is, $H^{\Lda}_n$ is the $F$-algebra
generated by the set
$\{ T_1,\ldots,T_{n-1}, X_1,\ldots,X_n \}$
if $\xi=1$ and by the set
$\{ T_1,\ldots,T_{n-1}, X_1^{\pm 1},\ldots, X_n^{\pm 1}\}$ if $\xi\ne 1$
 subject to the relations~\eqref{H1}--\eqref{H3}, the relation
 $\prod_{i\in I} (X_1-\hat i)^{\lan \Lda,\a_i^{\vee}\ran}=0$
  and the following relations:
 \begin{enumerate}[(a)]
 \item if $\xi=1$:
 $X_l X_t = X_t X_l$, 
 $X_{r+1} = T_r X_r T_r + T_r$ and, if $t\notin \{r,r+1\}$, 
 $X_t T_r = T_r X_t$; 
 \item 
 if $\xi\ne 1$:
 $X_t^{\pm 1} X_l^{\pm 1} = X_l^{\pm 1} X_t^{\pm 1}$, $X_t X_t^{-1} =1=X_t^{-1} X_t$, 
 $X_{r+1} = \xi^{-1} T_r X_r T_r$ and, if $t\notin \{r,r+1\}$, $T_r X_t = X_t T_r$,
 \end{enumerate}
for $1\le r<n$, $1\le t,l\le n$.
In particular, $H^{\Lda_0}_{n}$ is isomorphic to $\cl H_n$ via the map given by $T_r\mapsto T_r$ for $1\le r<n$ and $X_1\mapsto \hat 0$. 
In the sequel, we identify these two algebras. 

Brundan and Kleshchev~\cite[Main Theorem]{BrundanKleshchev2009a} and independently Rouquier~\cite[Corollary 3.17]{Rouquier2008} proved that the algebra $H_n^{\Lda}$ is isomorphic to $R_n^{\Lda}$. More precisely, we have the following result.

\begin{thm}[Brundan--Kleshchev]\label{thm:BKiso_detailed}
Let $y$ and $y'$ be indeterminates. 
There exist power series
$P_i, Q_i \in F[[y,y']]$, $i\in I$, such that $Q_i$ is invertible for each $i$ and:
\begin{enumerate}[(i)]
\item\label{BKid1}
For each $n\ge 0$ and $\Lda\in P_+$, there is an isomorphism
$\BK_n^{\Lda} \colon H^\Lda_n \isoto R^\Lda_n$ given by 
\begin{align}\label{BK1}
 \BK_n^{\Lda} (T_r) &= 
 \sum_{\bi\in I^n} (\psi_r Q_{i_r-i_{r+1}} (y_r,y_{r+1}) - P_{i_r-i_{r+1}} (y_r,y_{r+1}))e(\bi) \quad \text{and} \\
 \BK_n^\Lda (X_t) &= 
 \begin{cases}
  \sum_{\bi\in I^n} (y_r + i_r) e(\bi) &\text{if } \xi=1, \\
  \sum_{\bi\in I^n} \xi^{i_r} (1-y_r) e(\bi) & \text{if } \xi\ne 1
 \end{cases} \label{BK2}
\end{align}
for $1\le r<n$ and $1\le t\le n$. 
\item\label{BKid2} For every $\bi\in I^n$ and $t\in \{1,\ldots,n\}$, the element $(\BK_n^{\Lda} (X_t)-\hat i_t) e(\bi)$ is nilpotent.
\item\label{BKid3} If $\xi=1$, then $P_i,Q_i\in F[[y-y']]$ for all $i\in I$. 
\end{enumerate}
\end{thm} 

\begin{proof}
 It is proved in~\cite[Sections 3 and 4]{BrundanKleshchev2009a}
 (see, in particular,~\cite[(3.41)--(3.42) and (4.42)--(4.43)]{BrundanKleshchev2009a})
 that 
 one has an isomorphism $H^{\Lda}_n \isoto R^{\Lda}_n$ given by 
 \[
  T_r \mapsto 
  \sum_{\bi\in I^n} (\psi_r Q'_{r,\bi} (y_r,y_{r+1}) - P'_{r,\bi} (y_r,y_{r+1}))e(\bi)
 \]
 and~\eqref{BK2}
 if power series $P'_{r,\bi}, Q'_{r,\bi}\in F[[y,y']]$, $1\le r<n$, $\bi\in I^n$, satisfy certain explicit identities.\footnote{We write $P'_{r,\bi}, Q'_{r,\bi}$ for the power series 
 denoted in~\cite{BrundanKleshchev2009a} by $P_r (\bi), Q_r (\bi)$ for $\xi \ne 1$ and 
 by $p_r (\bi), q_r (\bi)$ for $\xi=1$.}
 If $\xi=1$, then the power series $P'_{r,\bi}, Q'_{r,\bi}$ given by~~\cite[(3.22) and (3.30)]{BrundanKleshchev2009a} satisfy the required identities. Moreover, one easily checks that 
  $P'_{r,\bi} = P_{i_r-i_{r+1}}$
 and $Q_{r,\bi} = Q'_{i_r-i_{r+1}}$ for all $r,\bi$ provided the power series $P_i$ and $Q_i$, $i\in I$, are defined as follows:
 \begin{align}
 \label{Pideg}
  P_i &= 
  \begin{cases}
  1  & \text{if } i=0, \\
  (i+y-y')^{-1} \hspace{6.3mm} & \text{if } i\ne 0,
  \end{cases}\\
\label{Sideg}
   Q_i &= 
   \begin{cases}
   1+ y'-y & \text{if } i=0, \\
   1-P_i & \text{if } i\notin \{0,1,-1\} \\
   (1-P_i^2)/(y'-y) & \text{if } e\ne 2 \text{ and } i=-1, \\
   1 & \text{if } e\ne 2 \text{ and } i=1, \\
   (1-P_i)/(y'-y)&  \text{if } e=2 \text{ and } i=1.
   \end{cases}
\end{align}
  Since $P_i,Q_i\in F[[y-y']]$ for all $i$, 
  this proves~\eqref{BKid1} when $\xi=1$ and~\eqref{BKid3}. 
 
 Assuming that $\xi\ne 1$, let 
 \begin{align}
P_{i} &= 
\begin{cases}
1& \text{if } i=0,\\ 
(1-\xi)\left(1- \xi^{i} (1-y)(1-y')^{-1} \right)^{-1} & \text{if } i\ne 0,
\end{cases} \label{Pinondeg} \\
Q_{i}&=\begin{cases}
          1 - \xi + \xi y'-y \hspace{31.5mm} &  \text{ if } i=0, \\
         \frac{\xi^{i} (1-y) - \xi (1-y')}{ \xi^{i} (1-y)- (1-y')} & \text{ if } i\notin \{0,1,-1\}, \\
\frac{ \xi^{-1} (1-y) - \xi (1-y')}{ (\xi^{-1} (1-y) - (1-y'))^2}
  & \text{ if } e\ne 2 \text{ and } i=-1, \\
          1 & \text{ if } e\ne 2 \text{ and } i=1,  \\
\frac{1}{\xi (1-y')-(1-y)} & \text{ if } e=2 \text{ and } i=1 \label{Sinondeg}
 \end{cases}
\end{align}
for all $i\in I$. (The only difference from power series given by~\cite[(4.27) and (4.36)]{BrundanKleshchev2009a} is a slight one in the formulas for $Q_i$ in the cases when $i\in \{1,-1\}$.)
As in~\cite{BrundanKleshchev2009a}, one checks that the 
power series $P'_{r,\bi} = P_{i_r-i_{r+1}}$ and $Q_{r,\bi} = Q_{i_r-i_{r+1}}$ 
satisfy the required properties, i.e.~the identities~\cite[(4.27) and (4.33)--(4.35)]{BrundanKleshchev2009a}, so one has an isomorphism given by~\eqref{BK1}--\eqref{BK2}. 

Finally,~\eqref{BKid2} follows from~\eqref{BK2} and the fact that $y_t\in R^{\Lda}_n$ is nilpotent for $1\le t\le n$, as $\deg(y_t)=2$ and $R^{\Lda}_n\cong H^{\Lda}_n$ is well known to be finite-dimensional. 
\end{proof}

From now on, we assume that $\BK_n^{\Lda}$ is as in Theorem~\ref{thm:BKiso_detailed}, with $P_i$ and $Q_i$ given by~\eqref{Pideg}--\eqref{Sinondeg}. 
We write 
\[
\BK_n = \BK_n^{\Lda_0} \colon \cl H_n \isoto R^{\Lda_0}_n.
\] 

For $0\le m\le n$, we define a unital algebra homomorphism
$\iota_m^n \colon R_m^{\Lda_0} \to R_n^{\Lda_0}$ by
\begin{equation}\label{iotamn}
 e(\bi) \mapsto \sum_{\bj\in I^{n-m}} e(\bi\bj), \quad 
\psi_r \mapsto \psi_r, \quad y_t \mapsto y_t 
\end{equation}
for $\bi\in I^m$, $1\le r<m$ and $1\le t\le m$.
For every $\a\in Q_+$, define 
\[
 I^{\a} = \{(i_1,\ldots,i_n) \in I^n \mid \sum_{r=1}^n \a_{i_r} = \a \}
\]
and $e_{\a} = \sum_{\bi\in I^{\a}} e(\bi)$, viewed either as an element 
of $R_n$ or of $R^{\Lda_0}_n$, depending on the context. The element $e_{\a}$ is a central idempotent. 
We write $R_{\a} = R_n e_{\a}$ and $R^{\Lda_0}_{\a} = R_n^{\Lda_0} e_{\a}$. 

We identify every partition $\lda$ with its Young diagram, defined to be the set
\[
\{ (a,b)\in \mZ_{>0}\times \mZ_{>0} \mid 1\le a\le \ell(\lda), \, 1\le b\le \lda_a \}.
\]
 Whenever set-theoretic notation is used for a partition 
$\lda$, it is to be viewed as a Young diagram. 
The residue of a box $(a,b)\in \mZ_{>0}\times \mZ_{>0}$ 
is defined to be 
$\res ((a,b)) = b-a+e\mZ \in I$, and the \emph{residue content} of $\lda$ is defined 
as $\cont(\lda) = \sum_{(a,b) \in \lda} \a_{\res ((a,b))}\in Q_+$.
Note that if $\mu\subset \lda$ are partitions and $\lda\sm \mu$ is an $e$-hook (i.e.~a connected skew diagram of size $e$ containing no $2\times 2$-squares), then 
each element of $I$ occurs exactly once among the residues of the boxes of $\lda\sm \mu$. It follows that, if $\rho$ is the $e$-core and $d$ is the $e$-weight of $\lda$, then $\cont(\lda) = \cont(\rho) + d\d$, where $\d:=\a_0+\a_1+\cdots +\a_{e-1} \in Q_+$ is the fundamental imaginary root.

Recall that, if $m\le n$, then 
$\cl H_m $ is viewed as a subalgebra of $\cl H_n $ 
via $T_r\mapsto T_r$, $1\le r<m$. We state some standard properties of 
the isomorphism $\BK_n$.  

\begin{prop}\label{prop:BKiso}
Let $n\ge 0$ be an integer. 
 \begin{enumerate}[(i)]
  \item\label{BKiso1} For $0\le m\le n$, we have 
$\iota_m^n \circ \BK_m = \BK_n|_{\cl H_m }$.
  \item\label{BKiso2} For all $(\rho,d)\in \Bl_e (n)$, we have 
$\BK_n (b_{\rho,d}) = e_{\cont(\rho)+d\d}$. 
Hence, for every $(\rho,d)\in \Bl_e (n)$, the map $\BK_n$ restricts to an algebra isomorphism from $\cl H_{\rho,d} $ onto $R^{\Lda_0}_{\cont(\rho)+d\d}$. 
 \end{enumerate}
\end{prop}

\begin{proof}
\eqref{BKiso1} easily follows from~\eqref{BK1}, and 
~\eqref{BKiso2} follows from~\cite[\S 2.9 and Theorem 5.6(ii)]{Kleshchev2010}.
\end{proof}

Let $\cl H_{\rho,d} $ be a RoCK block of $\cl H_n $, and consider the idempotent 
$f_{\rho,d}\in \cl H_{\rho,d}$ defined by~\eqref{frhod}. 
By Proposition~\ref{prop:BKiso}, we have 
\begin{equation}\label{BKf}
\BK_n (f_{\rho,d}) = \prod_{r=0}^d \iota_{|\rho|+er}^{|\rho|+ed} (e_{\cont(\rho)+r\d})
=\sum_{\substack{\bj\in I^{\cont (\rho)} \\ \bi^{(1)},\ldots,\bi^{(d)} \in I^{\d}}} e(\bj \bi^{(1)} \ldots \bi^{(d)}).
\end{equation}
 In particular, $\BK_n (f_{\rho,d})\in R^{\Lda_0}_n$ is homogeneous of degree $0$.

If $A$ is a graded algebra over $F$, then the wreath product $A\wr \fr S_d$ is defined as the algebra $A^{\otimes d} \otimes F\fr S_d$ with multiplication given by 
\[
 (x_1\otimes \cdots \otimes x_d \otimes \sigma) \otimes 
(y_1 \otimes \cdots \otimes y_d \otimes \tau) = 
x_1 y_{\sigma^{-1}(1)} \otimes \cdots \otimes x_d y_{\sigma^{-1} (d)} \otimes \sigma\tau
\]
for $x_1,\ldots,x_d,y_1,\ldots,y_d\in A$ and $\sigma,\tau \in \fr S_d$. 
We identify $F\fr S_d$ with a unital subalgebra of $A\wr \fr S_d$ via the map
$\s\mapsto 1^{\ot d} \ot \s$, $\s\in \fr S_d$, and we identify 
$A^{\ot d}$ with the unital subalgebra $A^{\ot d} \ot 1$ of $A\wr \fr S_d$ in the obvious way.
If $A$ is graded, then we will view $A\wr \fr S_d$ as a graded algebra via the rule 
\[
 \deg(x_1\otimes \cdots \otimes x_d \otimes \sigma) = \sum_{r=1}^d \deg (x_r)
\]
whenever $x_1,\ldots,x_d\in A$ are homogeneous and $\sigma\in \fr S_r$. 

Observe that the algebra $R_{\d}$ is nonnegatively graded 
(that is, $(R_{\d})_{<0}=0$): this follows from the fact that 
$e(\bi) e_{\d} = 0$ if $\bi\in I^e$ is such that $i_r=i_{r+1}$ for some $r$. 
Hence, the wreath product 
$R^{\Lda_0}_{\d}\wr \fr S_d$ is also nonnegatively graded. 

We will prove the following result, which may be viewed as a graded version of Theorem~\ref{thm:main1} and clearly implies that theorem (due to Proposition~\ref{prop:BKiso}\eqref{BKiso2}).

\begin{thm}\label{thm:main2}
 Let $n\ge 0$ and $(\rho,d)\in \Bl_e (n)$ be such that $\rho$ is a Rouquier core for $d$. If $f=\BK_n (f_{\rho,d})$, then $fR_{\cont(\rho)+d\d}^{\Lda_0} f$ and
$R_{\cont(\rho)}^{\Lda_0} \otimes_F (R_{\d}^{\Lda_0} \wr \fr S_d)$ are isomorphic as graded algebras. 
\end{thm}

\subsection{The standard basis and some general properties of $R_{\a}$}

Fix $\a \in Q_+$, and let $n=\hght(\a)$.
Fow $w\in \fr S_n$, let $\ell(w)$ be the smallest $m$ such that $w=s_{r_1} \cdots s_{r_m}$ for some $r_1,\ldots,r_m\in \{1,\ldots,n-1\}$. An expression $w=s_{r_1}\cdots s_{r_m}$ is said to be \emph{reduced} if $m=\ell(w)$. 
 We write $<$ and $\le$ for the Bruhat partial order on $\fr S_{n}$. 
That is, $v\le w$ if 
and only if there is a reduced expression $w=s_{r_1}\cdots s_{r_k}$ such that $v= s_{r_{a_1}} \cdots s_{r_{a_t}}$ for some $a_1,\ldots,a_t$ satisfying $1\le a_1<\cdots<a_t \le k$
(cf.~e.g.~\cite[Section 3.1]{Mathas1999}).

If $w\in \fr S_n$, we set $\psi_w = \psi_{r_1} \cdots \psi_{r_m}\in R_n$ where $w=s_{r_1}\cdots s_{r_m}$ is an arbitrary but fixed reduced expression for $w$. Note that, in general, $\psi_w$ depends on the choice of a reduced expression. In the sequel, all results involving elements $\psi_w$ are asserted to be true for \emph{any} choice of reduced expressions used to construct $\psi_w$.

\begin{thm}\label{thm:basis}
\begin{enumerate}[(i)]
\item\label{basis1} \cite[Theorem 2.5]{KhovanovLauda2009} \cite[Theorem 3.7]{Rouquier2008} The set 
\[
 \{ \psi_w y_1^{m_1} \cdots y_n^{m_n} e(\bi) \mid w\in \fr S_n, m_1,\ldots, m_n\in \mN, \bi\in I^{\a} \}
\]
is a basis of $R_{\a}$. 
\item\label{basis2} 
Let $r_1,\ldots,r_k \in \{1,\ldots,n-1\}$, $w=s_{r_1}\cdots s_{r_k}$
and $g_0,\ldots,g_k \in F[y_1,\ldots,y_n]$.
Then 
$g_0 \psi_{r_1}g_1 \psi_{r_2} \cdots g_{k-1} \psi_{r_k}g_{k} e_{\a}$ belongs to the span of 
elements of the form
\begin{equation}\label{eqbasis}
 \psi_{r_{a_1}} \cdots \psi_{r_{a_l}} y_1^{m_1} \cdots y_n^{m_n} e(\bi)
\end{equation}
where $1\le a_1<\cdots <a_l\le k$, $m_1,\ldots,m_n\in \mN$, $\bi\in I^{\a}$
and the expression $s_{r_{a_1}} \cdots s_{r_{a_l}}$ is reduced. 
%
\end{enumerate}
\end{thm}

\begin{proof}[Proof of~\eqref{basis2}]
Using relations~\eqref{rel:ypsi0}--\eqref{rel:ypsi2} repeatedly, we see that
$g_0 \psi_{r_1}g_1 \psi_1 \cdots g_{k-1} \psi_{r_k}g_{k} e_{\a}$ belongs to the span of
the elements~\eqref{eqbasis} without the condition that 
$s_{r_{a_1}} \cdots s_{r_{a_l}}$ be reduced.
Now the result follows from~\cite[Proposition 2.5]{BrundanKleshchevWang2011}.
\end{proof}


Let $\mu=(\mu_1,\ldots,\mu_l)$ be a composition of $n$, i.e.~a sequence of 
nonnegative integers such that $\sum_{r=1}^l \mu_r=n$. 
Let 
\[
\fr S_{\mu}= \Aut(\{1,\ldots,\mu_1\}) \times \Aut(\{\mu_1+1,\ldots,\mu_1+\mu_2\}) \times \cdots \cong \fr S_{\mu_1} \times \cdots \times \fr S_{\mu_l}, 
\]
a standard parabolic subgroup of $\fr S_n$. 
Denote by $\scr D^{\mu}_n$ (respectively, $\ls{\mu}{\scr D}_n$) the set of the minimal length left (resp.~right) coset representatives of $\fr S_{\mu}$ in $\fr S_n$. 
Note that an element $\s\in \fr S_n$ belongs to $\scr D^{\mu}_n$ if and only if 
$\s(r)<\s(t)$ for all $r,t\in \{1,\ldots,n\}$ 
such that $t\in \fr S_{\mu} \cdot r$ and $r<t$; this fact and its analogue for $\ls{\mu}{\scr D}_n$ will be used repeatedly.
If $\nu$ is another composition of $n$, set ${\vph{\scr D_n}}^{\nu\vph{\mu}}\!\scr D_n^{\mu}= \ls{\nu}{\scr D}_n \cap \scr D^{\mu}_n$: this is the set of the minimal length double 
$(\fr S_\nu, \fr S_\mu)$-coset representatives in $\fr S_n$.

Let $\mu=(n_1,\ldots,n_{l})$ be a composition of $n$.
There is an obvious map
$\iota_{\mu} \colon R_{n_1} \ot \cdots \ot R_{n_l}\to R_n$, defined as the unique algebra homomorphism satisfying 
\begin{align*}
 \iota_{\mu} (e(\bi^{(1)}) \ot \cdots \ot e(\bi^{(l)})) 
&= e(\bi^{(1)} \ldots \bi^{(l)}), \\
\iota_{\mu} (1^{\ot r-1} \ot \psi_k \ot 1^{\ot l-r})
&= \psi_{n_1+\cdots+n_{r-1} + k}, \\
\iota_{\mu} (1^{\ot r-1} \ot y_t \ot 1^{\ot l-r})
&= y_{n_1+\cdots+n_{r-1} + t}
\end{align*}
whenever $\bi^{(1)}\in I^{n_1},\ldots,\bi^{(l)} \in I^{n_l}$, $1\le r\le l$, 
$1\le k<n_r$ and $1\le t\le n_r$.

A \emph{composition} of $\a$ is a tuple $(\g_1,\ldots, \g_{l})$ such that $\g_j \in Q_+$ for each $j$ and $\sum_{j=1}^l \g_j = \a$. 
Let $\g=(\g_1,\ldots,\g_{l})$ be a composition of $\a$, and set $\mu=\mu(\g):=(\hght(\g_1),\ldots,\hght(\g_l))$. 
By restricting the map
$\iota_{\mu}$ to $R_{\g_1} \ot \cdots \ot R_{\g_l}$, we obtain an algebra homomorphism
$\iota_{\g} \colon R_{\g_1} \ot \cdots \ot R_{\g_l} \to R_{\a}$.  
The image of this homomorphism will be denoted by $R_{\g}$ and the image of 
$e_{\g_1}\ot \cdots \ot e_{\g_l}$ will be denoted by $e_{\g}$, so that $e_{\g}=e_{\g_1,\ldots,\g_{l}}$ 
is the identity element of the subalgebra $R_{\g}$ of $R_{\a}$. 
If $w\in \fr S_n$, then $w= w' v$ for some (uniquely determined) $w'\in \scr D_{n}^{\mu}$ and $v\in \fr S_{\mu}$.
We then have $\ell(w) = \ell(w')+ \ell(v)$, and hence one can obtain a reduced expression for $w$ by concatenating reduced expressions for $w'$ and $v$. Using 
a reduced expression of this form to define each $\psi_w$, one deduces the following result from Theorem~\ref{thm:basis} (cf.~the proof of~\cite[Proposition 2.16]{KhovanovLauda2009}). 

\begin{cor}\label{cor:parabfree} 
For any composition $\g$ of $\a$, $R_{\a}e_{\g}$ 
is freely generated as a right $R_{\g}$-module by the set
$\{ \psi_w \mid w\in \scr D^{\mu(\g)}_n \}$.
\end{cor}

\begin{prop}\label{prop:double}
 Let $\g=(\g_1,\ldots,\g_{l})$ and $\g'=(\g'_1,\ldots,\g'_m)$ be 
compositions of $\a$. Let $\mu=\mu(\g)$ and $\nu=\mu(\g')$. 
Then 
$e_{\g'} R_{\a} e_{\g} = 
\sum_{v\in \ls{\nu}{\scr D}^{\mu}_n} R_{\g'} \psi_v R_{\g}$. 
\end{prop}

\begin{proof}
By Corollary~\ref{cor:parabfree}, 
$e_{\g'} R_{\a} e_{\g} = \sum_{w\in \scr D^{\mu}_n} e_{\g'} \psi_w R_{\g}$. Let $w\in \scr D^{\mu}_n$, and let $u\in \fr S_\nu$ and 
$v\in \ls{\nu}{\scr D}_n$ be such that $w=uv$ and $\ell(w) = \ell(u) + \ell(v)$. 
It is easy to show that
$v\in \ls{\nu}{\scr D}_n^\mu$. We may assume that $\psi_w$ is defined in such a way that $\psi_w = \psi_u \psi_v$ (for each $w$ in question). 
Thus, $e_{\g'} \psi_w = e_{\g'} \psi_u \psi_v \in R_{\g'} \psi_v$, and the result follows. 
\end{proof}

Define $R'_{\a}$ to be the subalgebra of $R_{\a}$ generated by 
\[
\{ e(\bi) \mid \bi \in I^{\a} \} \cup \{ \psi_r e_{\a} \mid 1\le r<n\} \cup \{ (y_r - y_t) e_{\a} \mid 1\le r,t\le n\}.
\]

The following fact was observed 
in~\cite[Lemma 3.1]{BrundanKleshchev2012} (in a slightly different context). For the reader's convenience, we give a proof.

\begin{prop}\label{prop:tildebasis}
 As a right $F[y_2-y_1,y_3-y_2,\ldots, y_n-y_{n-1}]$-module, $R'_{\a}$ is freely generated by the set $\{ \psi_w e(\bi) \mid w\in \fr S_n, \bi\in I^{\a} \}$.
\end{prop}

\begin{proof}
 The fact that the given set generates a free right module $U$ over $F[y_2-y_1,\ldots, y_n-y_{n-1}]$ follows immediately from Theorem~\ref{thm:basis}. 
It remains only to show that $U=R'_{\a}$. Clearly, $U\subset R'_{\a}$.  

Let $z$ be an indeterminate, and consider $F[z]\otimes R_{\a}$, which is 
an $F[z]$-algebra by extension of scalars, and hence an $F$-algebra. As is observed in~\cite[\S 1.3.2]{KangKashiwaraKim2013}, there is an $F$-algebra homomorphism
$\om\colon R_n \to F[z]\otimes R_n$ given by 
\begin{equation}\label{KaKaKitwist}
 e(\bi) \mapsto 1\otimes e(\bi), \quad \psi_r \mapsto 1\otimes \psi_r, \quad y_t \mapsto z\otimes 1 + 1\otimes y_t
\end{equation}
for $\bi\in I^n$, $1\le r<d$, $1\le t\le d$. Note that $\om(x) = 1\otimes x$ for all $x\in R'_{\a}$. 
Also, it follows from Theorem~\ref{thm:basis} that $R_{\a}$ is freely generated as a left $U$-module by the set $\{ y_1^j e_{\a} \mid j\ge 0\}$. 
Let $0\ne x\in R'_{\a}$, and write 
$x=\sum_{j=0}^m u_j y_1^j$ where $u_j\in U$ for all $j$ and $u_m\ne 0$. Then $1 \otimes x = \om(x) \in z^m \otimes u_m + \sum_{j=0}^{m-1} z^j \otimes R_{\a}$, which forces $m=0$. Hence, $x\in U$. 
\end{proof}

Let $\g=(\g_1,\ldots,\g_{l})$ be a composition of $\a$. 
We define $R'_{\g} = R_{\g} \cap R'_{\a}$. 
Note that $R'_{\g}$ need not be equal to 
$\iota_{\g}(R'_{\g_1} \ot \cdots \ot R'_{\g_{l}})$.

\begin{cor}\label{cor:tildeg}
 As a right $F[y_2-y_1,y_3-y_2,\ldots, y_n-y_{n-1}]$-module, $R'_{\g}$ is freely generated by the set $\{ \psi_w e(\bi) \mid w\in \fr S_{\mu(\g)}, \bi\in I^{\a} \}$.
\end{cor}

\begin{proof}
This is an immediate consequence of Corollary~\ref{cor:parabfree} and Proposition~\ref{prop:tildebasis}.
\end{proof}

\subsection{Outline of the proof of Theorem~\ref{thm:main2}}\label{subsec:outline}

We denote by $\Par(n)$ the set of all partitions of $n$ and by $\Par_e(\rho,d)$ the set of all partitions with $e$-core $\rho$ and $e$-weight $d$.
A \emph{standard tableau} of size $n$ 
is a map $\mtt t\colon \{1,\ldots,n\} \to \mZ_{>0}\times \mZ_{>0}$ such that $\mtt t$ is 
a bijection onto the Young diagram of a partition $\lda$ and the inverse of this bijection is increasing along the rows and columns of $\lda$. 
In this situation, we say that $\lda$ is the \emph{shape} of $\mtt t$ 
and write $\Shape(\mtt t) = \lda$.  
We write $\Std(\lda)$ for the set of all standard tableaux of shape $\lda$. 
The \emph{residue sequence} of a standard tableau $\mtt t$ is defined as 
$\bi^{\mtt t} = (\res(\mtt t(1)), \ldots, \res(\mtt t(n)))\in I^n$. 
For any $e$-core $\rho$ and $d\ge 0$, define $I^{\rho,d}$ to be the set of all 
$\bi\in I^n$ such that there exist $\lda\in \Par_e(\rho,d)$ and a standard tableau $\mtt t$ of shape $\lda$ satisfying $\bi^{\mtt t} = \bi$.


Let $(\rho,d)$ be an element of $\Bl_e (n)$ (for some $n\ge 0$) such that $\rho$ is a Rouquier $e$-core for $d$. Let $\kappa$ be the residue of the RoCK block $\cl H_{\rho,d}$. As in the statement of Theorem~\ref{thm:main2}, let $f=\BK_n (f_{\rho,d})$. 
For $m\ge 0$, $\bi\in I^m$ and $j\in I$, 
set $\bi^{+j} = (i_1+j, \ldots, i_m + j)$.
Let $I^{\vn,1}_{+j} = \{\bi^{+j} \mid \bi \in I^{\vn,1}\}$. 
Define the set 
\begin{equation}\label{defE}
 \cl E_{d,j} = \{ w  (\bi^{(1)} \ldots \bi^{(d)}) \mid w \in \scr D_{ed}^{(e^d)}, \, \bi^{(1)}, \ldots, \bi^{(d)} \in I^{\varnothing,1}_{+j} \}, 
\end{equation}
 and set $\cl E_d=\cl E_{d,0}$. The following  alternative description of 
$\cl E_{d,j}$ is verified easily:
a tuple $\bi\in I^{de}$ lies in $\cl E_{d,j}$ if and only if one can partition the set $\{1,\ldots,de\}$ into subsets $Y_1,\ldots,Y_d$ of size $e$ each such that for each $r=1,\ldots,d$, one has $(i_{a_1},\ldots, i_{a_e})\in I^{\vn,1}_{+j}$ where
$Y_r = \{ a_1,\ldots,a_e \}$ and $a_1<\cdots <a_e$.  

Define $\ol{R}_{d\d}$ to be the quotient of 
$R_{d\d}$ by the $2$-sided ideal generated by the set 
$\{ e(\bi) \mid \bi\in I^{d\d} \sm \cl E_d\}$ and 
$\wh{R}_{d\d}$ to be 
the quotient of $R_{d\d}$ by the $2$-sided ideal 
generated by $\{ e(\bi) \mid \bi\in I^{d\d} \sm \cl E_{d,\k}\}$. 
It is clear from the definition of a KLR algebra 
that $R_{d\d}$ has a graded automorphism given by 
\begin{equation}\label{rot1}
 e(\bi) \mapsto e(\bi^{+\k}), \quad \psi_r e_{d\d} \mapsto \psi_r e_{d\d}, 
\quad  y_t e_{d\d} \mapsto y_t e_{d\d}
\end{equation}
whenever $\bi\in I^{d\d}$, $1\le r<de$ and $1\le t\le de$.
This automorphism corresponds to a rotational symmetry of the quiver $\Gamma$. 
Further, the map~\eqref{rot1} clearly induces an isomorphism 
$\ol{R}_{d\d} \isoto \wh{R}_{d\d}$, which restricts to an isomorphism
$\rot_{\k} \colon e_{\d^d} \ol{R}_{d\d} e_{\d^d} \isoto e_{\d^d} \wh{R}_{d\d} e_{\d^d}$.

There is an obvious graded homomorphism $R_{d\d} \to R^{\Lda_0}_{\cont(\rho)+d\d}$, obtained as the composition of 
the natural projection $R_{\cont(\rho)+d\d}\thra R_{\cont(\rho)+d\d}^{\Lda_0}$ 
with the map $R_{d\d} \to R_{\cont(\rho)+d\d}$, 
$x\mapsto \iota_{\cont(\rho),d\d} (e_{\cont(\rho)} \ot x)$.
Using special combinatorial properties of RoCK blocks, we show in Section~\ref{sec:comb} that 
this map factors through $\wh{R}_{d\d}$ and hence induces a graded algebra homomorphism 
$\Om\colon e_{\d^d} \wh{R}_{\d^d} e_{\d^d} \to fR^{\Lda_0}_{\cont(\rho)+d\d} f$.
Further, the image $C_{\rho,d}$ of $\Om$ has the property that $fR^{\Lda_0}_{\cont(\rho)+d\d} f$ 
is isomorphic to $R^{\Lda_0}_{\cont(\rho)} \ot C_{\rho,d}$ as a graded algebra (see Propositions~\ref{prop:cent} and~\ref{prop:im_om}). 
Thus, it is enough to show that $R_{\d}^{\Lda_0} \wr \fr S_d\cong C_{\rho,d}$ as graded algebras. 

In Section~\ref{sec:delta}, we prove some elementary results on the 
structure of $R_{\d}^{\Lda_0}$, which are needed later. 
In Section~\ref{sec:Theta}, we construct a graded algebra homomorphism
$\Theta\colon R^{\Lda_0}_{\d} \wr \fr S_d \to \ol{R}_{d\d}$. This allows us to define a homomorphism $\Xi\colon R^{\Lda_0}_{\delta}\wr \fr S_d \to C_{\rho,d}$ as 
the composition 
\begin{equation}\label{comp}
  R^{\Lda_0}_{\d} \wr \fr S_d \xrightarrow{\Theta} e_{\d^d} \ol{R}_{d\d} e_{\d^d}
\xrightarrow{\rot_\k} e_{\d^d} \wh{R}_{d\d} e_{\d^d}
\xrightarrow{\Om} C_{\rho,d}. 
\end{equation}
In Section~\ref{sec:surj}, we show that $\Xi$ is surjective. Proposition~\ref{prop:qdim} states that $R^{\Lda_0}_{\d} \wr \fr S_d$ and $C_{\rho,d}$ have the same (graded) dimension, so we are then able to deduce that $\Xi$ is an isomorphism, which concludes the proof. 

The definition of the map $\Theta$, unlike those of $\rot_{\k}$ and $\Omega$, is far from straightforward. The crux of the proof is the construction in Section~\ref{sec:Theta} of appropriate elements 
$\tau_r = \Theta(s_r)\in e_{\d^d} \ol{R}_{d\d} e_{\d^d}$, 
where, as before, $s_r =(r,r+1)\in \fr S_d\subset R^{\Lda_0}_{\d} \wr \fr S_d$
for $1\le r<d$. 
In order to define $\tau_r$ and prove that they satisfy required relations, we adapt to the present context the ideas that 
Kang, Kashiwara and Kim~\cite{KangKashiwaraKim2013} use to construct homomorphisms (``R-matrices'') between certain modules over KLR algebras.  

The results of Section~\ref{sec:Theta} are stated purely in the language of KLR algebras and do not involve a Rouquier core $\rho$.
Intertwiners of the same flavour as $\tau_r$ appear to play an important role in representation theory of KLR algebras and were originally discovered (in a different context) by Kleshchev, Mathas, and Ram~\cite[Section 4]{KleshchevMathasRam2012}.
More recently, module endomorphisms that are closely related to the elements $\tau_r$ have been constructed by Kleshchev and Muth (for KLR algebras of all untwisted affine types), also using the approach of~\cite{KangKashiwaraKim2013}: see~\cite[Theorem 4.2.1]{KleshchevMuth2013}.\footnote{
For any fixed $\bi\in I^{\vn,1}$, Kleshchev and Muth find a certain element of $R_{de}$, denoted by $\sigma_r +c$ in~\cite[(4.2.3)]{KleshchevMuth2013}, such that the image of this element in $\ol{R}_{d\d}$ multiplied by our $\Theta(e(\bi)^{\ot d} \ot 1)$ is equal to 
$\Theta (e(\bi)^{\ot d} \ot s_r)$.}

\begin{remark}\label{rem:KM}
The main result of Section~\ref{sec:Theta} is Theorem~\ref{thm:Theta}, which gives a partial description of the algebra $e_{\d^d} \ol R_{d\d} e_{\d^d}$.
A similar (but more explicit) result has independently been obtained by Kleshchev and Muth~\cite{KleshchevMuth2015} for all KLR algebras of untwisted affine ADE
types.  More precisely, Theorem~\ref{thm:Theta} can be deduced from
~\cite[Theorem 5.9]{KleshchevMuth2015}, which describes a certain idempotent truncation of $e_{\d^d} \ol R_{d\d} e_{\d^d}$ as an affine zigzag algebra and is proved using explicit diagrammatic computations, which are generally avoided below. Also, 
Proposition~\ref{prop:d1basis} together with Lemma~\ref{lem:Rbardiso} are
equivalent to 
the type A case of~\cite[Corollary 4.16]{KleshchevMuth2015}.  
\end{remark}

\section{Combinatorics of a RoCK block}\label{sec:comb}


\subsection{The algebra $C_{\rho,d}$}\label{subsec:Crhod}
Let $\mtt t$ be a standard tableau of size $n\ge 0$. 
For any $m\le n$, we write 
$\mtt t_{\le m}= \mtt t|_{\{ 1,\dots,m\}}$.
The degree $\deg(\mtt t)$ of $\mtt t$ is defined  as follows (see~\cite[\S 4.11]{BrundanKleshchev2009}). For $i\in I$, an \emph{$i$-node} is a node $a\in \mZ_{>0}\times \mZ_{>0}$ of residue $i$. Let $\mu$ be a partition. 
For a node $a\in \mZ_{>0}\times \mZ_{>0}$, we say that $a$ is an \emph{addable} node for $\mu$ if 
$a\notin \mu$ and $\mu\cup \{a\}$ is the Young diagram of a partition, and we say that $\mu$ is a \emph{removable} 
node of $\mu$ if $a\in \mu$ and $\mu \sm \{a\}$ is the Young diagram of a partition. 
We say that a node $(r,t) \in \mZ_{>0}\times \mZ_{>0}$ is \emph{below} a node $(r',s')$ if $r>r'$. 
If $a$ is an addable $i$-node of $\mu$, define 
\begin{equation}\label{demu}
 d_a (\mu) = \# \{\text{addable $i$-nodes for $\mu$ below $a$} \} 
    - \# \{ \text{removable $i$-nodes of $\mu$ below $a$} \}. 
\end{equation}
Finally, define recursively
\begin{equation}\label{eq:deg}
 \deg(\mtt t) = 
\begin{cases}
 d_{\mtt t(n)} (\Shape(\mtt t)) + \deg(\mtt t_{\le n-1}) & \text{if } n>0, \\
0 & \text{if } n=0. \\
\end{cases}
\end{equation}

Recall the definition of the set $I^{\rho,d}\subset I^{\cont(\rho)+d\d}$ from~\S\ref{subsec:outline}
for any $e$-core $\rho$ and $d\ge 0$.

\begin{thm}\label{thm:BKdim} \cite[Theorem 4.20]{BrundanKleshchev2009}
 For any integer $n\ge 0$ and $\bi,\bj\in I^n$, we have 
\[
 \qdim \left( e(\bi) R^{\Lda_0}_{n} e(\bj) \right) = 
\sum_{\substack{\lambda\in \Par(n) \\ 
\mtt s, \mtt t \in \Std(\lda) \\ 
\bi^{\mtt s}=\bi, \, \bi^{\mtt t} = \bj}} q^{\deg(\mtt s)+\deg(\mtt t)}.
\]
In particular, if $(\rho,d)\in \Bl_e (n)$, then 
$e(\bi)e_{\cont(\rho)+d\d} \ne 0$ in $R^{\Lda_0}_{\cont(\rho)+d\d}$ 
if and only if $\bi\in I^{\rho,d}$.
\end{thm}

The second assertion of the theorem follows from the first one
because any two partitions with the same residue content have the same $e$-core,~see~\cite[Theorem 2.7.41]{JamesKerber1981}.

If $X$ is a subset of $\mN \times \{0,\ldots,e-1\}$ and $a,b\in \mN \times\{0,\ldots,e-1\}$ 
are such that $a\in X$ and $b\notin X$, then we say that the set 
$(X\sm \{a\}) \cup \{b\}$ is obtained from $X$ by the move $a\to b$. 
If $(t,i)\in \mN\times\{0,\ldots,e-1\}$, then we say that the \emph{next} node after $(t,i)$ is the unique node $(t',j)\in \mN\times\{0,\ldots,e-1\}$ such that $et'+j=et+i+1$ 
(i.e.~$(t',j)=(t,i+1)$ if $i<e-1$ and $(t',j)=(t+1,0)$ if $i=e-1$). 
We will use the following elementary fact.  

\begin{lem}\label{lem:stepres}
Let $N\ge 0$ and $\lda$ be a partition such that $\ell(\lda)\le N$. Let $a\in  \mN\times\{0,\ldots,e-1\}$ and $b=(t,i)$ be the next node after $a$.
If $a\in \Ab_N (\lda)$, $b\notin \Ab_N (\lda)$ and $\Ab_N (\mu)$ is obtained from $\Ab_N (\lda)$ by the move $a\to b$, then $\mu\sm \lda$ consists of a single node of residue 
$i-N+e\mZ$. 
\end{lem}

In the rest of this section, we assume that $\cl H_{\rho,d} $ 
is a RoCK block of residue $\kappa$ and that $\Ab_N (\rho)$ is an 
abacus display witnessing this fact. 
If $X$ and $Y$ are subsets of $\mZ_{>0} \times \mZ_{>0}$ and there exists $c\in \mZ \times \mZ$ such that 
$Y= \{x+c \mid x\in X\}$, then we say that $Y$ is a translate of $X$. The concept of two skew tableaux (viewed as maps $\{1,\ldots,n\}\to \mZ_{>0}\times \mZ_{>0}$) being translates of each other is defined similarly. 
A partition of the form $(k,1^{e-k})$ for some $k\in \{1,\ldots,e\}$ will be called an $e$-hook partition. The following lemma includes a key combinatorial property of RoCK blocks, proved by Chuang and Kessar. 

\begin{lem}\label{lem:CK}
Let  $1\le r<d$. Suppose that $\mu\in \Par_e (\rho,r)$ and 
$\lda\in \Par_e (\rho, r+1)$ 
are such that $\mu\subset \lda$. Then 
$\lda\sm \mu$ is the translate of a Young diagram of an $e$-hook partition. 
Moreover, the residue of the top-left corner of $\lda\sm \mu$ is equal to $\kappa$. 
\end{lem}

\begin{proof}
 The first statement is a part of~\cite[Lemma 4(2)]{ChuangKessar2002}. 
 By standard properties of the abacus (cf.~\cite[Section 2.7]{JamesKerber1981}), there exists $(t,i)\in \Ab_N (\mu)$ such that 
 $\Ab_N(\lda)$ is obtained from $\Ab_N (\mu)$ by the move $(t,i)\to (t+1,i)$. 
 By~\cite[Lemma 4(1)]{ChuangKessar2002}, $\{ (t,i), (t,i+1),\ldots (t,e-1) \}\subset \Ab_N (\mu)$ and $\{ (t+1,0),\ldots, (t+1,i-1)\}\cap \Ab_N (\mu) = \varnothing$. 
 Hence, $\Ab_N (\lda)$ may be obtained from $\Ab_N (\mu)$ by the following moves (in the given order), each of which corresponds to adding a single box to a Young diagram:
 \begin{align}
 (t,e-1) &\to (t+1,0), \notag \\ 
 \label{moves} (t,e-2) &\to (t,e-1), \ldots, (t,i) \to (t,i+1), \\
 (t+1,0) &\to (t+1,1), \ldots, (t+1,i-1) \to (t+1,i). \notag
 \end{align}
 Hence, if $\nu$ denotes the partition such that $\nu \supset \mu$ and $\nu \sm \mu$  consists of a single box which is the top-left corner of $\lda\sm \mu$, then 
 $\Ab_N (\nu)$ is obtained from $\Ab_N (\mu)$ by the move $(t,e-1)\to (t+1,0)$. 
 By Lemma~\ref{lem:stepres}, the residue of the only box of $\nu\sm \mu$ 
 is $-N+e\mZ = \k$.  
\end{proof}

\begin{exam}\label{ex:rock2}
As in Example~\ref{ex:rock1}, let $e=3$ and $\rho=(8,6,4,2,2,1,1)$, so that $\k=2+3\mZ$. 
Let $\nu = (8,6,4,4,3,1,1)\in \Par_3 (\rho,1)$, 
$\mu=(11,6,4,4,3,1,1)\in \Par_3 (\rho,2)$ and 
$\lda=(11,6,4,4,3,3,2)\in \Par_3 (\rho,3)$. Then $\rho\subset \nu\subset \mu \subset \lda$, and each of $\nu\sm \rho$, $\mu\sm \nu$, $\lda \sm \mu$ is a translate of the Young diagram of a $3$-hook partition. These translates are shown as hooks with thick boundaries in the following Young diagram of shape $\lda$, which also gives the $3$-residues of all boxes: 
\begin{equation}\label{exdiag}
\begin{array}{c}
\begin{tikzpicture}[scale=0.5,draw/.append style={thin,black}]
 \Tableau{{0,1,2,0,1,2,0,1,2,0,1}, {2,0,1,2,0,1}, {1,2,0,1}, {0,1,2,0}, {2,0,1}, {1,2,0}, {0,1}}
 \draw[line width = 0.9mm] (8.5,.5)--(11.5,.5)--(11.5,-.5)--(8.5,-.5)--cycle;
 \draw[line width = 0.9mm] (2.5,-2.5)--(4.5,-2.5)--(4.5,-3.5)--(3.5,-3.5)--(3.5,-4.5)--(2.5,-4.5)--cycle;
 \draw[line width = 0.9mm] (1.5,-4.5)--(3.5,-4.5)--(3.5,-5.5)--(2.5,-5.5)--(2.5,-6.5)--(1.5,-6.5)--cycle;
\end{tikzpicture}         
\end{array}
\end{equation}
\end{exam}


Let $j\in I$. 
The following lemma is an immediate consequence of the description of the set $\cl E_{d,j}$ given in~\S\ref{subsec:outline}. 

\begin{lem}\label{lem:sh1}
 Let $\bi^{(1)},\ldots, \bi^{(d)}\in I^e$. 
If $\bi = (\bi^{(1)} \ldots \bi^{(d)}) \in \cl E_{d,j}$ and, 
for some $k>0$, $\bi^{(1)},\ldots, \bi^{(k)} \in I^{\d}$, 
then $\bi^{(1)},\ldots,\bi^{(k)} \in I^{\varnothing,1}_{+j}$. 
\end{lem}


\begin{lem}\label{lem:rocksh}
If $\bj \in I^{\rho,0}$ and $\bi\in I^{de}$ are such that 
$\bj\bi\in I^{\rho,d}$, then $\bi\in \cl E_{d,\k}$. 
\end{lem}

\begin{proof}
By the hypothesis, $\bj\bi = \bi^{\mtt t}$ for some 
 $\lda\in \Par_e (\rho,d)$ and some standard tableau $\mtt t$ of shape $\lda$. 
 Since $\rho$ is the $e$-core of $\lda$, there is a sequence 
\[
 \rho =\lda^{0} \subset \lda^{1} \subset \cdots \subset \lda^{d} = \lda
\]
of partitions such that $\lda^{r} \in \Par_e (\rho,r)$ for each $r=0,\ldots,d$. 
By Lemma~\ref{lem:CK}, each $\lda^{r}\sm \lda^{r-1}$ is  a translate of an $e$-hook partition, and the top-left corner of $\lda^{r}\sm \lda^{r-1}$ 
has residue $\k$. 
For $1\le r\le d$, let 
$\mtt t^{-1}(\lda^{r}\sm \lda^{r-1})=\{ a_{r1},\ldots,a_{re} \}$, 
with $a_{r1}<\cdots<a_{re}$. Since $\bj\in I^{\rho,0}$ and 
$\Par_e (\rho,0) = \{ \rho \}$, we have 
$\mtt t^{-1} (\rho) = \{ 1,\ldots, |\rho|\}$.
Hence, as $\mtt t$ is a standard tableau, 
$(i_{a_{r1}-|\rho|},\ldots,i_{a_{re}-|\rho|})\in I^{\vn,1}_{+\k}$. Therefore, 
the partition of $\{1,\ldots,de\}$ into the subsets 
$\{ a_{r1}-|\rho|,\ldots,a_{re}-|\rho| \}$, $r=1,\ldots,d$, witnesses the fact that 
$\bi\in \cl E_{d,\k}$. 
\end{proof}

\begin{lem}\label{lem:jkappa}
 For all $\bj \in I^{\rho,0}$, we have 
$\k \notin \{ j_{\max\{1,|\rho|-e+1\} }, \ldots, j_{|\rho|-1}, j_{|\rho|} \}$.
\end{lem}

The reader may find it helpful to check, by inspecting the residues in~\eqref{exdiag},
that the lemma holds for $e$ and $\rho$ as in Example~\ref{ex:rock2}, i.e.~that 
for any $\bi\in I^{\rho,0}$, none of the last $3$ entries of $\bi$ is equal to $2$.

\begin{proof}
Fix an integer $N\ge \ell(\rho)$ such that $-N+e\mZ =\k$. 
Let $\mtt t$ be a standard tableau of shape $\rho$ 
such that $\bi^{\mtt t} = \bj$. 
Suppose for contradiction that $j_a = \k$ for some $a>|\rho|-e$, and choose such $a$ to be largest possible. Let 
$\mu^{r}$ be the shape of $\mtt t_{\le a-1+r}$ for $r=0,\ldots,|\rho|-a+1$. 
Then, by Lemma~\ref{lem:stepres}, $\Ab_N(\mu^{1})$ is obtained from $\Ab_N(\mu^{0})$ 
by the move $(t,e-1) \to (t+1,0)$ for some $t\ge 0$. 
By maximality of $a$, the abacus $\Ab_N (\rho)$ can be obtained from $\Ab_N (\mu^{1})$ by $|\rho|-a<e$ ``horizontal'' moves, i.e.~moves of the form
$(t',u) \to (t',u+1)$ for $t'\ge 0$, $0\le u<e-1$. 
Recall that, since $\rho$ is a Rouquier core, for each $t'\in \mN$ there exists 
$u\in \{0,\ldots,e-1\}$ such that 
$\Ab_N (\rho) \cap (\{t'\} \times \{0,\ldots,e-1\}) = \{(t',u),(t',u+1),\ldots,(t',e-1)\}$, and the size of this intersection is weakly decreasing as $t'$ increases. 
Let $m$ be the number of beads in row $t$ of $\Ab_N (\rho)$. 
Since $(t,e-1)\notin \Ab_N(\mu^{1})$, at least $m$ horizontal moves in row $t$ are required to transform row $t$ of 
$\Ab_N(\mu^{1})$ to row $t$ of $\Ab_N(\rho)$. Further, row $t+1$ of $\Ab_N(\rho)$ has at most $m$ beads, so the leftmost bead of that row is in
column numbered at least $e-m$. On the other hand, the leftmost bead of row $t$ in $\Ab_N (\mu^{1})$ is in column $0$, so at least $e-m$ horizontal moves are required to transform row $t+1$ of $\Ab_N(\mu^{1})$ into row $t+1$ of 
$\Ab_N (\rho)$. Hence, in total, at least $e$ horizontal 
moves are needed to transform 
$\Ab_N(\mu^{1})$ into $\Ab_N (\rho)$, which is a contradiction. 
\end{proof}

Combining Equation~\eqref{BKf}, Theorem~\ref{thm:BKdim} and  
Lemmas~\ref{lem:sh1} and~\ref{lem:rocksh}, we obtain the following formula:
\begin{equation}\label{BKf2}
 \BK_{|\rho|+ed} (f_{\rho,d}) = 
\sum_{\substack{\bj\in I^{\rho,0} \\ \bi^{(1)},\ldots,\bi^{(d)} \in I^{\vn,1}_{+\k}}} e(\bj \bi^{(1)} \ldots \bi^{(d)}).
\end{equation}

If $\a\in Q_+$ and $\g=(\g_1,\ldots,\g_l)$ is a composition of $\a$, we define
$R^{\Lda_0}_{\g}$ to be the image of $R_{\g}$ under the natural projection
$R_{\a} \thra R_{\a}^{\Lda_0}$.

\begin{prop}\label{prop:bigwall}
If $f=\BK_{|\rho|+ed} (f_{\rho,d})$, then 
$f R^{\Lda_0}_{\cont(\rho)+d\d} f \subset 
e_{\cont(\rho),\d^d} R^{\Lda_0}_{\cont(\rho), d\d} e_{\cont(\rho),\d^d}$.
\end{prop}

\begin{proof}
Let $n= |\rho|+de$, $\mu = (|\rho|,e^d)=(|\rho|,e,\ldots,e)$ 
and $\nu=(|\rho|, de)$. 
By~\eqref{BKf}, we have $f=e_{\cont(\rho), \d^d}$. 
Hence, by Proposition~\ref{prop:double}, 
$f R^{\Lda_0}_{\cont(\rho)+d \d} f = 
\sum_{w\in \ls{\mu}{\scr D}^{\mu}_n} 
R_{\cont(\rho), \d^d}^{\Lda_0} \psi_w R^{\Lda_0}_{\cont(\rho),\d^d}$, 
so it will suffice to prove that $f \psi_w f =0$ for all 
$w\in \ls{\mu}{\scr D}^{\mu}_n \setminus \fr S_\nu$. 
Due to Equation~\eqref{BKf2} and relations~\eqref{rel:eid} and~\eqref{rel:epsi}, 
it is enough to show that for all such $w$ we have 
$w  (\bj \bi^{(1)} \ldots \bi^{(d)}) \ne 
\bj' \bi'^{(1)} \ldots,\bi'^{(d)}$ whenever 
$\bj,\bj'\in I^{\rho,0}$ and $\bi^{(r)}, \bi'^{(r)} \in I^{\vn,1}_{+\k}$ for $r=1,\ldots,d$. Let $Y=\{1,\ldots,|\rho|\}$ and $X_r = \{ |\rho|+(r-1)e+1,\ldots,|\rho|+re\}$ for $r=1,\ldots,d$.  Let $a\in \{1,\ldots,n\}$ be maximal subject to 
$w(a) \in Y$. Since $w\notin \fr S_\nu$, we have  $a>|\rho|$. 
Let $X_r \ni a$ and $b= |\rho| + (r-1)e+1$. Since $w\in \scr D^{\mu}_n$, we have $w(b) <w(a)\le |\rho|$. Since $i^{(r)}_1 = \k$, our assertion is true if $j'_{w(b)}\ne \k$, so we may assume that 
$j'_{w(b)} = \k$. By Lemma~\ref{lem:jkappa}, this implies that 
$w(b) \le |\rho|-e$. For each $c\in Z:=\{w(b),w(b)+1,\ldots,|\rho|\}$, we have 
$w^{-1} (c) \ge b$ because $w\in \ls{\mu}{\scr D}^{\mu}_n$ and 
$w^{-1}(c) \le a$ by maximality of $a$. Since $|Z|>e$ and 
$\{ b,b+1,\ldots,a\} \subset X_r$, this is clearly impossible. 
\end{proof}

If $V$ is a graded vector space, let $\END(V)$ be the algebra of all 
endomorphisms of $V$ (as an ungraded vector space), 
endowed with the unique grading such that 
$\deg(gv) = \deg(g) + \deg(v)$ for all homogeneous elements $g\in \END(V)$ and $v\in V$. 

\begin{prop}\label{prop:wt0}
 Let $\lda$ be an $e$-core. Then there exists a graded vector space $V$ such that
$R_{\cont(\lda)}^{\Lda_0}\cong \END (V)$ as graded algebras. 
\end{prop}

\begin{proof}
 It is well known that $R_{\cont(\lda)}^{\Lda_0} \cong \cl H_{\rho,0}$ is a split simple algebra (e.g.~because it is a cellular algebra with only one cell, see~\cite[Corollary 5.38]{Mathas1999}), so 
$\cl H_{\rho,0} \cong \End(V)$ as an ungraded algebra for some vector space $V$.
By~\cite[Theorem 9.6.8]{NastasescuVanOystaeyen2004}, $V$ can be graded as an 
$R^{\Lda_0}_{\cont(\rho)}$-module, and the result follows. 
\end{proof}

If $B$ is a subset and $A$ is a subalgebra of an algebra $A'$, 
the centraliser of $B$ in $A$ is defined as
$C_A (B) = \{ a\in A\mid ab=ba \; \forall b\in B\}$. 
We will use the following elementary fact. 

\begin{prop}\label{prop:cent}
 Let $A$ be a finite-dimensional graded $F$-algebra. 
Suppose that $B$ is a unital graded subalgebra of $A$ such that
$B \cong \END (V)$ for some graded vector space $V$. Let $C= C_A (B)$. Then 
there is a graded algebra isomorphism  $B\ot C \isoto A$ given by
$b\ot c \mapsto bc$ for $b\in B$, $c\in C$. Moreover, 
for any homogeneous primitive idempotent $\veps$ of $B$, we have a graded algebra isomorphism
$C\isoto \veps A \veps$ given by $c\mapsto \veps c$. 
\end{prop}

\begin{proof}
We view $V$ as a $B$-module via the given isomorphism. 
 Let $\{v_1,\ldots,v_m\}$ be a homogeneous basis of $V$.
For $1\le i,j\le m$, let $e_{ij}\in B$ be the element given by 
$e_{ij} v_k = \d_{jk} v_i$ for $k=1,\ldots,m$. Then $\{e_{ij} \mid 1\le i,j\le m\}$ is a homogeneous basis of $B$, and 
$\{ e_{ii} \mid 1\le i\le m\}$ is a full set of primitive idempotents in $B$; in particular, $\sum_{i=1}^m e_{ii}=1$.
Let $C'=e_{11} A e_{11}$.  
It is straightforward to check that, for any $x\in C'$, the element $\xi(x):=\sum_{i=1}^m e_{i1} xe_{1i}\in A$ commutes with $e_{jk}$ for $1\le j,k\le m$, so $\xi(x)\in C$.
It follows easily that 
the maps
$\xi\colon C'\to C$ and $C\to C'$, $y \mapsto e_{11} y$, are mutually inverse isomorphisms of graded algebras. 
For any $i$ and $j$, the graded vector space $e_{ii} A e_{jj}$ is isomorphic to $C'$, 
as the maps $C'\to e_{ii} A e_{jj}$, $x \mapsto e_{i1} x e_{1j}$ and 
$e_{ii}Ae_{jj} \to C'$, $y\mapsto e_{1i} y e_{j1}$ are mutual inverses. 
Observe also that for all $x\in C'$ we have $e_{ij} \xi(x) = e_{i1} x e_{1j}$ whenever $1\le i,j\le m$. 
It follows that $e_{ij} C= e_{ii} A e_{jj}$ for all $i,j$. 
Therefore, the graded algebra homomorphism defined in the statement of the proposition is an isomorphism $B\ot C \isoto A$. 
The last statement has already been proved for $\veps=e_{11}$ and follows in the general case because $\veps$ and $e_{11}$ 
are conjugate by an invertible element of $B_{\{0\}}$ 
(both being primitive idempotents of $B_{\{0\}}$). 
\end{proof}

If $\a,\b \in Q_+$ and $\hght(\b)=m\le n=\hght(\a)$, define a graded algebra 
homomorphism
$\iota_{\b}^{\a} \colon R^{\Lda_0}_{\b} \to R^{\Lda_0}_{\a}$ by 
$x\mapsto e_{\a} \iota_m^n (x)$ (cf.~\eqref{iotamn}). 
As before, let 
\begin{equation}\label{eq:fdef}
f=\BK_{|\rho|+de} (f_{\rho,d}) = e_{\cont(\rho),\d^d}\in R^{\Lda_0}_{\cont(\rho)+d\d}.
\end{equation}
Observe that $f$ centralises $\iota_{\cont(\rho)}^{\cont(\rho)+d\d} (R^{\Lda_0}_{\cont(\rho)})$ and that $f\ne 0$: the latter fact follows easily 
from~\eqref{BKf} and Theorem~\ref{thm:BKdim}. 
Hence, by Proposition~\ref{prop:wt0}, 
the map $R^{\Lda_0}_{\cont(\rho)} \to fR^{\Lda_0}_{\cont(\rho)+d\d}f$ given by 
$x\mapsto \iota_{\cont(\rho)}^{\cont(\rho)+d\d} (x)f$ is an injective unital graded algebra homomorphism, and its image is isomorphic to $\END(V)$ for some graded vector space $V$. Therefore, 
defining 
\begin{equation}\label{Crhod}
 C_{\rho,d} = C_{fR^{\Lda_0}_{\cont(\rho)+d\d}f} (\iota_{\cont(\rho)}^{\cont(\rho)+d\d} (R_{\cont(\rho)}^{\Lda_0})),
\end{equation}
we have a graded algebra isomorphism 
\begin{equation}\label{eq:tensoriso}
R_{\cont(\rho)}^{\Lda_0} \ot C_{\rho,d} \isoto f R^{\Lda_0}_{\cont(\rho)+d\d} f
\end{equation}
given by $a \ot c \mapsto \iota_{\cont(\rho)}^{\cont(\rho)+d\d}(a) c$ for $a\in R_{\cont(\rho)}^{\Lda_0}$ and $c\in C_{\rho,d}$, due to Proposition~\ref{prop:cent}.
Hence, in order to prove Theorem~\ref{thm:main2}, it suffices to construct a graded isomorphism from $R^{\Lda_0}_{\d}\wr \fr S_d$ onto $C_{\rho,d}$. 
Most of the remainder of the paper is devoted to this task. 

\begin{prop}\label{prop:im_om}
 Let $\om\colon R_{d\d} \to R^{\Lda_0}_{\cont(\rho),d\d}$ be the graded algebra homomorphism defined as 
the composition $R_{d\d} \to R_{\cont(\rho), d\d} \thra R^{\Lda_0}_{\cont(\rho), d\d}$ where the second map is the natural projection and the first one is given by
$x \mapsto \iota_{\cont(\rho),d\d} (e_{\cont(\rho)} \ot x)$. Then:
\begin{enumerate}[(i)]
 \item\label{im_om1} We have $C_{\rho,d} = \om(e_{\d^d} R_{d\d} e_{\d^d})$.
 \item\label{im_om2} For any $\bi\in I^{d\d} \sm \cl E_{d,\k}$, we have 
$\om(e(\bi))=0$. 
\end{enumerate}
\end{prop}

\begin{proof}
It is clear from the definition that 
$\om(e_{\d^d} R_{d\d} e_{\d^d}) \subset 
e_{\cont(\rho),\d^d} R_{\cont(\rho)+d\d}^{\Lda_0} e_{\cont(\rho),\d^d} = f R_{\cont(\rho)+d\d} f$
and that $\om(R_{d\d})$ commutes with 
$\iota_{\cont(\rho)}^{\cont(\rho)+d\d} (R_{\cont(\rho)}^{\Lda_0})$. 
Thus, $\om(e_{\d^d} R_{d\d} e_{\d^d}) \subset C_{\rho,d}$. 
For the converse, let $x\in C_{\rho,d}$. 
Then it follows from Proposition~\ref{prop:bigwall}
that $x=\sum_{j=1}^m a_j c_j$ for some  
$a_1,\ldots,a_m\in \iota_{\cont(\rho)}^{\cont(\rho)+d\d} (R_{\cont(\rho)}^{\Lda_0})$ and 
$c_1,\ldots,c_m\in \om(e_{\d^d} R_{d\d} e_{\d^d})\subset C_{\rho,d}$, 
where $a_1,\ldots,a_m$ may be assumed to form a basis of 
$\iota_{\cont(\rho)}^{\cont(\rho)+d\d} (R^{\Lda_0}_{\cont(\rho)})$, with 
$a_1=f$. 
Due to injectivity of the map in Proposition~\ref{prop:cent}, we infer that 
$x=c_1$, so $x\in \om(e_{\d^d} R_{d\d} e_{\d^d})$, 
and~\eqref{im_om1} is proved. 

For~\eqref{im_om2}, note that for all $\bj\in I^{\rho,0}$ we have  
$\bj\bi\notin I^{\rho,d}$ by Lemma~\ref{lem:rocksh} and hence 
$e(\bj\bi) = 0$ by Theorem~\ref{thm:BKdim}. 
Thus, $\om(e(\bi))=\sum_{\bj\in I^{\rho,0}} e(\bj \bi) = 0$. 
\end{proof}

Recall the definition of the quotient $\wh{R}_{d\d}$ of $R_{d\d}$ in~\S\ref{subsec:outline}. By Proposition~\ref{prop:im_om}, the map $\om$ defined in the statement of the proposition induces a homomorphism 
$\wh{R}_{d\d} \to R^{\Lda_0}_{\cont(\rho),d\d}$, which restricts to a 
surjective graded algebra homomorphism 
$\Om\colon e_{\d^d} \wh{R}_{d\d} e_{\d^d} \thra C_{\rho,d}$. 

\subsection{The graded dimension of $C_{\rho,d}$}
In this subsection, we prove the following result:

\begin{prop}\label{prop:qdim}
 We have $\qdim(C_{\rho,d}) = \qdim(R^{\Lda_0}_{\d}\wr \fr S_d) = 
d! \qdim(R^{\Lda_0}_{\d})^d$. In particular, $C_{\rho,d}$ is nonnegatively graded. 
\end{prop}

Turner (\cite[Proposition 81]{Turner2009}) proved the same result for ungraded dimensions in the case when $\xi=1$. The proof given below is similar. 
If $\mu$ is an $e$-core and $a\ge 0$, let $\Std_e(\mu,a)$ be the set of all standard tableaux with shape belonging to $\Par_e(\mu,a)$.
Let 
\[
\Std'_e (\rho,d) = \{ \mtt t\in \Std_e(\rho,d) \mid \mtt t_{\le |\rho|+rd} \in \Std_e(\rho,r) \text{ for all } r=0,\ldots,d \}. 
\]
Define the map 
\begin{align*}
 \b \colon \Std'_e(\rho,d) \to \Std_e(\rho) \times \Std_e(\vn,1)^{\times d}
\end{align*}
by $\b(\mtt t) = (\mtt t_{\le |\rho|}, \mtt s_1,\ldots,\mtt s_d)$ where, 
for $r=1,\ldots,d$, the tableau
$\mtt s_r$ is the unique standard tableau which is a translate of the skew tableau of size $e$ given by 
$m\mapsto \mtt t(|\rho|+e(r-1)+m)$, $m=1,\ldots,e$; such a translate exists and has a shape belonging to $\Par_e(\vn,1)$ by 
Lemma~\ref{lem:CK}. 

For any $i\in \{0,\ldots,e-1\}$, 
let $v_i$ be the number of beads in the 
$i$-th column of $\Ab_N(\rho)$. 
For any $\lda \in \Par_e(\rho,d)$, let $\lda^{(i)}$ be the partition such that 
$\Ab_{v_k}^{1} (\lda^{(i)})$ is the projection onto the first component 
of $\Ab_N (\lda) \cap (\mN \times \{i\})$.
Up to a permutation, the sequence 
$(\lda^{(0)},\ldots,\lda^{(e-1)})$ is known as the $e$-quotient of $\lda$. 

\begin{lem}\label{lem:preimage}
 Let $\lda \in \Par_e(\rho,d)$ and 
 $(\mtt u,\mtt s_1,\ldots,\mtt s_d)\in \Std(\rho) \times \Std_e(\vn,1)^{\times d}$. 
 For each $i=0,\ldots,e-1$, let 
$d_i = \# \{ r\in \{1,\ldots,d\} \mid \Shape(\mtt s_r) = (i+1,1^{e-i-1}) \}$. 
Then 
\[
|\b^{-1} (\mtt u,\mtt s_1,\ldots,\mtt s_d) \cap \Std(\lda)|
= 
\begin{cases}
\prod_{i=0}^{e-1} |\Std(\lda^{(i)})| & \text{if } |\lda^{(i)}|=d_i 
\text{ for } i=0,\ldots,e-1, \\
0 & \text{otherwise.}
\end{cases}
\]
\end{lem}

\begin{proof}
The map 
\begin{equation}\label{eq:map}
\mtt t\mapsto 
(\Shape(\mtt t_{\le |\rho|}), \Shape(\mtt t_{\le |\rho|+e}), \ldots, 
\Shape(\mtt t_{\le |\rho|+ed}))
\end{equation}
is clearly a bijection from 
$\b^{-1} (\mtt u,\mtt s_1,\ldots,\mtt s_d) \cap \Std(\lda)$ onto the set of 
sequences $\rho=\mu^{0} \subset \mu^{1} \subset \cdots \subset \mu^d =\lda$ 
of partitions such that $\mu^{r}\sm \mu^{r-1}$ is a translate of 
$\Shape(\mtt s_r)$ for each $r=1,\ldots,d$. If $1\le r\le d$, $0\le i\le e-1$, $\nu \in \Par_e(\rho,r-1)$ and $\mu\in {\Par_e(\rho,r)}$ are such that $\nu \subset \mu$, 
then by Lemma~\ref{lem:CK} 
and~\cite[Lemma 4(2)]{ChuangKessar2002},
$\mu\sm \nu$ is a translate of $(i+1,1^{e-i-1})$ if and only if 
$\Ab_N (\mu)$ is obtained from $\Ab_N (\nu)$ by the move $(t,i) \to (t+1,i)$ for some $t\ge 0$. 
It follows immediately that 
$\b^{-1} (\mtt u,\mtt s_1,\ldots,\mtt s_d) \cap \Std(\lda)=
\varnothing$
unless $|\lda^{(i)}|=d_i$ for all $i=0,\ldots,e-1$. 
Assuming that $|\lda^{(i)}|=d_i$ for all $i$, for any given 
sequence $\rho = \mu^0 \subset \mu^1 \subset \cdots \subset \mu^d =\lda$ as above
 and any $i\in \{0,\ldots,e-1\}$, let $\mtt t_i\in \Std(\lda^{(i)})$ be defined as follows. Let $\{m_1<\cdots<m_{d_i}\}$ be the set of elements 
$m\in \{1,\ldots,d\}$ such that $\Shape(\mtt s_m)=(i+1,1^{e-i-1})$: then 
$\Shape((\mtt t_i)_{\le k})=(\mu^{m_k})^{(i)}$ for all 
$k=1,\ldots,d_i$.
This assignment of a tuple $(\mtt t_0,\ldots,\mtt t_{e-1})$ to each sequence
$\rho=\mu^0 \subset \cdots \subset \mu^d =\lda$ with the above properties defines a bijection from the set of such sequences onto
$\Std(\lda^{(0)}) \times \cdots \times \Std(\lda^{(e-1)})$ and therefore, in view of~\eqref{eq:map}, completes the proof.
\end{proof}

\begin{lem}\label{lem:betadeg}
 For any $\mtt t\in \Std'_e (\rho,d)$, if $\beta(\mtt t) = (\mtt u,\mtt s_1,\ldots,\mtt s_d)$, then $\deg(\mtt t) = \deg(\mtt u) + \sum_{r=1}^d \deg (\mtt s_r)$.
\end{lem}

\begin{proof}
Due to~\eqref{eq:deg}, it is enough to prove that for any $r=1,\ldots,d$, 
\begin{equation}\label{betadeg1}
\sum_{k=1}^e d_{\mtt t(|\rho|+e(r-1)+k)} (\Shape(\mtt t_{\le |\rho|+e(r-1)+k})) =\deg(\mtt s_r). 
\end{equation}

We will use the following general fact, which follows easily from~\eqref{demu}.
Let $\lda$ be a partition with $\ell(\lda)\le N$. For $t\ge 0$ and $0\le i<e$, write 
\[
c_{<t,i} (\lda)=|\Ab_N (\lda) \cap \{ (0,i),\ldots, (t-1,i) \}|.
\]
Let $a=(k,\lda_k)$ be a removable node of $\lda$, and let $(t,i)$ be the bead of $\Ab_N (\lda)$ corresponding to this node, in the sense that 
$\lda_k+N-k = et+i$. Then we have 
\begin{equation}\label{da_abacus}
 d_a (\lda) = 
\begin{cases}
  c_{<t,i-1}(\lda) -  c_{<t,i}(\lda) & \text{if } i>0, \\
  c_{<t,e-1}(\lda) - c_{<t+1,0}(\lda) & \text{if } i=0. 
\end{cases}
\end{equation}

Let $\mu = \Shape(\mtt t_{\le |\rho|+e(r-1)})$, $\nu = \Shape(\mtt t_{\le |\rho|+er})$, and let 
$(t,i)\to (t+1,i)$ be the move converting 
$\Ab_N(\mu)$ to $\Ab_N(\nu)$. As in the proof of Lemma~\ref{lem:CK}, we have 
$\{ (t,i), (t,i+1),\ldots (t,e-1) \}\subset \Ab_N (\mu)$ and 
$\{ (t+1,0),\ldots, (t+1,i-1)\}\cap \Ab_N (\mu) = \varnothing$. 
There exists an ordering $M_1,\ldots,M_e$ of the $e$ moves listed under~\eqref{moves} such that, for each $k=1,\ldots,d$, 
$\Ab_N(\Shape(\mtt t(|\rho|+e(r-1)+k)))$ is obtained from
$\Ab_N(\Shape(\mtt t(|\rho|+e(r-1)+k-1)))$ by the move $M_k$. 
Let $\Asf=\Ab_N(\mu)$ and consider any abacus $\Asf'$ obtained from $\Asf$ by arbitrary addition or deletion of beads in any positions not belonging to the set  $Z:=\{ (t,i),\dots,(t,e-1), (t+1,0),\dots, (t+1,i)\}$. Let $\mu'$ be the partition defined by the condition that $\Ab_{N'} (\mu') = \Asf'$, where $N'$ is the number of beads in $\Asf'$, and set
\begin{equation}\label{eq:Asfd}
  \mathsf{d} (\Asf') = 
  \sum_{k=1}^e d_{a_k} (\mu' \cup \{a_1,\dots,a_k\}),
\end{equation}
where $a_1,\dots,a_e$ are the nodes added to $\mu'$ by the moves $M_1,\dots,M_e$ in this order. In particular, $\mathsf{d}(\Asf)$ is the left-hand side of~\eqref{betadeg1}. 
We claim that $\mathsf d(\Asf') = \mathsf d(\Asf)$. To prove this, it suffices to show that, for any $\Asf'$ as above,
$\mathsf{d}(\Asf')$ does not change when one 
alters $\Asf'$ by adding or deleting a bead in a position $(t',j) \notin Z$.
 If $t'>t+1$ or $t'=t+1$ and $j>i$, then 
a bead in position $(t',j)$ does not affect the calculation of 
$\mathsf{d} (\Asf')$ via 
the formula~\eqref{da_abacus}. On the other hand, if 
$t'<t$ or $t'=t$ and $j<i$, then the total contribution of any bead in position $(t',j)$ to the calculation of $\mathsf{d} (\Asf')$ via~\eqref{da_abacus} is $0$
(because such a bead contributes $1$ to one of the summands of~\eqref{eq:Asfd}, $-1$ to another summand and $0$ to the remaining summands).  
This proves the claim. 

Now consider the abacus $\Asf'$ obtained from $\Asf$ by deleting all beads outside positions $(t,i),\dots, (t,e-1), (t+1,0),\dots, (t+1,i)$ and then adding a bead in each of the positions $(t,0),\dots, (t,i-1)$. Then 
$\mathsf{d}(\Asf')= \mathsf{d} (\Asf)$.
On the other hand, for each $k=1,\dots,e$, the abacus obtained from 
$\Asf'$ by the moves $M_1,\dots,M_k$ is precisely the abacus 
$\Ab_{e}((\mtt s_r)_{\le k})$ with $t$ empty rows added on the top. It follows by~\eqref{da_abacus} that $\mathsf{d}(\Asf')= \deg(\mtt s_r)$, and we have proved~\eqref{betadeg1}.
\end{proof}

\begin{proof}[Proof of Proposition~\ref{prop:qdim}]
The second equality in the statement is obvious, so we only need to prove that
$\qdim (C_{\rho,d}) = d! \qdim(R_{\d}^{\Lda_0})^d$. 
For any $e$-core $\mu$ and $a\ge 0$, let 
\[
X_{\mu,a} = \{ (\mtt t,\mtt t') \in \Std_e(\mu,a) \mid \mtt t \text{ and } \mtt t' \text{ have the same shape } \}.
\]
 Using~\eqref{eq:fdef}, Theorem~\ref{thm:BKdim} and Lemmas~\ref{lem:preimage} and~\ref{lem:betadeg},
we compute
\begin{align*}
 & \qdim(f R^{\Lda_0}_{\cont(\rho)+d\delta} f) = 
\sum_{(\mtt t,\mtt t') \in X_{\rho,d} \cap (\Std'_e(\rho,d)\times \Std'_e(\rho,d))} q^{\deg(\mtt t)+\deg (\mtt t')} \\
&=
\sum_{\mtt u,\mtt u'\in \Std(\rho)} q^{\deg(\mtt u) + \deg(\mtt u')}  
 \sum_{\substack{d_0,\ldots,d_{e-1}\ge 0 \\ d_0+\cdots+d_{e-1} =d}}
\binom{d}{d_0,\ldots,d_{e-1}}^2 \\
& \times 
\prod_{i=0}^{e-1} \left( \sum_{\mtt s_{i,1},\ldots,\mtt s_{i,d_i}, 
\mtt s'_{i,1}, \ldots, \mtt s'_{i,d_i}\in \Std_e(i+1,1^{e-i-1})} 
q^{\sum_{l=1}^{d_i} (\deg(\mtt s_{i,l}) + \deg(\mtt s'_{i,l}))} 
\sum_{\lda^i \in \Par(d_i)} |\Std(\lda^i)|^2 \right) \\
&= d! \sum_{\mtt u,\mtt u'\in \Std(\rho)} q^{\deg(\mtt u) + \deg(\mtt u')}
\sum_{\substack{d_0,\ldots,d_{e-1}\ge 0 \\ d_0+\cdots+d_{e-1} =d}}
\binom{d}{d_0,\ldots,d_{e-1}}  \\
&\times \prod_{i=0}^{e-1} \left( \sum_{\mtt s_{i,1},\ldots,\mtt s_{i,d_i}, 
\mtt s'_{i,1}, \ldots, \mtt s'_{i,d_i}\in \Std_e(i+1,1^{e-i-1})} 
q^{\sum_{l=1}^{d_i} (\deg(\mtt s_{i,l}) + \deg(\mtt s'_{i,l}))} 
 \right) \\
&= d!\sum_{\mtt u,\mtt u'\in \Std(\rho)} q^{\deg(\mtt u) + \deg(\mtt u')} 
 \sum_{(\mtt s_1,\mtt s'_1),\ldots, (\mtt s_d,\mtt s'_d) \in X_{\vn,1}} 
q^{\sum_{r=1}^d (\deg(\mtt s_r) + \deg(\mtt s'_r))} \\
&= d!  \qdim(R^{\Lda_0}_{\cont(\rho)}) \qdim(R_{\d}^{\Lda_0})^d,
\end{align*}
where for the third equality we use the classical identity 
$\sum_{\mu \in \Par(m)} |\Std(\mu)|^2 = m!$, which holds for all $m\ge 0$.
The desired identity is now obtained by dividing both sides by 
$\qdim(R^{\Lda_0}_{\cont(\rho)})$
and using the graded isomorphism~\eqref{eq:tensoriso}.  
\end{proof}

\section{A homogeneous basis of $R^{\Lda_0}_{\d}$}\label{sec:delta}

Recall that $\d = \a_0+\cdots+\a_{e-1}\in Q_+$ is the fundamental imaginary root. 
In $R_{\d}$, we have $\psi_r \psi_{r+1} \psi_r e_{\d} = \psi_{r+1} \psi_r \psi_{r+1} e_{\d}$ for all $r=1,\ldots,e-2$ because for each $\bi\in I^{\d}$ we have $i_t\ne i_{t'}$ for $1\le t<t'\le e$. Hence, by Matsumoto's Theorem 
(see e.g.~\cite[Theorem 1.8]{Mathas1999}), for any $v\in \fr S_e$, the element
$\psi_v e_{\d}$ does not depend on the choice of a reduced expression for $v$. 
For each $\bi, \bj\in I^{\d}$, let $w_{\bi,\bj}$ be the unique element of $\fr S_e$ such that $w_{\bi,\bj} \bj = \bi$.

The following three lemmas are left as  exercises for the reader. 
In the first two, we identify every $i\in I$ with the corresponding element of 
$\{ 0,1,\ldots,e-1\} \subset \mZ$.
If $\bi\in I^n$, we say that a tuple
$\bj$ is a \emph{subsequence} of $\bi$ if $\bj = (i_{a_1},\ldots,i_{a_m})$
for some $a_1,\ldots,a_m\in \{1,\ldots,n\}$ such that $a_1<\ldots<a_m$.

\begin{lem}\label{lem:hkdesc}
An element $\bi\in I^{\d}$ belongs to $I^{\vn,1}$ if and only if 
$i_1=0$ and both
$(1,2,\ldots,i_e-1)$ and 
$(e-1,e-2,\ldots,i_e+1)$ 
are subsequences of $\bi$. 
\end{lem}

\begin{lem}\label{lem:hooktableaux}
 Let $\bi\in I^{\vn,1}$. 
Then the set $\{\mtt t\in \Std(e) \mid \bi^{\mtt t} = \bi\}$ 
consists of precisely two tableaux, namely, 
a standard tableau of shape $(i_e+1,1^{e-i_e-1})$ and degree $1$ 
and a standard tableau of shape $(i_e,1^{e-i_e})$ and degree $0$. 
\end{lem}

\begin{lem}\label{lem:lee}
 If $0\le m<e$, then all standard tableaux of size $m$ have degree $0$. 
\end{lem}

\begin{prop}\label{prop:d1basis}
  The algebra $R^{\Lda_0}_{\d}$ has a homogeneous basis 
  $\cl B_0 \sqcup \cl B_1 \sqcup \cl B_2$ where 
\begin{align*}
 \cl B_0 &= \{ \psi_{w_{\bj,\bi}} e(\bi) \mid \bi,\bj \in I^{\vn,1}, i_e=j_e \}, \\
 \cl B_1 &= \{ \psi_{w_{\bj,\bi}} e(\bi) \mid \bi, \bj \in I^{\vn,1}, i_{e} = j_e\pm 1 \}, \\
 \cl B_2 &= \{ \psi_{w_{\bj,\bi}} y_e e(\bi) \mid \bi, \bj \in I^{\vn,1}, i_e=j_e \}. 
\end{align*}
and $\deg (x) = m$ for all $x\in \cl B_m$, $m=0,1,2$. 
\end{prop}

If $e=2$, then the homogeneous basis given by the proposition is
simply $\{ e(01), y_2 e(01) \}$. 

\begin{proof}
Note that by Lemma~\ref{lem:lee} and Theorem~\ref{thm:BKdim} we have $R^{\Lda_0}_{e-1}=(R^{\Lda_0}_{e-1})_{\{0\}}$, so $y_1=\cdots=y_{e-1}=0$ in $R^{\Lda_0}_{e-1}$. Since the graded algebra homomorphism $\iota_{n-1}^n$ (see \S\ref{subsec:klr1}) is well defined, we have 
$y_1=\cdots=y_{e-1}=0$ in $R_e^{\Lda_0}$ as well. 
By Lemma~\ref{lem:hooktableaux} and Theorem~\ref{thm:BKdim}, 
the following statements are true for any 
$\bi,\bj\in I^{\vn,1}$:
\indent
\begin{enumerate}
 \item[1.] If $i_e = j_e$, then $\qdim(e(\bj) R^{\Lda_0}_{\d} e(\bi))=1+q^2$.
 \item[2.] If $i_e=j_e\pm 1$, then 
$\qdim(e(\bj) R^{\Lda_0}_{\d} e(\bi))=q$.
 \item[3.] If $j_e \notin \{ i_e,i_e-1,i_e+1\}$, then $e(\bj) R^{\Lda_0}_{\d} e(\bi)=0$. 
\end{enumerate}
In particular, $R^{\Lda_0}_{\d} = (R^{\Lda_0}_{\d})_{ \{0,1,2\} }$. 
By Theorem~\ref{thm:basis}\eqref{basis1}, it follows that, whenever 
$\bi, \bj\in I^{\vn,1}$ and $i_e=j_e$, the vector space 
$e(\bj) R^{\Lda_0}_{\d} e(\bi)$ is spanned by 
$\{ \psi_{w_{\bj,\bi}} e(\bi), \psi_{w_{\bj,\bi}} y_e e(\bi)\}$ and hence, by comparing graded dimensions, that these elements form a basis of 
$e(\bj) R^{\Lda_0}_{\d} e(\bi)$ and 
 have degrees $0$ and $2$ respectively (the latter fact is also easily seen directly from definitions). 
Similarly, if $j_e = i_e\pm 1$ (and hence $e>2$), 
then the singleton set 
$\{ \psi_{w_{\bj,\bi}} e(\bi) \}$ spans 
the vector space $e(\bj) R^{\Lda_0}_e e(\bi)$ and hence forms a basis of this space; moreover, the unique element of this set has degree $1$.
The proposition follows from the above assertions. 
\end{proof}

 

\begin{lem}\label{lem:d1gen}
Assume that $e>2$. Then 
\begin{enumerate}[(i)]
\item\label{d1gen1}
If $\bi,\bj \in I^{\vn,1}$ and $j_e\in \{ i_e-1,i_e+1\}$, then 
\[
\psi_{w_{\bi,\bj}} \psi_{w_{\bj,\bi}} e(\bi) = 
\begin{cases}
-y_e e(\bi) & \text{if } j_e=i_e-1, \\
\phantom{-}y_e e(\bi) & \text{if } j_e=i_e+1.
\end{cases}
\]
\item\label{d1gen2}
The algebra $R_{\d}^{\Lambda_0}$ is generated by $(R_{\d}^{\Lda_0})_{\{0,1\}}$. 
\end{enumerate}
\end{lem}

\begin{proof}
\eqref{d1gen1} It is easy to see 
that there exists $\bk\in I^{\vn,1}$ such that 
$k_e=i_e$ and $k_{e-1} = j_e$. Then $\bk':=s_{e-1}  \bk$ also belongs to $I^{\vn,1}$,
and $k'_e = j_e$. 
We have $w_{\bi,\bj} = w_{\bi,\bk} s_{e-1} w_{\bk',\bj}$ and, since $s_{e-1} \in \ls{(e-1,1)}{\scr D}^{(e-1,1)}_e$ and 
$w_{\bi,\bk}, w_{\bk',\bj} \in \fr S_{(e-1,1)}$, it follows that 
$\psi_{w_{\bi,\bj}} e_{\d} = \psi_{w_{\bi,\bk}} \psi_{e-1} 
\psi_{w_{\bk',\bj}} e_{\d}$.
Similarly, 
$\psi_{w_{\bj,\bi}}e_{\d} = \psi_{w_{\bj, \bk'}} \psi_{e-1} 
\psi_{w_{\bk,\bi}} e_{\d}$.
 Since $i_e=k_e$, we have $\deg(\psi_{w_{\bk,\bi}} e(\bi))=0$, and hence, applying repeatedly the case $i_r \noarrow i_{r+1}$ of the relation~\eqref{rel:psisq} and using the fact that $w_{\bi,\bk} = w_{\bk,\bi}^{-1}$, 
 we obtain $\psi_{w_{\bi,\bk}} \psi_{w_{\bk,\bi}} e(\bi)= e(\bi)$. 
 Similarly, $\psi_{w_{\bj,\bk'}} \psi_{w_{\bk',\bj}} e(\bk') = e(\bk')$. 
Set $\varepsilon =1$ if $j_e=i_e+1$ and $\varepsilon=-1$ if $j_e=i_e-1$. 
 Using the above equalities together with the fact that $y_{e-1} =0$ (see the proof of Proposition~\ref{prop:d1basis}) and that $\psi_{w_{\bk,\bi}}$ commutes with 
 $y_e$ (as $w_{\bk,\bi} (e)=e$), we obtain 
 \begin{align*}
 \psi_{w_{\bi,\bj}} \psi_{w_{\bj,\bi}} e(\bi) &= 
 \psi_{w_{\bi,\bk}} \psi_{e-1} \psi_{w_{\bk',\bj}}
 \psi_{w_{\bj, \bk'}} \psi_{e-1} \psi_{w_{\bk,\bi}} e(\bi)\\
 &= 
 \psi_{w_{\bi,\bk}} \psi_{e-1} \psi_{w_{\bk',\bj}}
 \psi_{w_{\bj,  \bk'}} e(\bk') \psi_{e-1} \psi_{w_{\bk,\bi}} e(\bi)
  \\
  &=\psi_{w_{\bi,\bk}} \psi_{e-1}^2 e(\bk')  \psi_{w_{\bk,\bi}} e(\bi) \\
  &=  \varepsilon 
  \psi_{w_{\bi,\bk}} (y_e-y_{e-1}) \psi_{w_{\bk,\bi}} e(\bi) \\
  &=  \varepsilon \psi_{w_{\bi,\bk}} y_e \psi_{w_{\bk,\bi}} e(\bi)
  = \varepsilon \psi_{w_{\bi,\bk}} \psi_{w_{\bk,\bi}} y_e e(\bi)
  = \varepsilon y_e e(\bi).
 \end{align*}

\eqref{d1gen2} 
Since $e\ge 3$, for each $\bi\in I^{\vn,1}$, there exists $\bj\in I^{\vn,1}$ such that $j_e \in \{i_e-1,i_e+1\}$. The result now follows from~\eqref{d1gen1} and Proposition~\ref{prop:d1basis}. 
\end{proof}

\section{Wreath product relations in a quotient of a KLR algebra}\label{sec:Theta}

In this section we construct a unital graded algebra homomorphism 
$\Theta\colon R^{\Lda_0}_{\d} \wr \fr S_d \to e_{\d^d} \ol{R}_{d\d}e_{\d^d}$ (cf.~\eqref{comp}).
As is mentioned in~\S\ref{subsec:outline}, 
we adapt ideas from~\cite[Section 1]{KangKashiwaraKim2013}
 in order to define the images $\tau_r$ of elementary transpositions $s_r \in \fr S_d$ under $\Theta$. The elements $\tau_r$ are defined by Equation~\eqref{tau_r}, and their needed properties are summarised in Theorem~\ref{thm:Theta}.
The present set-up is quite different from that of~\cite{KangKashiwaraKim2013}: in particular, the ``error terms'' $\e_r$ appearing in~\eqref{tau_r} have no analogue in~\cite{KangKashiwaraKim2013}. Consequently, the presentation below is largely self-contained.

\subsection{The intertwiners $\vphi_w$}\label{subsec:vphi}

Fix $n\ge 0$.
We recall necessary facts from~\cite[\S 1.3.1]{KangKashiwaraKim2013}. For $1\le r<n$, define $\vphi_r \in R_{n}$ by
\begin{equation}\label{defvphi}
 \vphi_r e(\bi) =\begin{cases}
                    (\psi_r y_r - y_r \psi_r) e(\bi) = (\psi_r (y_r-y_{r+1}) +1)e(\bi) & \text{if } i_r = i_{r+1}, \\
                     \psi_r e(\bi) & \text{if } i_r \ne i_{r+1}
                   \end{cases}
\end{equation}
for all $\bi \in I^n$.

If $w= s_{r_1} \cdots s_{r_m}$ is a reduced expression in $\fr S_n$, define $\vphi_w: = \vphi_{r_1} \cdots \vphi_{r_m}$. It follows from part~\eqref{KaKaKi2} of the following lemma and Matsumoto's theorem that $\vphi_w$ depends only on $w$, not on the choice of the reduced expression. In particular, we note that 
\begin{equation}\label{vphiprod}
 \vphi_v \vphi_w = \vphi_{vw}
\end{equation}
whenever $v,w\in \fr S_n$ and $\ell(vw) = \ell(v)+\ell(w)$. 
Also, we will repeatedly use the fact that
$\vphi_w e(\bi) = e(w \bi) \vphi_w$ for all $\bi\in I^n$, $w\in \fr S_n$.

\begin{lem}\label{lem:KaKaKi} \cite[Lemma 1.3.1]{KangKashiwaraKim2013}
For $1\le r<n$, $1\le t\le n$, $w\in \fr S_n$ and $\bi\in I^n$, 
 \begin{enumerate}[(i)] 
  \item $\vphi_r^2 e(\bi) = (L_{i_r, i_{r+1}} (y_r, y_{r+1}) + \d_{i_r,i_{r+1}} ) e(\bi)$;
 \item\label{KaKaKi2} $\vphi_r \vphi_{r+1} \vphi_r = \vphi_{r+1} \vphi_r \vphi_{r+1}$ if $r<n-1$;
 \item\label{KaKaKi3} $\vphi_w y_t = y_{w(t)} \vphi_w$; 
 \item\label{KaKaKi4} if $1\le k<n$ and $w(k+1) = w(k)+1$, then $\vphi_w \psi_k = \psi_{w(k)} \vphi_w$; 
 \item\label{KaKaKi5} $\vphi_{w^{-1}} \vphi_w e(\bi) = \prod_{\substack{1\le a<b\le n\\ w(a)>w(b)}} (L_{i_a, i_b} (y_a, y_b) + \d_{i_a,i_b}) e(\bi)$. 
 \end{enumerate}
\end{lem}

Suppose now that $n=2e$. 
Recalling the definition before Corollary~\ref{cor:tildeg},
consider the subalgebra $R'_{\d^2}$ of $R_{2\d}$.

\begin{lem}\label{lem:diagram}
Let $w=(1, e+1) (2,e+2) \cdots (e,2e) \in \fr S_{2e}$,
 and let $K$ be the ideal of $F[y_1,\ldots,y_{2e}]$ generated by the set 
$\{ y_r-y_t \mid 1\le r,t\le e\} \cup \{ y_r-y_t \mid e+1\le r,t\le 2e\}$. Then, in $R_{2\d}$, we have 
\begin{equation}\label{diag0}
 \vphi_w e_{\d,\d} \in \psi_w ((y_1-y_{e+1})^e +K)e_{\d,\d} + \sum_{v\in \scr D_{2e}^{(e,e)} \setminus \{ w\}} \psi_v R'_{\d,\d}. 
\end{equation}
\end{lem}

\begin{proof}
 The idea of the proof is the same as that of~\cite[Proposition 1.4.4]{KangKashiwaraKim2013}. 
Note that $w$ is fully commutative (see~\cite[Lemma 3.17]{KleshchevMathasRam2012}),
and hence $\psi_w$ does not depend on the choice of a reduced expression for $w$.
Let $\bj,\bk\in I^{\d}$ and $\bi=\bj\bk$. It is enough to show that 
\[
\vphi_w e(\bi) \in \psi_w ((y_1-y_{e+1})^e +K)e(\bi) + 
\sum_{v\in \scr D_{2e}^{(e^2)} \setminus \{ w\}} \psi_v R'_{\d,\d}
\]
for all such $\bi$.
Let $w=s_{r_{e^2}} \cdots s_{r_2} s_{r_1}$ be a reduced expression, and let $u_k = s_{r_{k-1}} \cdots s_{r_1}$ for $1\le k\le e^2$. 
We have 
\begin{equation}\label{diag1}
\vphi_w e(\bi) = \vphi_{r_{e^2}} \cdots \vphi_{r_1} e(\bi). 
\end{equation}
For each $k$, one can replace the multiple $\vphi_{r_k}$  by
$\vphi_{r_k} e(u_k \bi)$ without changing the value of the expression
on the right-hand side. 

Let $k$ run in the decreasing order through the elements of $\{ 1,\ldots,e^2\}$ satisfying $(u_k \bi)_{r_k} = (u_k \bi)_{r_k+1}$. Note that
there are exactly $e$ such values of $k$
 since each of $\bj$ and $\bk$  is a permutation of $(0,1,\ldots,e-1)$. 
For each such $k$, we have $\vphi_{r_k} e(u_k \bi) = \psi_{r_k} (y_{r_k} - y_{r_k+1}) 
e(u_k \bi) + e(u_k \bi)$. 
Consequently, the product~\eqref{diag1} decomposes as a sum of two summands, which correspond to $\psi_{r_k} (y_{r_k} - y_{r_k+1}) e(u_k \bi)$ and 
$e(u_k \bi)$ respectively. By Lemma~\ref{lem:KaKaKi}\eqref{KaKaKi3}, the first summand does not change if we remove the factor $(y_{r_k} - y_{r_k+1})$ and instead insert 
$y_{u_k^{-1} (r_k)} - y_{u_k^{-1} (r_k+1)}$
at the right end of the product (note that $u_k^{-1} (r_k)\le e$ and $u_k^{-1} (r_k+1)>e$). After these manipulations are performed for all such $k$ in the decreasing order, we have decomposed the product~\eqref{diag1} into $2^e$ summands. Of these, all the summands for which the second option was chosen at least once belong to 
$\sum_{v\in \scr D_{2e}^{(e^2)} \setminus \{ w\}} \psi_v R_{\d,\d}$ by Theorem~\ref{thm:basis}\eqref{basis2} together with the argument used to prove
Corollary~\ref{cor:parabfree}.
The remaining summand belongs to $\psi_w ((y_1-y_{e+1})^e e(\bi) +K)$, as in all cases we have 
 $y_{u_k^{-1} (r_k)} - y_{u_k^{-1} (r_{k+1})}\in y_1 - y_{e+1} +K$. 
Thus, 
$\vphi_w e(\bi) \in \psi_w ((y_1-y_{e+1})^e +K)e_{\d,\d} +a$ 
for some $a= \sum_{v\in \scr D_{2e}^{(e^2)}\sm \{w\}} \psi_v a_v$, 
with $a_v\in R_{\d,\d}e(\bi)$ for each $v$.
Since $\vphi_w e(\bi) \in R'_{2\d}$ and $\psi_w ((y_1-y_{e+1})^e +K) e(\bi) \subset R'_{2\d}$, we have 
$a\in R'_{2\d}$.
By Theorem~\ref{thm:basis}\eqref{basis1}, we can write 
$a_v = \sum_{z\in \fr S_{(e,e)}} \psi_z e(\bi) x_{v,z}$ 
with $x_{v,z} \in F[y_1,\ldots,y_{2e}]$, so that 
$a= \sum_{v\in \scr D_{2e}^{(e^2)}\sm \{w\}} 
\sum_{z\in \fr S_{(e,e)}} \psi_v \psi_z e(\bi) x_{v,z}$.
We may assume that for any such $v$ and $z$ the reduced expression for the definition of $\psi_{vz}$ is chosen in such a way that $\psi_{vz} = \psi_v \psi_z$. Then, by Theorem~\ref{thm:basis}\eqref{basis1} and Proposition~\ref{prop:tildebasis}, $x_{v,z} \in R'_{\d,\d}$ for all $v$ and $z$, and hence
$a\in \sum_{v\in \scr D_{2e}^{(e^2)} \sm \{w\}}\psi_v R'_{\d,\d}$. 
\end{proof}

\subsection{Quotients of $R_{\d}$}\label{subsec:quot0}
Throughout this subsection, we view $\psi_r$ and $y_t$ for $1\le r<e$ and $1\le t\le e$ as elements of either $R_\d$ or $R^{\Lda_0}_{\d}$ (depending on the context) via the natural projections $R_n\thra R_{\d}$ and $R^{\Lda_0}_n\thra R^{\Lda_0}_{\d}$.
Let $V$ be the $2$-sided ideal of $R_{\d}$ generated by 
$\cl E_1=\{ e(\bi) \mid \bi \in I^{\d} \setminus I^{\vn,1} \}$,
so that $\ol{R}_{\d} = R_{\d}/V$ (cf.~\S\ref{subsec:outline}).
Let $\ol{\hphantom{a}\vphantom{B}}\colon R_{\d} \thra \ol{R}_{\d}$ be the natural projection. 

\begin{lem}\label{lem:psi1}
 We have $\ol{\psi}_{1}=0$.
\end{lem}

\begin{proof}
 If not, then by the relation~\eqref{rel:epsi} there exists $\bi \in I^{\vn,1}$ such that 
 $s_1  \bi\in I^{\vn,1}$. Since every $\bj \in I^{\vn,1}$ satisfies $j_1=0$ and $j_2\ne 0$, this is impossible. 
\end{proof}

\begin{lem}\label{lem:eta}
There is a unital graded algebra homomorphism $\eta\colon R^{\Lda_0}_{\d} \to \ol{R}_{\d}$  given by 
\[
 \eta(e(\bi)) = \ol{e(\bi)}, \quad  \eta(\psi_r) = \ol{\psi}_r, \quad \eta(y_t) = \ol{y}_t - \ol{y}_1
\]
for $\bi\in I^{\d}$, $1\le r<e$ and $1\le t\le e$.
\end{lem}

\begin{proof} 
To begin with, $\eta$ is a homomorphism from the free algebra on the standard generators of $R^{\Lda_0}_{\d}$ to $\ol{R}_{\d}$.
It is immediate that $\eta$ respects the defining relations of
$R^{\Lda_0}_{\d}$ (including the cyclotomic relation 
$y_1^{\d_1,0} e(\bi)=0$), with the possible exception of relations~\eqref{rel:ypsi0}--~\eqref{rel:ypsi2} 
(note that $\eta (y_{r+1}- y_r) = \ol{y}_{r+1}- \ol{y}_r$ for $1\le r<e$ and that $\eta (y_1)=0$). 

Recall that $i_r \ne i_{r+1}$ for all $\bi \in I^{\d}$ and $1\le r<d$.  We have $\eta(\psi_1)=0$ by Lemma~\ref{lem:eta}, so both sides of~\eqref{rel:ypsi0},~\eqref{rel:ypsi1},~\eqref{rel:ypsi2} are mapped by $\eta$ to $0$ for $r=1$. For $r>1$, we have 
\[
 \eta(y_{r+1} \psi_r) = (\ol{y}_{r+1} - \ol{y}_{1}) \ol\psi_{r} = (\ol y_{r+1}\ol \psi_{r}  - \ol \psi_{r} \ol y_{1})  
= \ol\psi_{r} (\ol y_{r} - \ol y_{1}) = \eta(\psi_r y_r),
\]
where the second equality is due to~\eqref{rel:ypsi0} and the third one is due to~\eqref{rel:ypsi1}. Relations~\eqref{rel:ypsi0} and~\eqref{rel:ypsi1} for $r>1$ are checked similarly. 
%
\end{proof}

\begin{lem}\label{lem:Rbardiso}
Let $z$ be an indeterminate, and view $F[z]$ as a graded algebra with $\deg(z)=2$. 
We have a graded algebra isomorphism $R^{\Lda_0}_{\d} \otimes F[z]\isoto \ol{R}_{\d}$ given by $a \ot z^m \mapsto \eta(a)\ol{y}_1^m$ 
for all $m\ge 0$ and $a\in R^{\Lda_0}_{\d}$. 
\end{lem}

\begin{proof}
 By Lemma~\ref{lem:psi1} and the defining relations of $\ol R_{\d}$, the element $\ol{y}_1$ centralises $\eta(R^{\Lda_0}_{\d})$. Hence, it follows from Lemma~\ref{lem:eta} that we have a graded unital algebra homomorphism 
 $\xi\colon R^{\Lda_0}_{\d} \ot F[z]\to \ol{R}_{\d}$ defined as in the statement of the lemma. 
 
 By an observation in~\cite[\S 1.3.2]{KangKashiwaraKim2013} (cf.~\eqref{KaKaKitwist}), there is a unital algebra homomorphism $R_{\d} \to R_{\d}^{\Lda_0} \ot F[z]$ given by 
 \begin{equation}\label{twist2}
 e(\bi) \mapsto e(\bi) \ot 1, \quad \psi_r \mapsto \psi_r \ot 1, \quad 
  y_t \mapsto y_t \ot 1 + 1 \ot z 
 \end{equation}
 for $\bi\in I^{\d}$, $1\le r<e$, $1\le t\le e$.
 By Theorem~\ref{thm:BKdim}, we have $e(\bi)=0$ in $R^{\Lda_0}_{\d}$ for $\bi\in I^{\d}\sm I^{\vn,1}$, so the map~\eqref{twist2} factors through $\ol{R}_{\d}$. It is easy to check that the resulting homomorphism $\ol{R}_{\d}\to R^{\Lda_0}_{\d}\ot F[z]$ is both a left and a right inverse to $\xi$.   
\end{proof}

Using  the fact that $\ol R_{\d}$ is nonnegatively graded and 
Proposition~\ref{prop:d1basis}, 
we immediately deduce the following result. 

\begin{cor}\label{cor:d1_01}
The map $\eta$ restricts to a vector space isomorphism 
from $(R^{\Lda_0}_{\d})_{\{0,1\}}$ onto
$(\ol R_{\d})_{\{0,1\}}$. Moreover,
$
(\ol{R}_{\d})_{\{0,1\}} = \sum_{\bi\in I^{\vn,1}} 
\sum_{v\in \fr S_e} F\ol\psi_v \ol{e(\bi)}.
$
\end{cor}

Let $K$ be the left ideal of 
$\ol{R}_{\d}$ generated by the set $\{ \ol y_k-\ol y_t \mid 1\le k,t \le e\}$. 
For $\bi\in I^{\d}$, $1\le r<e$ and $1\le r,t\le e$, we have 
$(\ol y_k-\ol y_t) \ol \psi_r \ol{e(\bi)} = \ol{\psi}_r \ol{e(\bi)} 
(\ol y_{s_r(k)}-\ol y_{s_r(t)})$.
Hence, $K$ is a 2-sided ideal of $\ol R_{\d}$.

\begin{lem}\label{lem:ztensor}
Let $z$ be an indeterminate.
 The $F$-linear map $(\ol{R}_{\d})_{\{0,1\}} \otimes F[z] \to \ol{R}_{\d}/{K}$ given by $a\otimes z^m \mapsto a \bar y_1^m +K$ is an isomorphism of vector spaces. 
\end{lem}

\begin{proof} 
Using Proposition~\ref{prop:d1basis}, we see that 
the image of $K$ in $R^{\Lda_0}_{\d} \ot F[z]$ under the inverse of the isomorphism of 
Lemma~\ref{lem:Rbardiso} is equal to $(R^{\Lda_0}_{\d})_{\{2\}} \ot F[z]$. 
Hence, the map $(R^{\Lda_0}_{\d})_{\{0,1\}} \ot F[z] \to \ol R_{\d}/K$ given by 
$a \ot z^m \mapsto \eta(a)\ol y_1^m +K$ is an isomorphism of vector spaces. 
Due to Corollary~\ref{cor:d1_01}, the result follows.
\end{proof}

\subsection{A homomorphism from $R^{\Lda_0}_{\d} \wr \fr S_d$ to $e_{\d^d}\ol{R}_{d\d}e_{\d^d}$}\label{subsec:quot}


Fix an arbitrary integer $d\ge 0$. 
Let $\cl V$ be the $2$-sided ideal of $R_{\d^d}$ generated by the set 
\[
 \{ e(\bi^{(1)} \ldots \bi^{(d)} ) \mid (\bi^{(1)}, \ldots, \bi^{(d)}) \in (I^{\d})^{\times d} \setminus (I^{\vn,1})^{\times d} \}. 
\]
Let $\cl U$ be the $2$-sided ideal of $R_{d\d}$ generated by 
$\{ e(\bi) \mid \bi \in I^{d\d} \setminus \cl E_d \}$ (cf.~\eqref{defE}),
so that $\ol{R}_{d\d} = R_{d\d}/\cl U$.  
Define $\ol{R}_{\d^d}$ to be the image of $R_{\d^d}$ under the natural projection
$R_{d\d} \thra \ol{R}_{d\d}$.

\begin{lem}\label{lem:UV}
We have:
\begin{enumerate}[(i)]
\item\label{UV1} $\cl Ue_{\d^d} = R_{d\d} \cl V$; 
\item\label{UV2} $\cl U\cap R_{\d^d} = \cl V$;
\item\label{UV3} $\ol{R}_{d\d} e_{\d^d}$ is a free right $\ol{R}_{\d^d}$-module with basis 
$\{ \ol{\psi}_w \mid w\in \scr D_{de}^{(e^d)}\}$.
\end{enumerate}
\end{lem}

\begin{proof}
\eqref{UV1} By Lemma~\ref{lem:sh1}, $\cl V \subset \cl U$. Hence, $R_{d\d} \cl V \subset \cl U e_{\d^d}$, as $\cl Ve_{\d^d} = \cl V$. 
By Corollary~\ref{cor:parabfree}, $R_{d\d} e_{\d^d} = \sum_{w\in \scr D_{de}^{(e^d)}} \psi_w R_{\d^d}$. Since 
$R_{\d^d}= \sum_{\bi^{(1)},\ldots,\bi^{(d)} \in I^{\vn, 1}} e(\bi^{(1)} \ldots \bi^{(d)}) R_{\d^d}  + \cl V$, we have 
$R_{d\d} e_{\d^d} = \sum_{\bi\in \cl E_d} e(\bi) R_{d\d} e_{\d^d} + R_{d\d} \cl V$. Hence, for any $\bi\in I^{d\d} \setminus \cl E_d$, we have 
$e(\bi) R_{d\d} e_{\d^d} \subset R_{d\d} \cl V$. Since $\cl U = \sum_{\bi\in I^{d\d}\setminus \cl E_d} R_{d\d} e(\bi) R_{d\d}$, the inclusion 
$\cl U e_{\d^d} \subset R_{d\d} \cl V$ follows.

\eqref{UV2} and~\eqref{UV3} follow from (i) and Corollary~\ref{cor:parabfree}.  
\end{proof}

Due to Lemma~\ref{lem:UV}\eqref{UV2}, 
$\ol{R}_{\d^d}$  is naturally identified with $R_{\d^d}/\cl V$. 
Hence, the isomorphism $\iota_{\d^d} \colon R_{\d}^{\ot d} \isoto R_{\d^d}$
induces an algebra isomorphism $\iota\colon \ol{R}_{\d}^{\ot d}\isoto \ol{R}_{\d^d}$.

Throughout the rest of the section, 
symbols of the form $\psi_r$, $\psi_w$, $\vphi_r$, $\vphi_w$, $e(\bi)$, $e_{\a}$ and $y_t$ that would previously be interpreted as elements $R_{de}$ represent their images in $\ol{R}_{d\d}$. 
It follows from Lemma~\ref{lem:sh1} that, if $\bi^{(1)},\ldots, \bi^{(d)} \in I^{e}$, then we have $e(\bi^{(1)} \ldots \bi^{(d)}) = 0$ in $\ol{R}_{d\d}$ unless 
$\bi^{(1)},\ldots, \bi^{(d)} \in I^{\varnothing,1}$. Hence, 
\begin{equation}\label{eq:edd}
 e_{\d^d} = \sum_{\bi^{(1)}, \ldots, \bi^{(d)} \in I^{\varnothing,1}} e(\bi^{(1)} \ldots \bi^{(d)} ).
\end{equation}

For $1\le r<d$, let
\begin{equation}\label{w_r}
w_r = ( (r-1)e+1, re+1) ((r-1)e+2, re+2)) \cdots (re, (r+1)e) \in \fr S_{de}. 
\end{equation}
This element is fully commutative (cf.~the proof of Lemma~\ref{lem:diagram}), so 
$\sigma_r := \psi_{w_r} \in \ol R_{d\d}$ 
is well defined in the sense that it does not depend on the choice of a reduced expression for $w_r$. 
Let $B_d$ be the subgroup of $\fr S_{de}$ generated by the elements $w_r$, $1\le r<d$. Then we have a group isomorphism $o\colon \fr S_d\to B_d$, $s_r\mapsto w_r$, and 
$B_d \subset \ls{(e^d)}{\scr D}_{ed}^{(e^d)}$ (cf.~\cite[\S4.1]{KleshchevMuth2013}).  

For each $u\in \fr S_d$, choose a reduced expression $u=s_{r_1} \cdots s_{r_m}$ and define $\sigma_u = \sigma_{r_1} \cdots \sigma_{r_m}$.
We may assume that $\sigma_u = \psi_{o(u)}$, as the decomposition 
$o(u) = w_{r_1} \cdots w_{r_m}$ can be refined to a reduced expression for $o(u)$ in $\fr S_{de}$. 
Note that $\sigma_u e_{\d^d} = e_{\d^d} \sigma_u$ for all $u\in \fr S_d$. 
We have $\{1,\ldots,de\}=\bigsqcup_{r=1}^d X_r$, 
where $X_r: = \{ (r-1)e+1,\ldots,re\}$.

\begin{lem}\label{lem:Bd}
\begin{enumerate}[(i)]
\item\label{Bd1}
We have $e_{\d^d} \ol{R}_{d\d} e_{\d^d} = \sum_{u\in \fr S_d} \sigma_u \ol{R}_{\d^d}$. 
\item\label{Bd2}
For all $u\in \fr S_d$, we have $\sigma_u e_{\d^d} \in (\ol{R}_{d\d})_{\{0\}}$.
\item\label{Bd3}
The algebra 
$e_{\d^d} \ol{R}_{d\d} e_{\d^d}$ is nonnegatively graded.
\end{enumerate}
\end{lem}

\begin{proof}
For~\eqref{Bd2}, it is enough to consider the case when $u=s_r$ for some $r\in \{1,\ldots,d-1\}$ because each $\s_r$ centralises $e_{\d^d}$. The fact that $\s_r e_{\d^d}$ is homogeneous of degree $0$ is easily verified using~\eqref{eq:edd}. Part~\eqref{Bd3} follows from~\eqref{Bd1} and~\eqref{Bd2} 
 because $\ol R_{\d^d}=\iota(\ol R_{\d}^{\ot d})$ is nonnegatively graded. 

Thus, it remains only to prove~\eqref{Bd1}.
 By Theorem~\ref{thm:basis}\eqref{basis1}, we have 
 $\ol{R}_{d\d} e_{\d^d} = \sum_{w\in \scr D_{de}^{(e^d)}} \psi_w \ol{R}_{\d^d}$. If $w\in B_d$, say, $w=o(u)$, then 
$e_{\d^d} \psi_w \ol{R}_{\d^d}  = e_{\d^d} \sigma_u \ol{R}_{\d^d}=
\s_u \ol R_{\d^d}$. So it will suffice to prove the following claim: if $w\in \scr D_{de}^{(e^d)} \setminus B_d$, then 
$e_{\d^d} \psi_w \ol{R}_{\d^d} = 0$. 

If not, then by~\eqref{eq:edd} 
there exist $\bi^{(1)},\ldots,\bi^{(d)} \in I^{\vn,1}$ such that $w (\bi^{(1)} \ldots \bi^{(d)}) = \bj^{(1)} \ldots \bj^{(d)}$ for some 
$\bj^{(1)},\ldots, \bj^{(d)} \in I^{\vn,1}$.  Since 
$w\in  \scr D^{(e^d)}_{de} \setminus B_d$, there exist $k\in \{1,\ldots,d\}$ and $t\in \{2,\ldots,e\}$ such that $w((k-1)e+1))\in X_{l}$ and $w((k-1)e+t)\in X_m$ for some $m>l$. 
Let such $k,t$ be chosen so that $m$ is greatest possible. 
Let $r\in \{1,\ldots,d\}$ be such that $w ((r-1)e+1) = (m-1)e+1$ 
(note that $w^{-1} ((m-1)e+1) \equiv 1 \pmod e$ because the only zeros in 
$\bi^{(1)} \ldots \bi^{(d)}$ occur in positions with numbers congruent to $1$ modulo $e$). Note that $r\ne k$, as $(k-1)e+1$ and $(r-1)e+1$ have distinct images under $w$. 
By maximality of $m$, we have $w(X_r) \subset X_m$, whence $w(X_r)=X_m$. This contradicts the fact that $w((k-1)e+t) \in X_m$. 
\end{proof}

Recall the ideal $K$ of $\ol{R}_{\d}$ defined in~\S\ref{subsec:quot0} and let 
\[
 \cl K = \sum_{r=1}^d \iota (\ol{R}_{\d}^{\ot r-1} \ot K \ot \ol{R}_{\d}^{\ot d-r}) \subset \ol{R}_{\d^d}.
\]
In other words, $\cl K$ is the 2-sided (or, equivalently, left) 
ideal of $\ol{R}_{\d^d}$ generated by the elements of the form $(y_r-y_t)e_{\d^d}$ with $r,t\in X_k$ for $1\le k\le d$. 
We have a natural isomorphism
\begin{equation}\label{eq:isotensor}
 (\ol{R}_{\d}/K)^{\ot d} \isoto \ol{R}_{\d^d}/\cl K,\quad (x_1+K) \otimes \cdots \otimes (x_d+K) \mapsto \iota(x_1\otimes \cdots \otimes x_d) + \cl K 
\end{equation}
for $x_1,\ldots,x_d \in \ol{R}_{\delta}$.

Let $\cl J = e_{\d^d} \ol{R}_{d\d} \cl K$, so that $\cl J$ is a left ideal of $e_{\d^d} \ol{R}_{d\d}e_{\d^d}$. 
  The map $\ol{R}_{\d^d} \to e_{\d^d}\ol{R}_{d\d}e_{\d^d}/\cl J$ 
  given by $a\mapsto a+\cl J$ has kernel $\cl K$ and image
$(\ol{R}_{\d^d} +\cl J)/\cl J$. 
For any $a+\cl J\in e_{\d^d} \ol{R}_{d\d} e_{\d^d}/\cl J$ and $b\in \ol{R}_{\d^d}$, the product $(a+\cl J)(b+\cl J) = ab+\cl J$ is well-defined. 
Thus, $e_{\d^d} \ol{R}_{d\d}e_{\d^d}/\cl J$ is naturally an 
$(e_{\d^d} \ol{R}_{d\d} e_{\d^d}, (\ol{R}_{\d^d}+\cl J)/\cl J)$-bimodule.
It follows from Lemmas~\ref{lem:UV}\eqref{UV3} and~\ref{lem:Bd}\eqref{Bd1} that 
$e_{\d^d}\ol{R}_{d\d}e_{\d^d}/\cl J$ is free as a right 
$(\ol{R}_{\d^d}+\cl J)/\cl J$-module with basis 
$\{ \s_u + \cl J \mid u \in \fr S_d \}$. 
In particular, 
$(\ol{R}_{\d^d}+\cl J)/\cl J$ is a free 
right $(\ol{R}_{\d^d}+\cl J)/\cl J$-submodule of 
$e_{\d^d}\ol{R}_{d\d}e_{\d^d}/\cl J$. The subspace 
$(\ol{R}_{\d^d}+\cl J)/\cl J$ also has an 
$F$-algebra structure arising from the isomorphism 
$\ol{R}_{\d^d}/\cl K \isoto (\ol{R}_{\d^d}+\cl J)/\cl J$, 
$b+\cl K \mapsto b+\cl J$, and the right $(\ol{R}_{\d^d}+\cl J)/\cl J$-module structure on $(\ol{R}_{\d^d}+\cl J/\cl J)$ induced by this algebra structure coincides with the aforementioned module structure. These facts are used repeatedly in the sequel.

Define 
\begin{align*}
S&=((\ol{R}_{\d})_{\{0,1\}})^{\ot d} \subset \ol{R}_{\d}^{\ot d} \qquad\qquad \text{and} \\
T&=\sum_{u\in \fr S_d} \s_u  \iota(S) \subset e_{\d^d} \ol{R}_{d\d} e_{\d^d}.
\end{align*}
Note that $T$ is a graded vector subspace of 
$\ol{R}_{d\d}$. 
 For all $1\le r\le d$, set 
\[
z_r = y_{(r-1)e+1}e_{\d^d}+\cl J \in (\ol{R}_{\d^d}+\cl J)/\cl J
\]
These elements commute and 
generate the subalgebra $F[z_1,\ldots,z_d]$ of $(\ol{R}_{\d^d}+\cl J)/\cl J$. 

\begin{lem}\label{lem:fintens}
Let $\tilde{z}_1,\ldots,\tilde{z}_d$ be algebraically independent indeterminates. 
Consider the $F$-algebra $F[\tilde{z}_1,\ldots,\tilde{z}_d]$, graded by the rule $\deg (\tilde{z}_r)=2$ 
for $1\le r\le d$.
\begin{enumerate}[(i)]
 \item\label{ft1} The $F$-linear map $F\fr S_d \otimes S \otimes F[\tilde{z}_1,\ldots, \tilde{z}_d] \to e_{\d^d}\ol{R}_{d\d}e_{\d^d}/\cl J$ given by 
$u\otimes a \otimes g \mapsto \s_u \iota(a) g(z_1,\ldots,z_d)$ (for $u\in \fr S_d$) is an isomorphism of  vector spaces.
 \item\label{ft2} The $F$-linear map $T\otimes F[\tilde{z}_1,\ldots,\tilde{z}_d] \to e_{\d^d} \ol{R}_{d\d}e_{\d^d}/\cl J$ given by 
$t\otimes g \mapsto tg(z_1,\ldots,z_d)$ is a graded vector space isomorphism. 
\end{enumerate}
\end{lem}

\begin{proof}
Part~\eqref{ft1} is established by combining Lemma~\ref{lem:ztensor}, the isomorphism~\eqref{eq:isotensor} and the fact that 
$\{ \s_u+\cl J \mid u\in \fr S_d\}$ 
is a free generating set of $e_{\d^d}\ol{R}_{d\d}e_{\d^d}/\cl J$ as a right $(\ol{R}_{\d^d}+\cl J)/\cl J$-module. The map in part~\eqref{ft2} is easily seen to be a graded homomorphism. The fact that it is bijective clearly follows from~\eqref{ft1}. 
\end{proof}

\begin{cor}\label{cor:01inT}
 We have $(e_{\d^d} \ol{R}_{d\d} e_{\d^d})_{\{0,1\}} \subset T$.
\end{cor}

\begin{proof}
This follows from Lemmas~\ref{lem:Bd}\eqref{Bd3} and~\ref{lem:fintens}\eqref{ft2}. 
\end{proof}

If $\mu=(\mu_1,\ldots,\mu_l)$ 
is a composition of $d\d$, let 
$\bar\iota_\mu\colon R_{\mu_1} \otimes \cdots \otimes R_{\mu_l}\to \ol{R}_{d\d}$ be the composition 
of $\iota_{\mu}$ with the natural projection $R_{d\d} \thra \ol{R}_{d\d}$. 
Fix an integer $r$ such that $1\le r<d$. 
Let $S^{(r)} = \ol e_{\d}^{\ot r-1} \ot (\ol{R}_{\d})_{\{0,1\}} \ot (\ol{R}_{\d})_{\{0,1\}} \ot \ol e_{\d}^{\ot d-r-1} \subset S$.

\begin{lem}\label{lem:wr_in}
 We have $\vphi_{w_r} e_{\d^d} + \cl J \in \sigma_r (z_r-z_{r+1})^e + \iota(S^{(r)}) F[z_r-z_{r+1}]$. 
\end{lem}

\begin{proof}
Let $\b\colon R_{2\d} \to \ol{R}_{d\d}$ be the algebra homomorphism defined by
\[\b(x) = \bar\iota_{\d^{r-1}, 2\d, \d^{d-r-1}} 
(e_{\d}^{\ot r-1} \ot x \ot e_{\d}^{\ot d-r-1}). \]
 We use Lemma~\ref{lem:diagram} and apply $\b$
to both sides of~\eqref{diag0}. It follows from~\eqref{eq:edd} and 
the relation~\eqref{rel:epsi}
 that 
$e_{\d^d} \b(\psi_v) e_{\d^d} = 0$ in
$\ol{R}_{d\d}$ for all $v\in \scr D_{2e}^{(e^2)} \setminus \{ 1, w_1\}$, as the set $\{ 1, e+1\}$ is not $v$-invariant for any such $v$ and every element of $I^{\vn,1}$ has $0$ as its first entry.
Since $e_{\d^d} \vphi_{w_r} e_{\d^d} =\vphi_{w_r} e_{\d^d}$, we conclude that 
\[
 \vphi_{w_r} e_{\d^d} +\cl J \in \psi_{w_r} (z_r-z_{r+1})^e + Y,
\]
 in $e_{\delta^d} \ol R_{d\delta} e_{\delta^d}/\cl J$ 
 where $Y= (\b(R'_{\d^2}) + \cl J)/\cl J$. 
It follows from the definition of $\cl J$ (and of $\cl K$) together with Corollary~\ref{cor:tildeg} applied to $R'_{\d^2}$ that 
\[
 Y \subset \sum_{v\in \fr S_{(e,e)}} 
\b (\psi_v e_{\d^2}) F[z_r-z_{r+1}]. 
\]
By Corollary~\ref{cor:d1_01}, we have 
$\b(\psi_v e_{\d^2}) +\cl J \in \b(\iota_{\d^2}((R_{\d})_{\{0,1\}} \ot (R_{\d})_{\{0,1\}})) + \cl J$ for all 
$v\in \fr S_{(e,e)}$, 
and the result follows. 
\end{proof}

Observing that $\deg(\vphi_{w_r} e_{\d^d}) = 2e$, $\deg(\sigma_r e_{\d^d})=0$, 
$S^{(r)} = (S^{(r)})_{\{0,1,2\}}$ and $\deg(z_r-z_{r+1})=2$, we deduce from 
Lemma~\ref{lem:wr_in} that 
\begin{equation}\label{tau_r}
 \vphi_{w_r} e_{\d^d} +\cl J = \tau_r (z_r-z_{r+1})^e + \e_r (z_r-z_{r+1})^{e-1} 
\end{equation}
for some elements $\tau_r\in \sigma_r e_{\d^d} + \iota(S^{(r)}_{\{0\}}) \subset e_{\d^d} \ol{R}_{d\d} e_{\d^d}$ 
and $\e_r\in \iota(S^{(r)}_{\{2\}})\subset \ol R_{\d^d}$, 
which are determined uniquely due to Lemma~\ref{lem:fintens}\eqref{ft2}. 
By considering degrees, we see that 
$\e_r \in \iota(\ol{e}_{\d}^{\ot r-1} \otimes (\ol{R}_{\d})_{\{1\}} 
\ot (\ol{R}_{\d})_{\{1\}} \otimes \ol{e}_{\d}^{\ot d-r-1})$. 

For $1\le t\le d$, define a graded algebra homomorphism $\eta_t\colon R^{\Lda_0}_{\d}\to \ol{R}_{\d^d}$ by 
$\eta_t (x) = \iota(\ol{e}_{\d}^{\ot t-1} \ot \eta(x) \ot \ol{e}_{\d}^{\ot d-t})$, where $\eta\colon R^{\Lda_0}_{\d} \to \ol{R}_{\d}$ is given by Lemma~\ref{lem:eta}. 
Note that, by Corollary~\ref{cor:d1_01}, 
we have $(\ol{R}_{\d})_{\{1\}} = \eta((R^{\Lda_0}_{\d})_{\{1\}})$, and hence 
\begin{equation}\label{eps_r_in}
\e_r \in \eta_r ((R^{\Lda_0}_{\d})_{\{1\}}) \eta_{r+1} ((R^{\Lda_0}_{\d})_{\{1\}}).
\end{equation}

\begin{remark}\label{rem:KLsigma}
For any $\bi^{(1)},\ldots,\bi^{(d)} \in I^{\d}$,
the element $\s_r e(\bi^{(1)}\ldots \bi^{(d)})$ is the image of the element of $R_{d\d}$ represented by the following Khovanov--Lauda diagram:
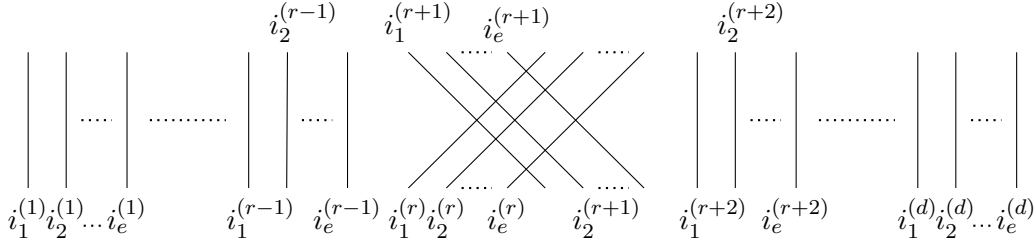
\[
\begin{tikzpicture}
\def\hh{1.8}
\def\xx{2.9}
\def\yy{\xx+2.1}
\def\zz{\yy+3.8}
\def\ww{\zz+2.9}
\draw (0,0)node[below]{$i^{(1)}_1$}--(0,\hh);
\draw (0.5,0)node[below]{$i^{(1)}_2$}--(0.5,\hh);
\draw[dots] (0.7,0.5*\hh) -- (1.1, 0.5*\hh);
\draw[dots] (0.73,-0.5) -- (1,-0.5);
\draw (1.3, 0)node[below]{$i^{(1)}_e$}--(1.3,\hh);
\draw[dots] (1.6, 0.5*\hh) -- (2.6, 0.5*\hh);   
\draw (\xx+0,0)node[below]{\hspace{2mm} $i^{(r-1)}_1$}--(\xx+0,\hh);
\draw (\xx+0.5,0)--(\xx1+0.5,\hh)node[above]{\hspace{3mm} $i_2^{(r-1)}$};
\draw[dots] (\xx+0.7,0.5*\hh) -- (\xx+1.1, 0.5*\hh);

\draw (\xx+1.3, 0)node[below]{$i^{(r-1)}_e$}--(\xx+1.3,\hh);
\draw (\yy,0)node[below]{$i_1^{(r)}$} -- (\yy+1.8, \hh); 
\draw (\yy+0.5, 0)node[below]{$i_2^{(r)}$} -- (\yy+2.3, \hh); 
\draw[dots] (\yy+0.7,0)--(\yy+1.1,0); 
\draw[dots] (\yy+0.7,\hh)--(\yy+1.1,\hh); 
\draw (\yy+1.3, 0)node[below]{$i_e^{(r)}$} -- (\yy + 3.1,\hh);
\draw (\yy+1.8, 0) -- (\yy,\hh)node[above]{\hspace{1.5mm} $i_1^{(r+1)}$}; 
\draw (\yy+2.3, 0)node[below]{\hspace{5mm} $i_2^{(r+1)}$} -- (\yy+0.5,\hh);
\draw[dots] (\yy+2.5,0)--(\yy+2.9,0); 
\draw[dots] (\yy+2.5,\hh)--(\yy+2.9,\hh); 
\draw (\yy+3.1, 0) -- (\yy+ 1.3,\hh)node[above]{\hspace{1.5mm} $i_e^{(r+1)}$};
\draw (\zz,0)node[below]{\hspace{3mm} $i^{(r+2)}_1$}--(\zz,\hh);
\draw (\zz+0.5,0)--(\zz+0.5,\hh)node[above]{\hspace{3mm} $i^{(r+2)}_2$};
\draw[dots] (\zz+0.7,0.5*\hh) -- (\zz+1.1, 0.5*\hh);
\draw (\zz+1.3, 0)node[below]{$i^{(r+2)}_e$}--(\zz+1.3,\hh);
\draw[dots] (\zz+1.6, 0.5*\hh) -- (\zz+2.6, 0.5*\hh);   
\draw (\ww,0)node[below]{$i^{(d)}_1$}--(\ww,\hh);
\draw (\ww+0.5,0)node[below]{$i^{(d)}_2$}--(\ww+0.5,\hh);
\draw[dots] (\ww+0.7,0.5*\hh) -- (\ww+1.1, 0.5*\hh);
\draw[dots] (\ww+0.73,-0.5) -- (\ww+1,-0.5);
\draw (\ww+1.3, 0)node[below]{$i^{(d)}_e$}--(\ww+1.3,\hh);
\end{tikzpicture}
\]
(cf.~Remark~\ref{rem:KLdiag}).
Any element of $S^{(r)}$ (including $\e_r$) is a linear combination of the images of diagrams of the form 
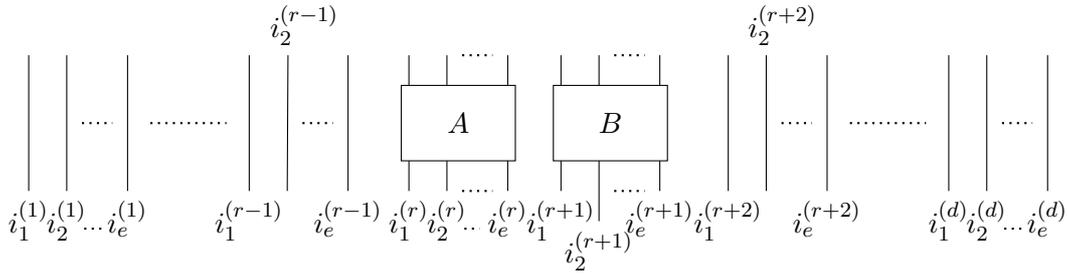
\[
\begin{tikzpicture}
\def\hh{1.8}
\def\lf{0.4}
\def\xx{2.9}
\def\yy{\xx+2.1}
\def\zz{\yy+4.2}
\def\ww{\zz+2.9}
\draw (0,0)node[below]{$i^{(1)}_1$}--(0,\hh);
\draw (0.5,0)node[below]{$i^{(1)}_2$}--(0.5,\hh);
\draw[dots] (0.7,0.5*\hh) -- (1.1, 0.5*\hh);
\draw[dots] (0.73,-0.5) -- (1,-0.5);
\draw (1.3, 0)node[below]{$i^{(1)}_e$}--(1.3,\hh);
\draw[dots] (1.6, 0.5*\hh) -- (2.6, 0.5*\hh);   
\draw (\xx+0,0)node[below]{$i^{(r-1)}_1$}--(\xx+0,\hh);
\draw (\xx+0.5,0)--(\xx1+0.5,\hh)node[above]{\hspace{3mm} $i_2^{(r-1)}$};
\draw[dots] (\xx+0.7,0.5*\hh) -- (\xx+1.1, 0.5*\hh);
\draw (\xx+1.3, 0)node[below]{$i^{(r-1)}_e$}--(\xx+1.3,\hh);
\draw (\yy,0)node[below]{$i^{(r)}_1$}--(\yy,\lf); 
\draw (\yy,\hh-\lf) -- (\yy, \hh);
\draw (\yy+0.5,0)node[below]{$i^{(r)}_2$}--(\yy+0.5,\lf);
\draw[dots] (\yy+0.73,-0.5) -- (\yy+0.93,-0.5); 
\draw (\yy+0.5,\hh-\lf) -- (\yy+0.5, \hh);
\draw[dots] (\yy+0.7,0)--(\yy+1.1,0); 
\draw[dots] (\yy+0.7,\hh)--(\yy+1.1,\hh); 
\draw (\yy+1.3,0)node[below]{$i^{(r)}_e$}--(\yy+1.3,\lf); 
\draw (\yy+1.3,\hh-\lf) -- (\yy+1.3, \hh);
\draw (\yy-0.1, \lf)rectangle(\yy+1.3+0.1,\hh-\lf);
\draw (\yy+0.65,0.5*\hh)node{$A$};
\draw (\yy+2,0)node[below]{$i^{(r+1)}_1$}--(\yy+2,\lf); 
\draw (\yy+2,\hh-\lf) -- (\yy+2, \hh);
\draw (\yy+2.5,-0.4)node[below]{$i^{(r+1)}_2$}--(\yy+2.5,\lf); 
\draw (\yy+2.5,\hh-\lf) -- (\yy+2.5, \hh);
\draw (\yy+3.3,0)node[below]{$i^{(r+1)}_e$}--(\yy+3.3,\lf); 
\draw (\yy+3.3,\hh-\lf) -- (\yy+3.3, \hh);
\draw[dots] (\yy+2.7,0)--(\yy+3.1,0); 
\draw[dots] (\yy+2.7,\hh)--(\yy+3.1,\hh); 
\draw (\yy+2-0.1, \lf)rectangle(\yy+2+1.3+0.1,\hh-\lf);
\draw (\yy+2+0.65,0.5*\hh)node{$B$}; 
\draw (\zz,0)node[below]{$i^{(r+2)}_1$}--(\zz,\hh);
\draw (\zz+0.5,0)--(\zz+0.5,\hh)node[above]{\hspace{3mm} $i^{(r+2)}_2$};
\draw[dots] (\zz+0.7,0.5*\hh) -- (\zz+1.1, 0.5*\hh);
\draw (\zz+1.3, 0)node[below]{$i^{(r+2)}_e$}--(\zz+1.3,\hh);
\draw[dots] (\zz+1.6, 0.5*\hh) -- (\zz+2.6, 0.5*\hh);   
\draw (\ww,0)node[below]{$i^{(d)}_1$}--(\ww,\hh);
\draw (\ww+0.5,0)node[below]{$i^{(d)}_2$}--(\ww+0.5,\hh);
\draw[dots] (\ww+0.7,0.5*\hh) -- (\ww+1.1, 0.5*\hh);
\draw[dots] (\ww+0.73,-0.5) -- (\ww+1,-0.5);
\draw (\ww+1.3, 0)node[below]{$i^{(d)}_e$}--(\ww+1.3,\hh);
\end{tikzpicture}
\]
where $\bi^{(1)},\ldots,\bi^{(d)}\in I^{\d}$ and $A,B$ are diagrams representing elements of $R_{\d}$. 
\end{remark}

Now we can state the main result of this section. The equalities in parts~\eqref{Theta1}--\eqref{Theta3} of the following theorem 
are between elements of $\ol{R}_{d\d}$.

\begin{thm}\label{thm:Theta}
For $1\le r<d$, we have:
 \begin{enumerate}[(i)]
  \item\label{Theta1} $\tau_r^2 = e_{\d^d}$; 
  \item\label{Theta2} $\tau_r \tau_{r+1} \tau_r = \tau_{r+1} \tau_r \tau_{r+1}$ if $r<d-1$; 
  \item\label{Theta2a}
  $\tau_r \tau_t = \tau_t \tau_r$ if $1\le t<d$ and $|t-r|\ge 2$. 
  \item\label{Theta3} if $x\in R^{\Lda_0}_{\d}$, then $\eta_r (x) \tau_r = \tau_r \eta_{r+1} (x)$ and $\eta_{r+1} (x) \tau_r = \tau_r \eta_{r} (x)$.  
  \item\label{Theta4} $y_{re+1}e_{\d^d} \in \tau_r y_{(r-1)e+1} \tau_r+ \cl T$, where
$\cl T$ is the unital subalgebra of $e_{\d^d} \ol{R}_{\d^d} e_{\d^d}$ generated by
$T \cup \cl K$.
\end{enumerate} 
\end{thm}

\begin{remark}\label{rem:red}
 If one of parts~\eqref{Theta1},\eqref{Theta2},\eqref{Theta4} of Theorem~\ref{thm:Theta}
 holds for $d=2$, then it holds in general, and if part~\eqref{Theta3} holds for $d=3$, then it holds in general. This may be seen by applying the graded algebra homomorphism $\zeta_{r,l}\colon \ol{R}_{l\d} \to \ol{R}_{d\d}$ 
induced by the map $R_{l\d}\to \ol{R}_{d\d}$, 
$x\mapsto \bar\iota_{\d^{r-1},l\d,\d^{d-r-l+1}} (e_{\d}^{\ot {r-1}} \ot x \ot e_{\d}^{\ot d-r-l+1})$,
with $l=3$ for part~\eqref{Theta3} and $l=2$ for the other parts. 
(The fact that this map factors through $\ol{R}_{l\d}$ follows easily from 
the description of $\cl E_d$ in \S\ref{subsec:outline}.) 
In particular, it follows from the definitions that 
$\tau_r = 
\zeta_{r,2} (\ol e_{\d}^{\ot r-1} \ot \tau_1 \ot 
\ol e_{\d}^{\ot d-r-1})$. 
\end{remark}

Due to Lemma~\ref{lem:fintens}\eqref{ft2}, we have 
$e_{\d^d}\ol{R}_{d\d} e_{\d^d}/\cl J= \bigoplus_{m\ge 0} T (F[z_1,\ldots,z_d]_{\{2m\}})$. Define the linear map 
$\pi_m \colon e_{\d^d}\ol{R}_{d\d} e_{\d^d}/\cl J \to e_{\d^d} \ol{R}_{d\d} e_{\d^d}/\cl J$
as the projection onto the component 
$T (F[z_1,\ldots,z_d]_{\{2m\}})$. We write $\tau=\tau_1$, $\e=\e_1$ and $\s=\s_1$.

\begin{proof}[Proof of Theorem~\ref{thm:Theta}(\ref{Theta3})]
We include only the proof of the first equality in part~(\ref{Theta3}), as the proof of the second one is similar. Due to Remark~\ref{rem:red}, we may (and do) assume that $d=2$ and $r=1$. 
First, we consider the case when $e>2$. 
Due to Proposition~\ref{prop:d1basis} and Lemma~\ref{lem:d1gen}\eqref{d1gen2}, it is enough to prove the desired equality for all $x\in \cl B_0\sqcup \cl B_1$, where $\cl B_0, \cl B_1$ are as in Proposition~\ref{prop:d1basis}. Thus, we may assume that $x=\psi_{w_{\bj,\bi}} e(\bi)$
for some $\bi,\bj\in I^{\vn,1}$ with $j_e \in \{i_e, i_e+1,i_e-1\}$. 
Note that $x = \vphi_{w_{\bj,\bi}} e(\bi)$ because the entries of $\bi$ are pairwise distinct (cf.~\eqref{defvphi}).
We view $w_{\bj,\bi}$ as an element of $\fr S_{2e}$ via the embedding 
$\fr S_e\to \fr S_{2e}$, $s_k \mapsto s_k$, and we denote by $w'_{\bj,\bi}$ the image of $w_{\bj,\bi}$ under the embedding $\fr S_e \to \fr S_{2e}$,
$s_k \mapsto s_{e+k}$ (for $1\le k<e$).  
By definitions of $\vphi_{w_1}$ and $\eta_t$, we have 
$\eta_1 (e(\bi)) \vphi_{w_1} = \vphi_{w_1} \eta_2(e (\bi))$.
Since $w_{\bj,\bi} w_1 = w_1 w'_{\bj,\bi}$ and 
$w_1 \in \ls{(e,e)}{\scr D}_{2e}^{(e,e)}$, 
using~\eqref{vphiprod}, we obtain
\[
\eta_1 (x) \vphi_{w_1} e_{\d^d} = \vphi_{w_{\bj,\bi}} \eta_1 (e(\bi)) \vphi_{w_1} =
\vphi_{w_1} \vphi_{w'_{\bj,\bi}} \eta_2 (e (\bi)) 
= \vphi_{w_1} \eta_{2} (x) e_{\d^d}. 
\]
Using~\eqref{tau_r}, we deduce the following equalities in $\ol{R}_{2\d}/\cl J$:
\begin{align}
   \eta_1 (x) \left(\tau (z_1-z_2)^e + \e (z_1-z_{2})^{e-1}\right)  
&= \tau (z_1-z_{2})^e \eta_{2} (x) + \e (z_1-z_{2})^{e-1} \eta_{2} (x)  \notag\\
&= \tau \eta_{2} (x) (z_1-z_{2})^e  + \e \eta_{2} (x) (z_1-z_{2})^{e-1}, \label{prTheta2}
\end{align}
where the second equality follows from relations~\eqref{rel:ypsi0},~\eqref{rel:ypsi1},~\eqref{rel:ypsi2}. 

We claim that $\pi_e (\eta_1 (x) \e (z_1-z_{2})^{e-1})=0$. 
If $\deg(x)=0$, then $\eta_1 (x) \e\in \iota(S_{\{2\}})$, and the claim follows. 
If $\deg(x) =1$, then, using~\eqref{eps_r_in} and 
Corollary~\ref{cor:d1_01}, we have 
\begin{equation}\label{inK}
 \eta_1 (x) \e \in \eta_1 ((R^{\Lda_0}_{\d})_{\{2\}}) 
\eta_2 ((R^{\Lda_0}_{\d})_{\{1\}}) 
\subset \cl K,
\end{equation}
where the second inclusion follows from the fact that 
$\eta((R^{\Lda_0}_{\d})_{\{2\}}) = 
F\{ \psi_{w_{\bj,\bi}}(y_e-y_1) e(\bi) \mid \bi,\bj\in I^{\vn,1} \}$. This concludes the proof of the claim. A similar argument shows that 
$\pi_e (\e \eta_{2} (x) (z_1-z_{2})^{e-1})=0$. Hence, applying $\pi_e$ to Equation~\eqref{prTheta2} and using Lemma~\ref{lem:fintens}\eqref{ft2}, we obtain 
\[
 \pi_0 (\eta_{1} (x) \tau + \cl J) = 
 \pi_0 (\tau \eta_{2} (x) + \cl J).
\]
By Corollary~\ref{cor:01inT}, we have $\eta_1 (x) \tau, \tau \eta_{2} (x) \in T$. Thus, 
$\eta_1 (x) \tau = \tau \eta_{2} (x)$ by Lemma~\ref{lem:fintens}\eqref{ft2}.

Now assume that $e=2$. This case is treated by a direct calculation, as follows. First, note that $e_{\d^2} = e(0101)$ and that 
\begin{equation}\label{Theta3_zeros}
\psi_{1} e_{\d^2} =0= \psi_{3} e_{\d^2}. 
\end{equation}
Indeed, $\psi_{3} e(0101) = e(0110) \psi_3=0$ because $(0110)\notin \cl E_2$, and the other equality is proved similarly. 
We have 
\begin{align}
 \vphi_{w_1} e_{\d^2} &= \vphi_2 \vphi_1 \vphi_3 \vphi_2 e (0101) = 
\psi_2 \vphi_1 e(0011) \vphi_3 \vphi_2 = 
\psi_2 (\psi_1 (y_1-y_2)+1)e(0011)\vphi_3 \vphi_2 \notag \\
&= \psi_2 \psi_1 \vphi_3 e(0011) \vphi_2 (y_1-y_3) + \psi_2 \vphi_3 e(0011) \vphi_2, \label{Theta3_1}
\end{align}
where the last equality is due to Lemma~\ref{lem:KaKaKi}\eqref{KaKaKi3}. 
Further, 
\[
 \vphi_3 e (0011)\vphi_2 = (\psi_3 (y_3-y_4) +1) \vphi_2 e (0101)=
\psi_3 \psi_2 e (0101) (y_2-y_4) + \psi_2 e (0101).
\]
Substituting this into~\eqref{Theta3_1}, we obtain 
\begin{equation*}
 \vphi_{w_1} e_{\d^2} = 
(\psi_2\psi_1\psi_3 \psi_2 (y_1-y_3)(y_2-y_4)
+ \psi_2 \psi_1 \psi_2 (y_1-y_3)
+ \psi_2 \psi_3 \psi_2 (y_2-y_4) 
+ \psi_2^2)e_{\d^2}.
\end{equation*}
Since $\psi_2^2 e_{\d^2} = - (y_2-y_3)^2 e_{\d^2}$ by~\eqref{rel:psisq}, 
we deduce that 
\begin{equation}\label{Theta3_2}
 \vphi_{w_1} e_{\d^2} + \cl J = 
(\sigma -1) (z_1-z_2)^2 + (\psi_2 \psi_1 \psi_2+\psi_2\psi_3\psi_2) (z_1-z_2). 
\end{equation}
Now by the braid relations~\eqref{rel:psibr} and by~\eqref{Theta3_zeros}, we have
\[
 \psi_2 \psi_1 \psi_2 e_{\d^2} + \cl J = 
(\psi_1 \psi_2 \psi_1 + 2y_2 - y_1 - y_3) e_{\d^2} +\cl J
 = z_1 - z_2.
\]
Similarly, 
$\psi_2 \psi_3 \psi_2 e_{\d^2} + \cl J = z_1 - z_2.$
Substituting the last two identities into~\eqref{Theta3_2}, we obtain
$\vphi_{w_1} e_{\d^2} + \cl J = (\sigma +1) (z_1-z_2)^2$, 
whence $\tau = (\sigma+1) e_{\d^2}$. 

By Proposition~\ref{prop:d1basis}, 
it is enough to show that $\eta_1 (x) \tau = \tau \eta_2(x)$ for 
each $x\in \{ e_{\d^2}, y_2 e_{\d^2}\}$. For $x=e_{\d^2}$, this is clear, whereas for $x=y_2 e_{\d^2}$, we have 
\begin{align*}
 \eta_1(x) \tau &= (y_2-y_1) (\psi_2\psi_1\psi_3\psi_2 +1) e_{\d^2} \\
 &= 
\psi_2 \psi_1 y_3 \psi_3  e (0011) \psi_2 - \psi_2 y_1 \psi_1 e (0011)  \psi_3 \psi_2 
+ (y_2-y_1) e (0101) \\
&= (\psi_2 \psi_1  (\psi_3 y_4 -1) \psi_2  
- \psi_2 (\psi_1 y_2 -1) \psi_3 \psi_2  + y_2-y_1) e (0101) \\
&= (\psi_2 \psi_1 \psi_3 \psi_2 y_4 - \psi_2 \psi_1 \psi_2 - \psi_2 \psi_1 \psi_3 \psi_2 y_3 + \psi_2 \psi_3 \psi_2 +y_2-y_1) e(0101)\\
&
\hspace{-0.75cm}\begin{array}{c}
=
\big(\sigma (y_4-y_3) + y_2 -y_1 
 - (\psi_1\psi_2\psi_1 + 2y_2 - y_1 -y_3) \nopagebreak \\
 \hspace{40.5mm} + (\psi_3\psi_2\psi_3 - 2y_3 +y_2+y_4)\big) e(0101) \\
\end{array} \\
&= (\sigma +1)(y_4-y_3) e(0101) = \tau \eta_2 (x)
\end{align*}
by~\eqref{rel:ypsi0},~\eqref{rel:ypsi1},~\eqref{rel:ypsi2},~\eqref{rel:psibr}
and~\eqref{Theta3_zeros}. 
\end{proof}

\begin{lem}\label{lem:J}
The set $\cl J$ is a 2-sided ideal of $e_{\d^d} \ol R_{d\d} e_{\d^d}$.
\end{lem}

\begin{proof}
Clearly, it suffices to show that $\cl K e_{\d^d} \ol R_{d\d} e_{\d^d} \subset \cl J$.
By Lemma~\ref{lem:Bd}\eqref{Bd1}, it is enough to prove that 
$\cl K \s_r e_{\d^d} \subset \cl J$ whenever $1\le r<d$ 
(as $\cl K$ is a 2-sided ideal of $\ol R_{\d^d}$). 
It follows from Proposition~\ref{prop:d1basis} that 
$\cl K$ is generated, as a left ideal of $\ol R_{\d^d}$, 
by the set $\bigcup_{t=1}^d \eta_t ((R^{\Lda_0}_\d)_{\{2\}})$.
Also, 
$\s_r e_{\d^d} \in \tau_r + \ol R_{\d^d}$. 
By Theorem~\ref{thm:Theta}\eqref{Theta3},
we have 
$\eta_t ((R^{\Lda_0}_{\d})_{\{2\}}) \tau_r = \tau_r \eta_{s_r(t)} 
((R^{\Lda_0}_{\d})_{\{2\}})\subset \cl J$ for $1\le t\le d$, and the result follows. 
\end{proof}

Set $H= \iota\circ (\eta^{\ot d})\colon (R^{\Lda_0}_{\d})^{\ot d} \to \ol{R}_{\d^d}$, so that 
$H(x_1 \ot \cdots \ot x_d) = \eta_1(x_1)\cdots \eta_d (x_d)$ for all $x_1,\ldots,x_d\in R^{\Lda_0}_{\d}$. 

\begin{lem}\label{lem:inT}
\begin{enumerate}[(i)]
\item\label{inT1} The unital subalgebra of $e_{\d^d} \ol{R}_{d\d} e_{\d^d}$ generated by 
$\tau_1,\ldots,\tau_{d-1}$ is contained in $T_{\{0\}}$. 
\item\label{inT2}
 The unital subalgebra of $e_{\d^d} \ol{R}_{d\d} e_{\d^d}$ generated by 
$\{\tau_1,\ldots,\tau_{d-1}\} \cup \iota(S)$ is contained in $T+\cl J$. 
\end{enumerate}
\end{lem}

\begin{proof}
 By Corollary~\ref{cor:01inT}, we have $\tau_{r_1} \cdots \tau_{r_m}\in (e_{\d^d} \ol{R}_{d\d} e_{\d^d})_{\{0\}} \subset T_{\{0\}}$ for any $r_1,\ldots,r_{m}\in \{1,\ldots, d-1\}$, which proves~\eqref{inT1}.

 Let $\scr S = H((R^{\Lda_0}_{\d})^{\ot d})$. 
It follows from the definition of $\eta$ that 
$\iota(S) \subset \scr S$. 
By Theorem~\ref{thm:Theta}\eqref{Theta3}, $\scr S \tau_r = \tau_r \scr S$ for $1\le r<d$. Hence, the subalgebra defined in~\eqref{inT2} is contained in the sum of 
the subspaces of the form $\tau_{r_1} \cdots \tau_{r_m} \scr S$ with $r_1,\ldots,r_m \in \{1,\ldots,d-1\}$. 
We have 
\begin{equation}\label{eddT}
T_{\{0\}} \iota(S) = \sum_{u\in \fr S_d} \sigma_u \iota(S_{\{0\}})\iota(S)=
\sum_{u\in \fr S_d} \sigma_u \iota(S) = T,
\end{equation}
where the first equality holds by Lemma~\ref{lem:Bd}\eqref{Bd1}. 
For $1\le r\le d$, we have $\eta_r ((R^{\Lda_0}_{\d})_{\{2\}}) \subset \cl K$ by Proposition~\ref{prop:d1basis}. Hence, 
$\scr S\subset \iota(S)+\cl K\subset \iota(S)+\cl J$. Therefore, 
\[
 \tau_{r_1} \cdots \tau_{r_m} \scr S \subset \tau_{r_1} \cdots \tau_{r_m} \iota(S) + \cl J \subset T_{\{0\}} \iota(S)+ \cl J \subset T+ \cl J. \qedhere 
\]
\end{proof}

By Lemma~\ref{lem:J}, we have a natural algebra structure on
$e_{\d^d} \ol R_{d\d} e_{\d^d}/{\cl J}$. 

\begin{lem}\label{lem:zvphi} For $1\le t\le d$, 
the equality
$z_t \vphi_{w_r} e_{\d^d} = \vphi_{w_r} z_{s_r (t)}$ holds in 
$e_{\d^d}\ol{R}_{d\d} e_{\d^d}/\cl J$. 
\end{lem}

\begin{proof}
We have 
\[
z_t \vphi_{w_r} e_{\d^d} =  
y_{(t-1)e+1} \vphi_{w_r} e_{\d^d} + \cl J 
= \vphi_{w_r} y_{(s_r (t)-1)e+1} e_{\d^d} + \cl J = \vphi_{w_r} z_{s_r (t)},
\]
where the second equality is due to Lemma~\ref{lem:KaKaKi}\eqref{KaKaKi3}. 
\end{proof}


\begin{proof}[Proof of Theorem~\ref{thm:Theta}
(\ref{Theta1}),(\ref{Theta2}),(\ref{Theta2a}),(\ref{Theta4})]
Part~\eqref{Theta2a} is clear from the definitions. 
For the remaining parts,
we will repeatedly use Lemma~\ref{lem:zvphi} 
without explicit reference.
By Remark~\ref{rem:red}, we may assume that $r=1$. 
We consider part~\eqref{Theta1}. By Lemma~\ref{lem:KaKaKi}\eqref{KaKaKi5}, 
\begin{equation}\label{prTheta1_1}
\vphi_{w_1}^2 e_{\d^d} + \cl J= (-1)^{e} (z_1-z_2)^{2e}.
\end{equation} 
On the other hand, by~\eqref{tau_r}, 
\begin{align}
\notag \vphi_{w_1}^2 e_{\d^d} +\cl J &= (\tau (z_1-z_2)^e + \e (z_1-z_2)^{e-1}) \vphi_{w_1} \\
\notag &= (-1)^e \tau \vphi_{w_1}  (z_1-z_2)^e + 
(-1)^{e-1} \e \vphi_{w_1} (z_1-z_2)^{e-1} \\
\label{prTheta1_2} &= (-1)^e \tau^2 (z_1-z_2)^{2e} + (-1)^e \tau \e (z_1-z_2)^{2e-1} \\
\notag &+ (-1)^{e-1} \e \tau (z_1-z_2)^{2e-1} + (-1)^{e-1} \e^2 (z_1-z_2)^{2e-2}.  
\end{align}
By Lemma~\ref{lem:inT}\eqref{inT2}, we have 
$\e \tau, \tau \e, \e^2 \in T+\cl J$. 
Therefore, applying $\pi_{2e}$ to the right-hand sides of~\eqref{prTheta1_1} and~\eqref{prTheta1_2}, we obtain 
$(z_1-z_2)^{2e} = \tau^2 (z_1-z_2)^{2e}$. 
By Lemma~\ref{lem:inT}\eqref{inT1}, we have $\tau^2\in T$, whence $\tau^2 = e_{\d^d}$ by Lemma~\ref{lem:fintens}\eqref{ft2}. 

For~\eqref{Theta2}, writing $\nu_0=\tau_1$, $\nu_1=\e_1$, $\nu'_0=\tau_2$, $\nu'_1=\e_2$, we have
\begin{align*}
 \vphi_{w_1} \vphi_{w_2} \vphi_{w_1} e_{\d^d} + \cl J &= 
\sum_{a\in\{0,1\}} \nu_a (z_1-z_2)^{e-a} \vphi_{w_2} \vphi_{w_1} \\ 
&= \sum_{a\in \{0,1\}} \nu_a \vphi_{w_2} \vphi_{w_1} (z_2 - z_3)^{e-a} \\
&= \sum_{a,b\in \{0,1\}} \nu_a \nu'_b (z_2-z_3)^{e-b} \vphi_{w_1} 
(z_2 - z_3)^{e-a} \\ 
&= \sum_{a,b\in \{0,1\}} \nu_a \nu'_b \vphi_{w_1} (z_1-z_3)^{e-b} (z_2-z_3)^{e-a} \\
&= \sum_{a,b,c\in \{0,1\}} \nu_a \nu'_b \nu_c (z_1-z_2)^{e-c} 
(z_1-z_3)^{e-b} (z_2-z_3)^{e-a}.
\end{align*}

By Lemma~\ref{lem:inT}\eqref{inT2}, we have 
$\nu_a\nu'_b\nu_c\in T+\cl J$ in all cases.
Hence, $\pi_{3e}$ fixes the summand corresponding to 
$a=b=c=0$ and sends the other 7 summands to $0$, so 
\[
\pi_{3e} (\vphi_{w_1} \vphi_{w_2} \vphi_{w_1} e_{\d^d} + \cl J) = 
\tau_1 \tau_2 \tau_1 (z_1-z_2)^e (z_1-z_3)^e (z_2 - z_3)^e.
\] 
A similar computation yields 
\[
\pi_{3e} (\vphi_{w_2} \vphi_{w_1} \vphi_{w_2} e_{\d^d} + \cl J)
=\tau_{2} \tau_1 \tau_{2} (z_1-z_2)^e (z_1-z_3)^e (z_2 - z_3)^e.
\] 
By Lemma~\ref{lem:KaKaKi}\eqref{KaKaKi2}, we have 
$\vphi_{w_1} \vphi_{w_2} \vphi_{w_1} = \vphi_{w_2} \vphi_{w_1} \vphi_{w_2}$. 
By Lemmas~\ref{lem:inT}\eqref{inT1} and~\ref{lem:fintens}\eqref{ft2}, the equality 
$\tau_1 \tau_2 \tau_1 = \tau_2 \tau_1 \tau_2$ follows. 

For~\eqref{Theta4}, 
we compute
\[
 \vphi_{w_1} z_1 \vphi_{w_1} = \vphi_{w_1}^2 z_2 = (-1)^e (z_1-z_2)^{2e} z_2
\]
using~\eqref{prTheta1_1}
and, on the other hand, 
\begin{align*}
 \vphi_{w_1} z_1 \vphi_{w_1} &= 
(\tau (z_1-z_2)^e + \e (z_1-z_2)^{e-1}) z_1 \vphi_{w_1} \\
&=\tau z_1 \vphi_{w_1} (z_2-z_1)^e + \e \vphi_{w_1}z_2 (z_2-z_1)^{e-1} \\
&= (-1)^e \tau y_{1} \tau (z_1-z_2)^{2e} 
+ (-1)^e \tau z_1 \e (z_1-z_2)^{2e-1} + \e \vphi_{w_1} z_2 (z_2-z_1)^{e-1} \\
&= (-1)^e \tau y_{1} \tau (z_1-z_2)^{2e} 
+ (-1)^e \tau z_1 \e (z_1-z_2)^{2e-1} \\
&+ (-1)^{e-1} \e \tau z_2 (z_1-z_2)^{2e-1} + (-1)^{e-1} \e^2 z_2 (z_1-z_2)^{2e-2}. 
\end{align*}
Hence, using Lemma~\ref{lem:inT}\eqref{inT2} (and the equality $z_1\e=\e z_1$), we obtain 
\[
 (-1)^e (z_1-z_2)^{2e} z_2 = \pi_{2e+1} (\vphi_{w_1} z_1 \vphi_{w_1}) 
= \pi_{2e+1} ((-1)^e \tau y_{1} \tau (z_1-z_2)^{2e}), 
\]
whence, by Lemma~\ref{lem:fintens}\eqref{ft2}, we have
$z_2 = \pi_1 (\tau y_1 \tau + \cl J)$. 
Since $\deg(\tau y_1 \tau)=2$, it follows  that 
$\tau y_1 \tau + \cl J = z_2 + (x+\cl J)$ for some $x\in T_{\{2\}}$. Hence,
$
y_{e+1} - \tau y_{1} \tau \in T+\cl J.
$
By Lemma~\ref{lem:Bd}\eqref{Bd1},
 $\cl J = \sum_{u\in \fr S_d} \sigma_u \cl K \subset T\cl K$, 
 and~\eqref{Theta4} follows. 
\end{proof}

Write $I^{(\vn,1)^d} = \{ (\bi^{(1)}\ldots \bi^{(d)})
\mid \bi^{(1)}, \ldots, \bi^{(d)} \in I^{\vn,1} \}$. 
Note that $e_{\d^d} = \sum_{\bi\in I^{(\vn,1)^d}} e(\bi)$ by Lemma~\ref{lem:sh1}. 
Recall that 
$(R^{\Lda_0}_{\d})^{\ot d}$ and $F\fr S_d$ are identified with subalgebras 
of $R^{\Lda_0} \wr \fr S_d$, and recall 
the homomorphism $H\colon (R_{\d}^{\Lda_0})^{\ot d} \to \ol{R}_{\d^d}$ defined before Lemma~\ref{lem:inT}.

\begin{cor}\label{cor:hom}
\begin{enumerate}[(i)]
\item\label{hom1}
 There is a graded unital algebra homomorphism 
\[
\Theta \colon R^{\Lda_0}_{\d}\wr \fr S_d \to e_{\d^d} \ol{R}_{d\d} e_{\d^d}
\]
 given by $\Theta (a)= H(a)$ 
and $\Theta(s_r) = \tau_r$ 
for all $a\in (R^{\Lda_0}_{\d})^{\ot d}$ and $1\le r<d$.
\item\label{hom2} Suppose that $\om\colon e_{\d^d} \ol{R}_{d\d} e_{\d^d}\thra A$ is a unital algebra homomorphism onto an algebra $A$ such that for all $\bi\in I^{(\vn,1)^d}$ we have 
$\om(y_1e(\bi))\in 
F\{ \om((y_e-y_1)e(\bi)) \}$. Then the composition 
$\om\circ \Theta \colon R^{\Lda_0}_{\d} \wr \fr S_d \to A$ is surjective. 
\end{enumerate}
\end{cor}

\begin{proof}
 The fact that $\Theta$ is a homomorphism follows directly from Theorem~\ref{thm:Theta}\eqref{Theta1}--\eqref{Theta3}. 
Since $H$ is graded and 
$\deg(s_r)=0=\deg(\tau_r)$ for $1\le r<d$, we see that $\Theta$ is graded as well. 

For~\eqref{hom2}, observe 
that $\iota(S) \subset \im H \subset \im \Theta$, whence 
$\s_r e_{\d^d} \in \tau_r + \iota(S) \subset \im \Theta$ for $1\le r<d$.
By Lemma~\ref{lem:Bd}\eqref{Bd1} and Corollary~\ref{cor:d1_01}, 
the algebra 
$e_{\d^d} \ol{R}_{d\d} e_{\d^d}$ is generated by 
the subset $T \cup F[y_1,\ldots,y_{de}] e_{\d^d}$. Further,
$T= \sum_{u\in \fr S_d} \s_u \iota(S) \subset \im \Theta$ (cf.~\eqref{eddT}), so it is enough to show that $\om(y_l e_{\d^d})\in \im(\om\circ \Theta)$ whenever $1\le l\le de$. Note also that $\cl K= \cl K_{\ge 2}$ and
 $\cl K_{\{2\}} \subset \im H$ because
$\cl K_{\{2\}}$ is generated by $\iota(S)_{\{0\}}$ and the elements of the form
$(y_t-y_{t'}) e(\bi)$ with $t,t'\in \{(k-1)e+1,\dots,ke\}$ for some 
$k\in \{1,\ldots,d\}$ and
 $\bi\in I^{(\vn,1)^d}$, and all such elements 
 $(y_t-y_{t'}) e(\bi)$ belong to $\im H$ by definition of $\eta$. 

Let $\cl T$ be defined as in Theorem~\ref{thm:Theta}\eqref{Theta4};
then $\cl T_{\{2\}} \subset \im \Theta$. 
Now we prove by induction on $r$ that 
$\om(y_{re+1} e_{\d^d}) \in \im (\om \circ\Theta)$ for all $r=1,\ldots,d$.
The case $r=1$ holds by the hypothesis, 
and the inductive step follows from the fact that
$y_{re+1} e_{\d^d} - \tau_r y_{(r-1)e+1} \tau_r \in \cl T_{\{2\}}$ for $1\le r< d$, which is derived from Theorem~\ref{thm:Theta}\eqref{Theta4} by considering degrees. 
Since $\cl K_{\{2\}}\subset \im H$, it follows that $\om(y_l e_{\d^d})\in \im (\om \circ \Theta)$ for all $l$, as claimed.
\end{proof}

\section{Surjectivity of the homomorphism}\label{sec:surj}

Let $\cl H_{\rho,d} $ be a RoCK block of residue $\kappa$.
We have constructed all the maps in the diagram~\eqref{comp}: 
the homomorphism $\Om$ is defined after Proposition~\ref{prop:im_om}, 
and $\Theta$ is defined by Corollary~\ref{cor:hom}\eqref{hom1}.
Thus, we have a graded unital algebra homomorphism
$\Xi =\Xi^{(d)}:= \Om \circ \rot_{\k} \circ \Theta\colon R^{\Lda_0}_{\d} \wr \fr S_d \to C_{\rho,d}$.
Due to Proposition~\ref{prop:qdim}, in order to complete the proof of Theorem~\ref{thm:main2}, 
we only need to show that $\Xi$ is surjective. 
First, we state and prove Proposition~\ref{prop:d1nz}, which applies to the case when $d=1$.
Using Corollary~\ref{cor:hom}\eqref{hom2}, we will then deduce surjectivity of $\Xi$ in the general case in~\S\ref{subsec:endproof}.



\subsection{The case $d=1$}\label{subsec:d=1}
In this subsection, we assume that $\rho$ is a 
Rouquier $e$-core of residue $\k$ for the integer $1$, write $C=C_{\rho,1}$ 
and consider the homomorphism
$\Xi:=\Xi^{(1)} \colon R^{\Lda_0}_{\d} \to C$. 
By Proposition~\ref{prop:qdim}, 
we have $\qdim C = \qdim (R^{\Lda_0}_{\d})$, and hence 
$C= C_{\{0,1,2\}}$ by Proposition~\ref{prop:d1basis}. 

For $\bi\in I^{e}$, define
$e'(\bi)=\sum_{\bj\in I^{\rho,0}} e(\bj\bi)\in R^{\Lda_0}_{|\rho|+e}$. Set 
\begin{equation*} 
f=\BK_{|\rho|+e} (f_{\rho,1}) = \sum_{\bj\in I^{\rho,0}, \, \bi\in I^{\vn,1}_{+\k}} e(\bj\bi)\in R^{\Lda_0}_{|\rho|+e};
\end{equation*}
the second equality follows from~\eqref{BKf} and Proposition~\ref{prop:im_om}\eqref{im_om2}.
By the definitions, the map 
$\Xi\colon R^{\Lda_0}_{\d} \to C_{\rho,1}$ is given by 
\begin{equation}\label{Xi1}
e(\bi) \mapsto e'(\bi^{+\k}), 
\quad \psi_r \mapsto \psi_{|\rho|+r} f, \quad 
y_t \mapsto (y_{|\rho|+t}-y_{|\rho|+1}) f
\end{equation}
for $\bi\in I^{\d}$, $1\le r<e$ and $1\le t\le e$. 
 By Proposition~\ref{prop:im_om}, the algebra $C$ is generated by the set
\begin{equation*} 
 \{ e'(\bi^{+\k}) \mid \bi \in I^{\vn,1} \} \cup 
\{ \psi_{|\rho|+r} f \mid 1\le r<e\} 
\cup \{ y_{|\rho|+r} f \mid 1\le r\le e\}.
\end{equation*}
It follows that $\Xi$ restricts to a surjective vector space homomorphism from
$(R^{\Lda_0}_\d)_{<2}$ onto $C_{<2}$, which is seen 
to be an isomorphism by comparing dimensions.
In particular, $R^{\Lda_0}_{\d}/(R^{\Lda_0}_{\d})_{\{2\}} \cong C/C_{\{2\}}$ as graded algebras.

\enlargethispage{1.5mm}
\begin{lem}\label{lem:dimC2}
 For all $\bi \in I^{\vn,1}$, we have $\dim(e'(\bi^{+\k}) C_{\{2\}} e'(\bi^{+\k}))=1$.
\end{lem}

\begin{proof}
 Recall the notation introduced before Lemma~\ref{lem:preimage}.
By Theorem~\ref{thm:BKdim} and Lemma~\ref{lem:CK}, 
$\qdim(e'(\bi^{+\k}) R^{\Lda_0}_{\cont(\rho)+\d} e'(\bi^{+\k}))=
\sum_{\mtt t,\mtt u} q^{\deg(\mtt t)+\deg(\mtt u)}$
where the sum is over all pairs $(\mtt t,\mtt u)\in \Std'_e (\rho,d)^{\times 2}$
such that $\Shape(\mtt t)=\Shape(\mtt u)$ and,
if $\b(\mtt t)=(\mtt t_{\le |\rho|}, \mtt s)$ and 
$\b(\mtt u)= (\mtt u_{\le |\rho|}, \mtt r)$, then $\bi^{\mtt s} = \bi^{\mtt r} = \bi$. 
By Lemmas~\ref{lem:preimage} and~\ref{lem:betadeg} (together with Theorem~\ref{thm:BKdim}), 
this sum is equal to 
 $\qdim(R_{\cont(\rho)}^{\Lda_0}) \sum_{\mtt s,\mtt r} q^{\deg(\mtt s)+\deg(\mtt r)}$, where 
 the sum is over all pairs $(\mtt s,\mtt r)$ of standard tableaux 
 such that $\bi^{\mtt s} = \bi^{\mtt r} = \bi$ and $\Shape(\mtt s)=\Shape(\mtt r)$. 
 Dividing by $\qdim(R_{\cont(\rho)}^{\Lda_0})$ and using 
 the isomorphism~\eqref{eq:tensoriso}, we deduce that 
 $\qdim(e'(\bi^{+\k}) C e'(\bi^{+\k})) = \sum_{\mtt s,\mtt r} q^{\deg(\mtt s)+\deg(\mtt r)}$, and the result follows by Lemma~\ref{lem:hooktableaux}.
 \end{proof}

\begin{prop}\label{prop:d1nz}
For every $\bi\in I^{\vn,1}$, we have 
$(y_{|\rho|+e}-y_{|\rho|+1})e'(\bi^{+\k})\ne 0$ in 
$R^{\Lda_0}_{\cont(\rho)+\d}$, and hence 
$y_{|\rho|+1}e'(\bi^{+\k}) \in F(y_{|\rho|+e}-y_{|\rho|+1})e'(\bi^{+\k})$.
\end{prop}

If the first statement of the proposition holds, then, by Lemma~\ref{lem:dimC2}, 
$e'(\bi^{+\k}) C_{\{2\}} e'(\bi^{+\k}) = F(y_{|\rho|+e}-y_{|\rho|+1})e'(\bi^{+\k})$ for any $\bi\in I^{\vn,1}$, and the second statement of the proposition follows.
So we only need to prove the first statement.

\begin{proof}[Proof of Proposition~\ref{prop:d1nz} for $e\ge 3$]
It is well known that the algebra $R^{\Lda_0}_{d+e} \cong \cl H_{|\rho|+e}$ is symmetric 
(see e.g.~\cite[Corollary V.5.4]{SkowronskiYamagata2011}). 
Hence, writing $f=\BK_{|\rho|+e} (f_{\rho,1})$, we see that
$f R^{\Lda_0}_{\cont(\rho)+\d} f= fR^{\Lda_0}_{|\rho|+e}f$ is symmetric as well 
by~\cite[Theorem IV.4.1]{SkowronskiYamagata2011}.
Further, by Proposition~\ref{prop:wt0} and the isomorphism~\eqref{eq:tensoriso}, the algebra $C:=C_{\rho,1}$ is Morita equivalent to $f R^{\Lda_0}_{\cont(\rho)+\d} f$, 
whence $C$ is also symmetric by~\cite[Corollary IV.4.3]{SkowronskiYamagata2011}. 
Thus, $C$ is isomorphic to $C^*:=\Hom_F (C,F)$ as a $(C,C)$-bimodule. 
Note that we have a grading on $C^*$ defined in the usual way: for $n\in \mZ$ and $\xi\in C^*$, we have $\xi\in C^*_{\{n\}}$ if and only if $\xi|_{C_{\mZ\sm \{-n\}}} = 0$. 

 Note that $C$ has only one block, i.e.~is indecomposable as a $(C,C)$-bimodule. This follows, for example, from the fact that the indecomposable algebra $\cl H_{\rho,1}$ is Morita equivalent to $f_{\rho,1} \cl H_{\rho,1} f_{\rho,1}$ 
 by Proposition~\ref{prop:Mor}(\eqref{Mor2}$\Leftrightarrow$\eqref{Mor3}) proved below and hence is Morita equivalent to $C$. (Alternatively, 
 using Proposition~\ref{prop:d1basis},
 one can show without difficulty that $Z(R^{\Lda_0}_{\d}/(R^{\Lda_0}_{\d})_{\{2\}}))_{\{0\}}$ is $1$-dimensional, so 
 $Z(C)_{\{0\}}$ must also be $1$-dimensional because 
 $C/C_{\{2\}} \cong R^{\Lda_0}_{\d}/(R^{\Lda_0}_{\d})_{\{2\}}$.)
By~\cite[Lemma 2.5.3]{BeilinsonGinzburgSoergel1996}, 
it follows that $C^*\cong C\lan n\ran$ as graded $(C,C)$-bimodules 
for some $n\in \mZ$. Since $C= C_{\{0,1,2\}}$ with  
$C_{\{2\}} \ne 0$
and $C^* = C^*_{\{0,-1,-2\}}$,
we have $n=-2$. The graded  isomorphism $C\isoto C^*\lan 2\ran$ of $(C,C)$-bimodules sends $1$ to some element $\xi\in (C^*)_{\{-2\}}$ such that the bilinear form $C\times C \to F$ given by $(a,b)\mapsto \xi(ab)$ is symmetric and non-degenerate. 

Let $\bi\in I^{\vn,1}$. 
Since $e\ge 3$, we can find $\bj\in I^{\vn,1}$ such that 
$j_e \in \{i_e+1, i_e-1\}$. 
Then $\psi_{w_{\bi,\bj}} e(\bj)$ is a non-zero element of 
$(R^{\Lda_0}_\d)_{\{1\}}$ by Proposition~\ref{prop:d1basis}.
Hence, $\Xi(\psi_{w_{\bi,\bj}} e(\bj))\ne 0$, 
so there exists $a\in C_{\{1\}}$ 
such that $\xi( \Xi(\psi_{w_{\bi,\bj}} e(\bj)) a)\ne 0$. Since  
the bilinear form in question is symmetric and 
$\Xi(\psi_{w_{\bi,\bj}} e(\bj))=
e'(\bi^{+\k})\Xi(\psi_{w_{\bi,\bj}} e(\bj))$, we have
$\xi( \Xi(\psi_{w_{\bi,\bj}} e(\bj)) a e'(\bi^{+\k}))\ne 0$, so we may assume that 
$a=e'(\bj^{+\k}) ae'(\bi^{+\k})$. Now the subspace 
$e'(\bj^{+\k}) C_{\{1\}} e'(\bi^{+\k}) = 
\Xi( e(\bj) (R^{\Lda_0}_{\d})_{\{1\}} e(\bi) )$ is $1$-dimensional and is spanned by $\Xi( \psi_{w_{\bj,\bi}} e(\bi) )$ by Proposition~\ref{prop:d1basis}. This implies that 
\[
0 \ne \Xi(\psi_{w_{\bi,\bj}} e(\bj)) \Xi(\psi_{w_{\bj,\bi}} e(\bi))
= \Xi( \psi_{w_{\bi,\bj}} \psi_{w_{\bj,\bi}} e(\bi)) 
= \pm \Xi( y_e e(\bi)) = \pm (y_{|\rho|+e} - y_{|\rho|+1}) e'(\bi^{+\k}),
\]
where the penultimate equality holds by Lemma~\ref{lem:d1gen}\eqref{d1gen1}. 
\end{proof}

\subsection{A calculation involving the Brundan--Kleshchev isomorphism}\label{subsec:BKcalc}

In order to prove Proposition~\ref{prop:d1nz} for $e=2$, we will compute 
$\BK_{|\rho|+e} (T_{|\rho|+1} f_{\rho,1})$. 
For later use in Section~\ref{sec:alt}, we begin with a more general set-up and for now 
 let the integers $e\ge 2$ and $d\ge 1$ be arbitrary. 

Choose and fix an integer $N$ large enough so that 
for every $w\in \fr S_{de}$ and $\bi\in I^{d\d}$ we have 
$\deg(\psi_w e(\bi)) + N >2$. 
Let $\Lda = N(\Lda_0+\cdots +\Lda_{e-1}) \in P_+$.
It follows from Theorem~\ref{thm:basis} that the 2-sided ideal of $R_{d\d}$ generated by
$\{ y_t^N e_{d\d} \mid 1\le t\le de\}$ has a zero component in degree $m$ for all integers $m\le 2$. This ideal is the kernel of the canonical projection
$R_{d\d} \thra R^{\Lda}_{d\d}$, 
which therefore restricts to a vector space isomorphism on each $m$-component 
for $m\le 2$. 
Let $\ol{R}^{\Lda}_{d\d}$ be the quotient of $R^{\Lda}_{d\d}$ by 
the two-sided ideal generated by $\{ e(\bi) \mid \bi\in I^{d\d} \sm \cl E_d\}$ (cf.~the definition of $\ol{R}_{d\d}$ in~\S\ref{subsec:outline}). 
As usual, $\ol{\phantom{A}}$ denotes the natural projections 
$R_{de} \to \ol R_{d\d}$ 
and $R_{de}^{\Lda} \thra \ol R_{d\d}^{\Lda}$. 
As in~\S\ref{subsec:quot}, symbols that would ordinarily represent elements of $R_{de}$, will denote instead their images in $\ol R_{d\d}$. 
In the rest of this subsection and in Section~\ref{sec:alt}, we
abuse notation by identifying 
$(R_{d\d})_{\le 2}$ with $(R^{\Lda}_{d\d})_{\le 2}$ and
$(\ol{R}_{d\d})_{\le 2}$ with $(\ol R^{\Lda}_{d\d})_{\le 2}$. 
In particular, $\ol{\BK_{de}^{\Lda}(T_r)}_{\{0,1,2\}} e_{\d^d}$ 
is viewed as an element of $\ol R_{d\d}$ for $1\le r <de$. 

\begin{lem}\label{lem:BKcalc} Assume that $e=2$. If $1\le r\le d$, we have
\[
\ol{\BK_{2d}^{\Lda} (T_{2r-1})}_{\{0,1,2\}} e_{\d^d} =  
\begin{cases}
(1+y_{2r-1}-y_{2r}) e_{\d^d} & \text{if } \Char F =2, \\ 
\left(-1 + \frac{y_{2r}-y_{2r-1}}{2}\right) \! e_{\d^d} & \text{if } \Char F\ne 2. 
\end{cases}
\]
\end{lem}

\begin{proof}
By Lemma~\ref{lem:psi1} and the statements after Lemma~\ref{lem:UV}, 
$\psi_{2r-1} e_{\d^d} =0$. Note that $e_{\d^d} = e(0101\ldots 01)$. 
Hence, by~\eqref{BK1}, we have 
$\ol{\BK_{2d}^{\Lda} (T_{2r-1})} e_{\d^d} = -P_1 (y_{2r-1},y_{2r}) e_{\d^d}$.
Using the formulas~\eqref{Pideg} and~\eqref{Pinondeg} for $P_1$, 
one concludes the proof by an easy calculation (note that $\xi=-1$ if $\Char F\ne 2$). 
\end{proof}

\begin{proof}[Proof of Proposition~\ref{prop:d1nz} for $e=2$]
As in the statement of the proposition, we consider the case when $d=1$ and assume all the notation of~\S\ref{subsec:d=1}. Note that 
$\bi=(01)$, so $\bi^{+\k}$ is either $(01)$ or $(10)$, and 
$f:=\BK_{|\rho|+e} (f_{\rho,1}) = e'(\bi^{+\k})$. 
It follows from~\eqref{BK1} and the definitions of $\Om$ and $\rot_{\k}$ that 
$
 \BK_{|\rho|+2} (T_{|\rho|+1})_{\{0,1,2\}} f=
\Om(\rot_{\k}(\ol{\BK^{\Lda}_{2} (T_1)}_{\{0,1,2\}}  e_{\d}))
$
and that $\BK_{|\rho|+2} (T_{|\rho|+1}) f$
belongs to the unital subalgebra of 
$fR^{\Lda_0}_{|\rho|+2}f$ generated by the set 
$\{y_{|\rho|+1} f, y_{|\rho|+2} f \}$ 
 and hence to $C$ (note that we have $\psi_{|\rho|+1} f=\rot_{\k} (\Om(\psi_1 e_{\d}))=0$). 
Since $C=C_{\{0,1,2\}}$, 
we see that 
$\BK_{|\rho|+2} (T_{|\rho|+1}) f = \Om(\rot_{\k}(\ol{\BK^{\Lda}_{2} (T_1)}_{\{0,1,2\}} e_{\d}))$.
Assume for contradiction that 
$(y_{|\rho|+2}- y_{|\rho|+1}) f=0$. Using Lemma~\ref{lem:BKcalc}, we deduce that
$\BK_{|\rho|+2} (T_{|\rho|+1}) f = \Om(\rot_{\k} (e_{\d})) = -f$. 
Hence, $T_{|\rho|+1} f_{\rho,1} = -f_{\rho,1}$. 

Note that $T_{|\rho|+1}$ commutes with $f_{\rho,1}$.
Let us view $\cl H_{|\rho|+2}$ and 
 $M:=f_{\rho,1} \cl H_{|\rho|+2} f_{\rho,1}$ as $\cl H_2$-modules with $T_1$ acting via 
 left multiplication by $T_{|\rho|+1}$.
 It follows from the identity just proved that $M$
  is a direct sum of $1$-dimensional simple $\cl H_2$-modules. 
On the other hand, $M$ must be projective because it is a direct summand of 
$\cl H_{|\rho|+2}$, which is a free $\cl H_2$-module with basis 
$\{ T_w\mid w\in \ls{(1^{|\rho|},2)}{\scr D}_{|\rho|+2} \}$. 
This is a contradiction because the only indecomposable projective $\cl H_2$-module is $\cl H_2$ itself (note that $\cl H_2$ is isomorphic to the truncated polynomial algebra $F[x]/(x^2)$). 
\end{proof}

\subsection{Conclusion of the proof of Theorem~\ref{thm:main2}}\label{subsec:endproof}

We return to the case when $d\ge 1$ is arbitrary and 
$\cl H_{\rho,d}$ is a RoCK block of residue $\k$. As usual, write $f=\BK_{|\rho|+de} (f_{\rho,d})$. 
Let $\bi\in I^{(\vn,1)^d}$ (see the definition before Corollary~\ref{cor:hom}), and write 
$e'(\bi) = \sum_{\bj\in I^{\rho,0}} e(\bj\bi)\in fR^{\Lda_0}_{|\rho|+de} f$. 
For $1\le r\le e$, we have 
$(\Om\circ \rot_\k) (y_r e(\bi)) = y_{|\rho|+r} e'(\bi^{+\k})$.
Applying the map $x\mapsto \iota_{|\rho|+e}^{|\rho|+de}{(x)} e'(\bi)$ to 
the second statement of Proposition~\ref{prop:d1nz}, 
we see that 
$y_{|\rho|+1} e'(\bi^{+\k}) \in 
F(y_{|\rho|+e}-y_{|\rho|+1})e'(\bi^{+\k})$. 
Hence, by Corollary~\ref{cor:hom}\eqref{hom2}, 
the graded algebra homomorphism 
$\Xi\colon R^{\Lda_0}_{\d} \wr \fr S_d \to C_{\rho,d}$ 
is surjective, 
whence it is an isomorphism by Proposition~\ref{prop:qdim}.
The proof of Theorem~\ref{thm:main2} is complete.

\section{Two observations}\label{sec:concl}

\subsection{Another formula for the idempotent $f_{\rho,d}$}
Let $\rho$ be a Rouquier core for an integer $d\ge 0$. 
If $\cl O$ is an integral domain, $t\in \cl O^{\times}$ and $1\le r\le d$, let 
$\b_r \colon \cl H_e (\cl O,t)\to \cl H_{|\rho|+ed}(\cl O,t)$
be the unital algebra homomorphism defined by $T_j \mapsto T_{|\rho|+(r-1)e+j}$, $1\le j<e$.
As before, for $n\ge m\ge 0$, we view $\cl H_m (\cl O,t)$ as a subalgebra of $\cl H_n (\cl O,t)$ via the embedding $T_j\mapsto T_j$, $1\le j<m$.  

If $1\le r\le d$, let $b_{\vn,1}^{(r)}= \b_r (b_{\vn,1})$. 
Note that the idempotents $b_{\vn,1}^{(1)},\ldots,b_{\vn,1}^{(d)}$ commute with $f_{\rho,d}$ and with each other.
Turner~\cite[Chapter IV]{Turner2009} considers the idempotent $f_{\rho,d} b_{\vn,1}^{(1)} \cdots b_{\vn,1}^{(d)}$ rather than $f_{\rho,d}$ (cf.~also~\cite[Section 4]{ChuangKessar2002}), 
but the following proposition shows that this makes no difference. 

\begin{prop}\label{prop:falt}
 In any RoCK block $\cl H_{\rho,d}$, we have $f_{\rho,d} b_{\vn,1}^{(1)} \cdots b_{\vn,1}^{(d)} = f_{\rho,d}$. 
\end{prop}

\begin{proof}
As in~\cite[Section 5]{DipperJames1987},
 let $\cl O=F[t]_{(t-\xi)}$, the localisation of the polynomial ring $F[t]$ at the ideal $(t-\xi)F[t]$, and consider the field of fractions $K=F(t)$ of $\cl O$.
For any $n\ge 0$, we have  
$\cl H_n \cong \cl H_n (\cl O,t)/(t-\xi)F[t] \cl H_n (\cl O,t)$.
 Since $t$ is not a root of unity in $K$, the algebra $\cl H_n(K,t)$ is semisimple, with $\{ S^{\lda,K,t} \mid \lda\in \Par(n) \}$ 
being a complete set of non-isomorphic simple modules (see~\cite[Theorem 4.3]{DipperJames1987}). 
For every partition $\lda$ of $n$, let $b_{\lda}$ be the primitive central idempotent of $\cl H_n (K,t)$ such that $b_{\lda} S^{\lda,K,t} = S^{\lda,K,t}$. 
For any $(\pi,l) \in \Bl_e (n)$, let $\tilde{b}_{\pi,l}=\sum_{\lda\in \Par_e (\pi,l)} b_{\lda}$; further, let $\tilde{b}^{(r)}_{\vn,1}=\b_{r} (\tilde{b}_{\vn,1})$. It follows from the results of~\cite[Section 5]{DipperJames1987}
together with~\cite[Lemma 4.6 and Theorem 4.7]{DipperJames1987} that 
$\tilde{b}_{\pi,l} \in \cl H_n (\cl O,t)$ and that 
$b_{\pi,l}$ is the image of $\tilde{b}_{\pi,l}$ under the canonical projection $\cl H_n (\cl O,t)\thra \cl H_n$; further, $b_{\vn,1}^{(r)}$ is the image of $\tilde{b}_{\vn,1}^{(r)}$ under the canonical map $\cl H_{|\rho|+de} (\cl O,t) \thra \cl H_{|\rho|+de}$.   

If $\lda$, $\mu$ and $\nu$ are partitions with $|\lda|=|\mu|+|\nu|$, let $c^{\lda}_{\mu\nu}$ be the corresponding Littlewood--Richardson coefficient 
(see e.g.~\cite[Section 2.8]{JamesKerber1981}).
It follows from~\cite[Proposition 13.7(iii)]{Goldschmidt1993} that the Littlewood--Richardson rule holds for the algebras $\cl H_n (K,t)$: 
if $\lda$, $\mu$ and $\nu$ are partitions such that 
$|\lda|=|\mu|+|\nu|$ and if we identify 
$\cl H_{|\mu|}(K,t) \otimes_K \cl H_{|\nu|} (K,t)$ with the parabolic subalgebra 
generated by $\{ T_j \mid 1\le j<|\mu| \text{ or } |\mu|<j < |\lda|\}$ of $\cl H_{|\lda|}(K,t)$ in the obvious way, then 
the $\cl H_{|\mu|} (K,t) \otimes_K \cl H_{|\nu|}(K,t)$-module 
$S^{\mu,K,t} \otimes_K S^{\nu,K,t}$ appears in the restriction of $S^{\lda,K,t}$ to 
$\cl H_{|\mu|}(K,t) \otimes_K \cl H_{|\nu|} (K,t)$ with multiplicity $c^{\lda}_{\mu\nu}$. 
Let $1\le r\le d$, $\mu\in \Par_e (\rho,r-1)$, $\lda\in \Par_e (\rho,r)$ and 
$\nu\in\Par(e) \sm \Par_e (\vn,1)$. By Lemma~\ref{lem:CK} and the standard combinatorial description of the Littlewood--Richardson coefficients, we have $c^{\lda}_{\mu\nu}=0$, so 
$S^{\mu,K,t} \otimes_K S^{\nu,K,t}$ is not a summand of the restriction of 
$S^{\lda,K,t}$ to $\cl H_{|\rho|+(r-1)e} (K,t) \otimes_K \cl H_e (K,t)$. Therefore, 
$b_{\lda}b_{\mu} \b_{r} (b_{\nu})=0$. Summing over all such $\lda,\mu,\nu$, we deduce that 
$\tilde{b}_{\rho,r-1} \tilde{b}_{\rho,r} (1-\tilde{b}^{(r)}_{\vn,1})=0$, whence 
$\tilde{b}_{\rho,r-1} \tilde{b}_{\rho,r} \tilde{b}^{(r)}_{\vn,1} = \tilde{b}_{\rho,r-1} \tilde{b}_{\rho,r}$. Applying the projection onto $\cl H_{|\rho|+de}$ and using the statements at the end of the previous paragraph, we see that $b_{\rho,r-1} b_{\rho,r} b^{(r)}_{\vn,1} = b_{\rho,r-1} b_{\rho,r}$. The result follows (cf.~\eqref{frhod}).
\end{proof}

\subsection{A condition for Morita equivalence}

Let $\cl H_{\rho,d}$ be a RoCK block.

\begin{prop}\label{prop:Mor}
 The following are equivalent:
\begin{enumerate}[(i)]
 \item\label{Mor1} $f_{\rho,d} D \ne 0$ for all simple $\cl H_{\rho,d}$-modules $D$. 
 \item\label{Mor2} $f_{\rho,d} \cl H_{\rho,d} f_{\rho,d}$ is Morita equivalent to $\cl H_{\rho,d}$.
 \item\label{Mor3} $d< \Char F$ or $\Char F=0$. 
\end{enumerate}
If these statements hold, then $\cl H_{\rho,d}$ is Morita equivalent to 
$\cl H_{\vn,1} \wr \fr S_d$. 
\end{prop}

\begin{proof} 
 For a finite-dimensional algebra $A$, denote by $\ell (A)$ the number of isomorphism classes of simple $A$-modules. 
The equivalence between~\eqref{Mor1} and~\eqref{Mor2} follows from the general properties of idempotent truncations recalled in~\S\ref{subsec:main}.
In view of those properties and Theorem~\ref{thm:main1},
statement~\eqref{Mor2} holds if and only if 
$\ell (\cl H_{\rho,d})=\ell (\cl H_{\vn,1} \wr \fr S_d)$. 
 
Recall that, for an integer $p\ge 2$, a partition $\lda=(\lda_1,\ldots,\lda_s)$ is said to be 
\emph{$p$-restricted} if $\lda_j-\lda_{j+1}<p$ for $1\le j<s$ and $\lda_s<p$. 
The simple $\cl H_n$-modules are parameterised by the $e$-restricted partitions of $n$, and for any 
$(\pi,l)\in \Bl_e (n)$, the simple $\cl H_{\pi,l}$-modules are parameterised by the $e$-restricted elements of $\Par_e (\pi,l)$;
see~\cite[Theorem 7.6]{DipperJames1986} 
and~\cite[Theorem 4.13]{DipperJames1987}.

Suppose that $A$ is a finite-dimensional algebra over a field $k$ such that $k$ is a splitting field for $A$, and let $m=\ell(A)$. 
If $\Char k=p>0$, let $\cl{RP}_k$ be the set of $p$-restricted partitions; if $\Char k=0$, let $\cl {RP}_k$ be the set of all partitions. 
 Then a well-known argument shows that 
$\ell(A\wr \fr S_d)$ is the number of tuples $(\lda^1,\ldots,\lda^m)$ of elements of $\cl{RP}_k$ such that $\sum_{j=1}^m |\lda^j|=d$; see, for example,~\cite[Appendix]{Macdonald1980}.
The number of $e$-restricted partitions in $\Par_e (\vn,1)$ is equal to $e-1$. 
Combining these facts, we see that 
$\ell(\cl H_{\vn,1} \wr \fr S_d)$ is 
the cardinality of the set $Y_d$ of all tuples 
$(\lda^1,\ldots,\lda^{e-1})$ of elements of $\cl{RP}_F$ such that 
$\sum_{j=1}^{e-1} |\lda^j|=d$. 

Each partition $\lda\in \Par_e (\rho,d)$ corresponds to a tuple $(\lda^{(0)},\ldots,\lda^{(e-1)})$, where the definition of $\lda^{(i)}$ is given before Lemma~\ref{lem:preimage}. It is easy to see that such a partition $\lda$ is $e$-restricted if and only if $\lda^{(e-1)}=\vn$ (cf.~\cite[Lemma 4.1(1)]{ChuangTan2002}). 
Hence, $\ell(\cl H_{\rho,d})= |X_d|$ where $X_d$ is the set of all tuples 
$(\lda^1,\ldots,\lda^{e-1})$ of partitions such that $\sum_i |\lda^i|=d$. Observe that $Y_d\subset X_d$.  
Note that, for a prime $p$, all partitions of an integer $n$ are $p$-restricted if and only if $n<p$. Hence, if~\eqref{Mor3} holds, then 
$Y_d=X_d$ and so $\ell (\cl H_{\vn,1} \wr \fr S_d)=\ell (\cl H_{\rho,d})$; otherwise, $\ell (\cl H_{\vn,1} \wr \fr S_d)<\ell (\cl H_{\rho,d})$. This proves the equivalence between~\eqref{Mor2} and~\eqref{Mor3}. 

The last assertion of the proposition is an immediate consequence of Theorem~\ref{thm:main1}. 
\end{proof}

\section{Alternative descriptions of the isomorphism}\label{sec:alt}
In this section, 
we assume all the notation and conventions of~\S\ref{subsec:quot} 
and work again with the algebra $e_{\d^d} \ol{R}_{d\d} e_{\d^d}$ for a fixed $d\ge 0$. We have constructed elements
$\tau_r\in \s_r e_{\d^d} + \iota(S^{(r)}_{\{0\}})$ for $1\le r<d$ satisfying the relations of Theorem~\ref{thm:Theta}. 
We show that such elements $\tau_r$ are in some sense unique and use this fact to give alternative descriptions of $\tau_r$.

\subsection{A uniqueness result}
\begin{lem}\label{lem:tau_unique}
Let $1\le r<d$. 
Suppose that an element $\tau'\in \s_r e_{\d^d} + \iota(S^{(r)}_{\{0\}})$ satisfies the following property: 
\begin{enumerate}[(i)]
\item if $e>2$, then 
$\eta_r (x) \tau' = \tau' \eta_{r+1}(x)$ for all 
$x\in (R^{\Lda_0}_{\d})_{\{0,1\}}$; 
\item if $e=2$, then $(y_{2r}-y_{2r-1})\tau' = \tau' (y_{2(r+1)} - y_{2r+1})$. 
\end{enumerate}
 Then $\tau'=\tau_r$. 
\end{lem}

\begin{proof}
 To simplify notation, we will assume that $r=1$: the proof in the general case is obtained by a straightforward modification of the one below (cf.~Remark~\ref{rem:red}). 
 We write $\tau=\tau_1$. 
By the hypothesis and Theorem~\ref{thm:Theta}\eqref{Theta3}, 
$\tau' = \tau+\iota(a)$ for some 
$a\in S_{\{0\}}^{(1)}$.
Recall that we have an algebra isomorphism $\iota \colon \ol R_{\d}^{\ot d} \isoto \ol R_{\d^d}$, defined after Lemma~\ref{lem:UV}.

In the case when $e=2$, it follows from the hypothesis and Theorem~\ref{thm:Theta}\eqref{Theta3} 
that $(y_2-y_1)\iota(a) = \iota(a)(y_3-y_4)$. 
Moreover, we have $(R_{\d}^{\Lda_0})_{\{0\}}=F\{1\}$ by Proposition~\ref{prop:d1basis}, whence $S_{\{0\}}=1$ by Corollary~\ref{cor:d1_01}, 
so $a$ must be a scalar multiple of the identity. 
By the same Proposition and Lemma~\ref{lem:Rbardiso},  we have 
$(\ol y_2-\ol y_1)\ol e_{\d}\ne 0$ in $\ol R_{\d}$, whence the elements
$(y_2-y_1)e_{\d^d} = \iota( (\ol y_2-\ol y_1) \ol e_{\d} \ot \ol e_{\d}^{\ot d-1})$
and 
$(y_4-y_3) e_{\d^d} = 
\iota( \ol e_{\d} \ot (\ol y_2 -\ol y_1)\ol e_{\d} \ot 
\ol e_{\d}^{\ot d-2})$
are linearly independent. Hence,
$a=0$, so the lemma holds for $e=2$.

Assuming that $e>2$, note that 
$a$ satisfies 
$\eta_1 (x) \iota(a) = \iota(a) \eta_2 (x)$ for all 
$x\in (R^{\Lda_0}_{\d})_{\{0,1\}}$. 
Using Corollary~\ref{cor:d1_01}, we identify
$(R^{\Lda_0}_{\d})_{\{0,1\}}$ with 
$(\ol{R}_{\d})_{\{0,1\}}$.
Then $a=\iota(a'\ot \ol e_{\d}^{\ot d-2})$ for some 
$a'\in (R^{\Lda_0}_{\d})_{\{0\}} \ot (R^{\Lda_0}_{\d})_{\{0\}}$ .
Further, we have $(x\otimes e_{\d}) a' =a' (e_{\d} \ot x)$ for all $x\in (R^{\Lda_0}_{\d})_{\{0,1\}}$. 
We will
prove that $a'=0$ (and hence $\tau=\tau'$) by considering 
$a'(e_{\d} \otimes e(\bi))$ for each $\bi\in I^{\vn,1}$. First, consider the case when 
$i_e=i_{e-1}\pm 1$. Since $\deg(a)=0$, we have 
$a' (e_{\d}\otimes e(\bi) \psi_{e-1}) \in (R^{\Lda_0}_{\d})\otimes (R^{\Lda_0}_{\d})_{\{1\}}$. 
On the other hand,
\begin{align*}
a'(e_{\d} \ot e(\bi) \psi_{e-1}) 
&= a'(e_{\d} \otimes  e(\bi) \psi_{e-1} ) (e_{\d} \ot e(s_{e-1}  \bi))  \\
&= (e(\bi)\psi_{e-1} \ot e_{\d}) a' (e_{\d} \ot e(s_{e-1}  \bi))
\in R^{\Lda_0}_{\d} \ot (R^{\Lda_0}_{\d})_{\{0\}},
\end{align*}
whence $a' (e_{\d} \ot e(\bi)\psi_{e-1})=0$. Now, as $e>2$, it follows easily from Proposition~\ref{prop:d1basis} that the linear map
from 
$(R^{\Lda_0}_{\d})_{\{0\}} e(\bi)$ to $(R^{\Lda_0}_{\d})_{\{1\}} e(s_{e-1} \bi)$
given by $c \mapsto c\psi_{e-1}$ is injective. Hence, $a'(e_{\d} \ot e(\bi))=0$.  

Finally, let $\bi\in I^{\vn,1}$ be arbitrary. 
Using Lemma~\ref{lem:hkdesc}, one can easily show that there exists $\bj\in I^{\vn,1}$ such that $j_e = i_e$ and 
$j_{e-1}\in \{ j_e\pm 1\}$.
By the case considered previously, $a'(e_{\d} \ot e(\bj))=0$. We have 
$\psi_{w_{\bi,\bj}} \psi_{w_{\bj,\bi}} e(\bi) = e(\bi)$ 
(see the proof of Lemma~\ref{lem:d1gen}), and hence 
\[
 a'(e_{\d} \ot e(\bi))=a'(e_{\d}\ot \psi_{w_{\bi,\bj}} e(\bj)
\psi_{w_{\bj,\bi}})=   
(\psi_{w_{\bi,\bj}} e(\bj)\ot e_{\d}) a' 
(e_{\d} \ot e(\bj)) (e_{\d} \ot \psi_{w_{\bj,\bi}} e(\bi)) =0. \qedhere
\]
\end{proof}

\subsection{An explicit formula for $\tau_r$}
Using Lemma~\ref{lem:tau_unique}, we can give (without a complete proof) an explicit formula for $\tau=\tau_1$ when $d=2$. For $e=2$, we have already shown in the proof of Theorem~\ref{thm:Theta}\eqref{Theta3} that 
$\tau= (\s+1)e(0101)$, where we set $\s:=\s_1$. 
For arbitrary $e$, 
\begin{equation}\label{tauexpl}
 \tau =  \s e_{\d,\d} + 
 \sum_{\substack{\bi,\bj\in I^{\vn,1} \\ i_e=j_e}} (-1)^{i_e+1} 
\eta_1 (\psi_{w_{\bj,\bi}}) \eta_2 (\psi_{w_{\bi,\bj}}) e(\bi\bj),
\end{equation}
where $i_e$ is viewed as an element of $\mZ$ via the identification of $I$ with $\{0,1,\ldots,e-1\}$. 
For example, for $e=3$ we have 
\[
 \tau = \s e_{\d,\d} - e(012012) + e(021021). 
\]

It is possible to show that the right-hand side $\tau'$ of~\eqref{tauexpl}
satisfies $\eta_1 (x) \tau' = \tau' \eta_2 (x)$ for all $x\in (R^{\Lda_0}_{\d})_{\{0,1\}}$ by technical calculations using the defining relations of $\ol{R}_{2\d}$, and consequently that $\tau'=\tau$ by Lemma~\ref{lem:tau_unique}. These calculations are not included in the present paper, 
but equivalent calculations have been independently done by Kleshchev and Muth~\cite[Section 6]{KleshchevMuth2015} for KLR algebras of all untwisted affine ADE types (cf.~Remark~\ref{rem:KM}).
For arbitrary $r$ and $d$, we have $\tau_r = 
\zeta_{r,2} (\ol e_{\d}^{\ot r-1} \ot \tau \ot \ol e_{\d}^{\ot d-r-1})$ (cf.~Remark~\ref{rem:red}).

\subsection{A formula for $\tau_r$ via Hecke generators}

Let $e\ge 2$ and $d\ge 1$, and assume all the notation and conventions introduced 
in~\S\ref{subsec:BKcalc} prior to Lemma~\ref{lem:BKcalc}. 
From now on, we identify $H^{\Lda}_{de}$ with $R^{\Lda}_{de}$ via the isomorphism
$\BK_{de}^{\Lda}$. 
As usual, for each $w\in \fr S_{de}$ let 
$T_w = T_{j_1}\cdots T_{j_m}\in H^{\Lda}_{de}$, where $w=s_{j_1} \cdots s_{j_m}$ 
is a reduced expression. 
Recall the elements $w_r\in \fr S_{de}$ defined by~\eqref{w_r}.

\begin{prop}\label{prop:tauH}
We have 
$\tau_r = (-1)^e (e_{\d^d} \ol T_{w_r} e_{\d^d})_{\{0\}}$ 
whenever $1\le r<d$. 
\end{prop}

It follows from~\eqref{BK1} that 
$(e_{\d^d} \ol T_{w_r} e_{\d^d})_{\{0\}} = 
\zeta_{r,2} ((e_{\d,\d} \ol T_{w_1} e_{\d,\d})_{\{0\}})$. 
Hence, the general case of Proposition~\ref{prop:tauH} follows from the case when 
$d=2$ and $r=1$ (cf.~Remark~\ref{rem:red}). We begin the proof of the proposition in that case with two lemmas.  

Recall the power series $P_i, Q_i\in F[[y,y']]$ given by~\eqref{Pideg}--\eqref{Sinondeg}, and write $P^{(0)}_i$ and $Q^{(0)}_i$ for the constant coefficients of $P_i$ and $Q_i$ respectively. In particular, if $\xi=1$, then 
\begin{equation} \notag 
Q_i^{(0)}=
\begin{cases}
1 & \text{if } i=0, \\
1-i^{-1} & \text{if } i\notin \{0,1,-1\}, \\
2 & \text{if } e\ne 2 \text{ and } i=-1, \\
1 & \text{if } e\ne 2 \text{ and } i=1, \\
1 & \text{if } e=2 \text{ and } i=1,
\end{cases}
\end{equation}
and if $\xi \ne 1$, then 
\begin{equation} \notag 
Q_i^{(0)} =
\begin{cases}
1- \xi & \text{if } i=0, \\
\xi (\xi^{i-1}-1)/(\xi^i -1) & \text{if } i\notin \{0,1,-1\}, \\
\xi (\xi^{-2}-1)/ (\xi^{-1}-1)^2 & \text{if } e\ne 2 \text{ and } i=-1, \\
1 & \text{if } e\ne 2 \text{ and } i=1, \\
1/(\xi-1) & \text{if } e=2 \text{ and } i=1. 
\end{cases}
\end{equation}
One easily checks (using the fact that $e$ is prime if $\xi=1$) that 
\begin{equation}\label{prodS}
 Q_0^{(0)} Q_1^{(0)} \cdots Q_{e-1}^{(0)} = -1.
\end{equation}

\begin{lem}\label{lem:Twrin}
Assume that $d=2$. Then  
$(-1)^e (e_{\d,\d} \ol T_{w_1} e_{\d,\d})_{\{0\}} 
\in \s_1 e_{\d,\d} + \iota(S_{\{0\}})$. 
\end{lem}

\begin{proof}
 Let $\bj^{(1)},\bj^{(2)}\in I^{\vn,1}$, and let 
 $w_1 = s_{t_{e^2}} \cdots s_{t_2} s_{t_1}$ be a reduced expression for $w_1$, so that 
 \begin{equation}\label{Twrin1}
 (e_{\d,\d} \ol T_{w_1} e(\bj^{(1)}\bj^{(2)}))_{\{0\}} = 
 (e_{\d,\d} \ol T_{i_{e^2}} \cdots \ol T_{i_{1}} e(\bj^{(1)}\bj^{(2)}))_{\{0\}}.
 \end{equation}
 Applying~\eqref{BK1} to the terms $\ol T_{i_1}, \ldots, \ol T_{i_{e^2}}$ in this order and using the relations~\eqref{rel:eid} and~\eqref{rel:epsi}, we decompose 
 the right-hand side of~\eqref{Twrin1} as a sum of $2^{e^2}$ terms; 
 these terms correspond to the choice of either the first or the second summand of~\eqref{BK1} at each of the $e^2$ steps.
 The term corresponding to choosing the first summand 
 in every case is $a_{\{0\}}$, where
 $a=a_{e^2} \cdots a_1$, 
 $a_k = \psi_{t_k} Q_{i^{(k)}_{t_k}-i^{(k)}_{t_k+1}} \big( y_{i^{(k)}_{t_k}}, 
 y_{i^{(k)}_{t_{k+1}}}\big)$ for each $k$,  and 
 $\bi^{(k)} = (i^{(k)}_1,\ldots,i^{(k)}_{2e}) = 
 s_{t_{k-1}} \cdots s_{t_1} (\bj^{(1)} \bj^{(2)})$. 
 By Theorem~\ref{thm:basis}\eqref{basis2}, 
 the sum $b$ of the other $2^{e^2}-1$ 
 terms belongs to 
 $(\sum_{w\in \fr S_{2e}, \, w<w_1} e_{\d,\d} \psi_w \ol R_{\d,\d})_{\{0\}} 
 \subset (\ol{R}_{\d,\d})_{\{0\}} = \iota(S_{\{0\}})$, where the containment follows from 
 the claim in the proof of Lemma~\ref{lem:Bd}. 
 
 Using Lemma~\ref{lem:Bd}\eqref{Bd2}, since each of $\bj^{(1)},\bj^{(2)}$ is a permutation of $(0,\ldots,e-1)$, 
 we have 
 \[
 a_{\{0\}} = \Big(\prod_{i,j\in I} Q^{(0)}_{i-j}\Big) \psi_{w_1} e(\bj^{(1)}\bj^{(2)})
 = (-1)^e \s_{1} e(\bj^{(1)} \bj^{(2)})
 \]
 by~\eqref{prodS} and the definition of $\s_1$. Therefore, 
 $(e_{\d,\d} \ol T_{w_1} e(\bj^{(1)}\bj^{(2)}))_{\{0\}} = a_{\{0\}} + b \in 
 (-1)^e \s_1 e_{\d,\d} + \iota(S_{\{0\}})$, and the lemma follows by summing over all
 $\bj^{(1)}, \bj^{(2)}\in I^{\vn,1}$. 
\end{proof}

\begin{lem}\label{lem:Xcomm}
Let $1\le r\le d$. The element $\ol{X}_{(r-1)e+1} e_{\d^d}$ of $\ol{R}_{d\d}^{\Lda}$ 
 commutes with $e_{\d^d}$ and with each of
$\ol T_j e_{\d^d}$ whenever $(r-1)e+1\le j<de$. Also, 
$\ol T_{(r-1)e+t}$ commutes with $e_{\d^d}$ if $1\le t<e$. 
\end{lem}

\begin{proof}
 By~\eqref{BK2}, the element 
$\ol{X}_{(r-1)e+1}$ belongs to the 
subalgebra of $\ol{R}^{\Lda}_{de}$ generated by 
$\{ \ol{e(\bi)} \mid \bi\in I^{d\d}\} \cup \{ \ol{y}_{(r-1)e+1} \}$, and 
each $\ol T_l$ ($1\le l<de$) belongs to the subalgebra generated by 
$\{ \ol{e(\bi)}\mid \bi\in I^{d\d} \}\cup \{\ol\psi_l, \ol{y}_l, \ol{y}_{l+1} \}$.  
Hence, each of the elements $\ol{X}_{(r-1)e+1}$ and $\ol T_{(r-1)e+t}$ ($1\le t<e$) 
commutes with $e_{\d^d}$.  

By the defining relations of $H^{\Lda}_{de}$, 
it is clear that 
$\ol{X}_{(r-1)e+1}$ commutes with $\ol T_{j}$ 
for $(r-1)e+1<j<de$. 
Thus, it only remains to show that 
$\ol{X}_{(r-1)e+1} \ol T_{(r-1)+1} e_{\d^d} = \ol T_{(r-1)e+1} \ol{X}_{(r-1)+1} e_{\d^d}$. By the defining relations~\eqref{rel:eid}--\eqref{rel:ycomm}, $\ol{X}_{(r-1)e+1}e_{\d^d}$ commutes with $y_{(r-1)e+1}$, $y_{(r-1)e+2}$ and the elements $\ol{e(\bi)}$ for $\bi \in I^{d\d}$. Moreover, we have $\ol \psi_{(r-1)e+1} e_{\d^d}=0$ due to Lemma~\ref{lem:sh1} and the statement after Lemma~\ref{lem:UV}. 
By an observation in the previous paragraph, the required identity follows. 
\end{proof}

\begin{proof}[Proof of Proposition~\ref{prop:tauH}]
 By the discussion following the statement of the proposition, we may (and do) 
 assume that $d=2$ and $r=1$.
Whenever $1\le t\le 2$, define elements 
$\wt X_{t,l}\in \ol{R}^{\Lda}_{d\d}$ for $1\le l\le e$ as follows:
\begin{enumerate}[(i)]
 \item if $\xi=1$, then $\wt X_{t,1}=0$ and 
$\wt X_{t,l+1} = \ol T_{(t-1)e+l} \wt X_{t,l} \ol T_{(t-1)e+l} + 
\ol T_{(t-1)e+l}$ for $1\le l<e$;
 \item if $\xi\ne 1$, then $\wt X_{t,1} =1$ and 
$\wt X_{t,l+1} = 
\xi^{-1} \ol T_{(t-1)e+l} \wt X_{t,l} \ol T_{(t-1)e+l}$ for $1\le l<e$. 
\end{enumerate}
We claim that
\begin{align}
 \label{Xti1} \ol X_{(t-1)e+l} e_{\d,\d} 
&= (\wt X_{t,l} + \ol X_{(t-1)e+1}) e_{\d,\d} & \text{if } \xi=1, \\
\label{Xti2}
\ol X_{(t-1)e+l} e_{\d,\d} &= 
\wt X_{t,l} \ol X_{(t-1)e+1} e_{\d,\d} & \text{if } 
\xi\ne 1\phantom{,}
 \end{align}
for $1\le l\le e$ and $1\le t\le 2$. These equalities can be proved by induction on $l$: the base case $l=1$ is clear, and, for $\xi\ne 1$, the inductive step $l\to l+1$ is proved, using the defining relations of $H^{\Lda}_{2e}$ and Lemma~\ref{lem:Xcomm} as follows:
\begin{align*}
 \ol X_{(t-1)e+l+1} e_{\d,\d} &= 
 \xi^{-1} \ol T_{(t-1)e+l} \ol X_{(t-1)e+l} 
 \ol T_{(t-1)e+l} e_{\d,\d} \\
 &= 
\xi^{-1} \ol T_{(t-1)e+l} \wt X_{t,l} 
\ol X_{(t-1)e+1}  e_{\d,\d} \ol T_{(t-1)e+l}  e_{\d,\d} \\
&= \xi^{-1} \ol T_{(t-1)e+l} \wt X_{t,l} \ol T_{(t-1)e+l} \ol X_{(t-1)e+1} e_{\d,\d}\\
&=  \wt X_{t,l+1}  \ol X_{(t-1)e+1} e_{\d,\d}.
\end{align*}
The proof of the inductive step for $\xi=1$ is similar and is left as an exercise.

Let $t\in \{1,2\}$.
It follows from Lemma~\ref{lem:Xcomm} that $\wt X_{t,l}$ commutes with $\ol X_{(t-1)e+1} e_{\d,\d}$ for $1\le l\le e$.
Using this observation, Equations~\eqref{Xti1}--\eqref{Xti2}
 and the fact that 
$(\ol{X}_{(t-1)e+1}- \hat 0) e_{\d,\d}$ is nilpotent for $t=1,2$ 
(cf.~Theorem~\ref{thm:BKiso_detailed}\eqref{BKid2} and~\eqref{bari}), we see that 
\begin{equation}\label{genev}
\eta_t (e(\bi)) \ol{R}^{\Lda}_{d\d} e_{\d,\d} = 
\{ v\in \ol R^{\Lda}_{d\d} e_{\d,\d} \mid (\wt X_{t,l} - \hat i_l)^L v = 0 \text{ for } L \gg 0 \text{ and all } l=1,\ldots,e \}
\end{equation}
whenever $\bi\in I^{\vn,1}$. 
For  $1\le l<e$, we have 
\begin{equation}\label{Twr}
T_{l} T_{w_1} = T_{s_{l} w_1} = T_{w_1 s_{e+l}}= T_{w_1} T_{e+l}
\end{equation}
since $\ell(s_{l}w_1) = \ell(w_1) +1$. It follows easily by an inductive argument that 
\begin{equation}\label{Xti}
\wt X_{1,l} \ol T_{w_1} = \ol T_{w_1} \wt X_{2,l}
\end{equation} 
for all $l=1,\ldots,e$.
By~\eqref{genev} and~\eqref{Xti}, 
we have $e_{\d,\d} \ol T_{w_1} \eta_{2}(e(\bi)) \ol{R}^{\Lda}_{d\d} e_{\d,\d}\subset 
\eta_1 (e(\bi)) \ol{R}^{\Lda}_{d\d} e_{\d,\d}$, whence 
$e_{\d,\d} (\ol T_{w_1})_{\{0\}} \eta_{2} (e(\bi)) = 
\eta_1 (e(\bi)) (\ol{T}_{w_1})_{\{0\}} 
\eta_{2} (e(\bi))$ for all $\bi\in I^{\vn,1}$.
A similar argument (with~\eqref{genev} replaced by an analogous statement where $e_{\d,\d} \ol{R}^{\Lda}_{d\d}$ is viewed as a right module over $\ol R_{d\d}$) 
establishes the first equality in the following equation:
\begin{equation}\label{etaiT}
\eta_1 (e(\bi)) (\ol{T}_{w_1})_{\{0\}} e_{\d,\d} = 
\eta_1 (e(\bi)) (\ol{T}_{w_1})_{\{0\}} \eta_{2} (e(\bi))
= e_{\d,\d} (\ol{T}_{w_1})_{\{0\}} \eta_2( e(\bi) ).
\end{equation}

Assume first that $e>2$. 
By~\eqref{BK1}, we have 
\begin{equation}\label{const}
 \psi_{(t-1)e+l} \eta_t (e(\bi)) = 
((\ol T_{(t-1)e+l})_{\{0,1\}} +P^{(0)}_{i_l-i_{l+1}}) (Q^{(0)}_{i_l-i_{l+1}})^{-1} 
\eta_t (e(\bi))
\end{equation}
whenever $1\le l<e$
since the left-hand side is homogeneous of degree $0$ or $1$.
Let $\tau'= (-1)^e (e_{\d,\d} \ol{T}_{w_1} e_{\d,\d})_{\{0\}} \in 
(\ol R^{\Lda}_{d\d})_{\{0\}} = (\ol R_{d\d})_{\{0\}}$.
Recall that $e_{\d,\d} \ol R_{d\d} e_{\d,\d}$ is 
nonnegatively graded by Lemma~\ref{lem:Bd}\eqref{Bd3}.
It follows from~\eqref{etaiT} that 
$\eta_1 (e(\bi)) \tau' = \tau' \eta_2 (e(\bi))$ for all 
$\bi\in I^{\vn,1}$. By~\eqref{const},~\eqref{Twr} and degree considerations, 
we have $\eta_1 (\psi_l e(\bi)) \tau' = \tau' \eta_{2} (\psi_l e(\bi))$ whenever 
$1\le l<e$ and $\bi\in I^{\vn,1}$.
By Lemmas~\ref{lem:tau_unique} and~\ref{lem:Twrin} together with the fact that $(R^{\Lda_0}_{\d})_{\{0,1\}}$ is contained in the subalgebra generated by the elements of the form $e(\bi)$ and $\psi_l e(\bi)$ for $e(\bi) \in I^{\vn,1}$ and $1\le l<e$, 
we have $\tau_1=\tau'$.

Finally, consider the case when $e=2$.
By Lemma~\ref{lem:BKcalc}, we have 
\[
(y_{2t}-y_{2t-1}) e_{\d,\d} = 
\begin{cases}
((\ol T_{2t-1})_{\{0,1,2\}} +1 )e_{\d,\d} & \text{if } \Char F=2, \\
(2 (\ol T_{2t-1})_{\{0,1,2\}} +1) e_{\d,\d} & \text{if } \Char F\ne 2. 
\end{cases}
\]
for $t=1,2$. Using this formula, Equation~\eqref{Twr}, Lemma~\ref{lem:Xcomm} and 
degree considerations, we obtain 
$(y_2 - y_1) (e_{\d,\d} \ol T_{w_1} e_{\d,\d})_{\{0\}} =
(e_{\d,\d} \ol T_{w_1} e_{\d,\d})_{\{0\}} (y_4-y_3)$, and hence 
$\tau_1 = (e_{\d,\d} \ol T_{w_1} e_{\d,\d})_{\{0\}}$ by Lemmas~\ref{lem:tau_unique} and~\ref{lem:Twrin}. 
\end{proof}

Let $\cl H_{\rho,d}$ be a RoCK block of residue $\k$.
We identify $R^{\Lda_0}_{|\rho|+de}$ with 
$\cl H_{|\rho|+de}$ via the isomorphism $\BK_{|\rho|+de}$. 
 Thus, $\cl H_{|\rho|+de}$ becomes a graded algebra,
 and $C=C_{\rho,d}$ becomes a graded subalgebra of $\cl H_{|\rho|+de}$.
 If $1\le r<d$, define $w'_r = \prod_{j=1}^e (|\rho|+(r-1)e+j, |\rho|+re+j)\in \fr S_{|\rho|+de}$ (cf.~\eqref{w_r}). 

 \begin{prop}\label{prop:Xires}
 The restriction of the graded algebra isomorphism 
 $\Xi \colon \cl H_{\vn,1} \wr \fr S_d\isoto C$ defined in Section~\ref{sec:surj} 
 to $F\fr S_d$ may be described as follows: 
 $\Xi(s_r) = (-1)^e (f_{\rho,d}T_{w'_r}f_{\rho,d})_{\{0\}}$
 whenever $1\le r<d$. 
 \end{prop}
 
 \begin{proof}
 We may assume that the integer $N$ determining $\Lda$ in~\S\ref{subsec:BKcalc} 
 is chosen to be large enough so that 
 $R^{\Lda_0}_{|\rho|+de}= (R^{\Lda_0}_{|\rho|+de})_{\le 2N-2}$.
 By definition, $\Xi(s_r) = \Om (\rot_{\k} (\tau_r))$. 
 We define $\wh R^{\Lda}_{d\d}$ to be the quotient of $R^{\Lda}_{d\d}$ by the two-sided ideal generated by $\{ e(\bi) \mid \bi\in I^{d\d} \sm \cl E_{d,\k}\}$ and identify $(\wh R_{d\d})_{\{0,1,2\}}$ with 
  $(\wh R_{d\d}^{\Lda})_{\{0,1,2\}}$ (cf.~\S\ref{subsec:BKcalc}). 
 It follows from~\eqref{BK1} that 
 the automorphism of $R_{d\d}^{\Lda}$ given by~\eqref{rot1} 
 fixes $T_k e_{d\d}$ whenever $1\le k<de$, so by Proposition~\ref{prop:tauH} we have 
 $\rot_{\k} (\tau_r) = (-1)^e \rot_\k ((e_{\d^d}\ol{T}_{w_r} e_{\d^d})_{\{0\}}) = 
 (-1)^e e_{\d^d}(\wh{T}_{w_r})_{\{0\}} e_{\d^d}$,
 where $\wh{\phantom{a}} \colon R_{d\d} \thra \wh R_{d\d}$ is the natural projection. 
 By the choice of $N$, the map
 $\om\colon R_{d\d} \to R^{\Lda_0}_{\cont(\rho),d\d}$ of Proposition~\ref{prop:im_om} 
 induces a homomorphism
 $\om^\Lda \colon R_{d\d}^{\Lda} \to R^{\Lda_0}_{\cont(\rho),d\d}$.
 It follows from~\eqref{BK1} that 
 $\om^{\Lda} (T_k e_{d\d}) = 
 \iota_{|\rho|}^{|\rho|+de} (e_{\cont(\rho)}) T_{|\rho|+k} e_{\cont(\rho)+d\d}$
 whenever $1\le k< de$. Therefore, 
  $\om^{\Lda} (T_{w_r} e_{d\d}) = 
   \iota_{|\rho|}^{|\rho|+de} (e_{\cont(\rho)}) T_{w'_r} e_{\cont(\rho)+d\d}$, 
   and so, as $\Om(\wh e_{\d^d})=f_{\rho,d}$, we have 
   $\Om (e_{\d^d} (\wh T_{w_r})_{\{0\}} e_{\d^d})= (f_{\rho,d} T_{w'_r} f_{\rho,d})_{\{0\}}$.
 \end{proof}

\begin{thm}\label{thm:Hecke}
Suppose that $\xi=1$, and let $f=f_{\rho,d}$. Then we have a graded algebra isomorphism
$\cl H_{\rho,0} \ot (\cl H_{\vn,1} \wr \fr S_d) \isoto f\cl H_{|\rho|+de} f$ given as follows:
\begin{align*}
 a\ot (b_{\vn,1}^{\ot r-1} \ot T_l b_{\vn,1} \ot b_{\vn,1}^{\ot d-r}) 
 &\mapsto a T_{|\rho|+(r-1)e+l} f &\text{ for } 1\le r\le d, \, 1\le l<e,  \\
 a\ot s_r &\mapsto a(fT_{w'_r} f)_{\{0\}} &\text{ for } 1\le r<d
\end{align*}
for all $a\in \cl H_{\rho,0}$. 
\end{thm}

\begin{proof}
Due to~\eqref{eq:tensoriso} and Proposition~\ref{prop:BKiso},
it is enough to show that we have an isomorphism from 
$\cl H_{\vn,1} \wr \fr S_d$ onto $C$ given by
\begin{align} 
b_{\vn,1}^{\ot r-1} \ot T_l b_{\vn,1} \ot b_{\vn,1}^{\ot d-r}
 &\mapsto T_{|\rho|+(r-1)e+l} f &\text{ for } 1\le r\le d, \, 1\le l<e, \label{Hecke1} \\
s_r &\mapsto (fT_{w'_r} f)_{\{0\}} &\text{ for } 1\le r<d. \label{Hecke2}
\end{align}
We identify $\cl H_{\vn,1}$ with $R^{\Lda_0}_{\d}$ via $\BK_e$, so that 
$b_{\vn,1} = e_{\d}$. 
For $1\le r\le d$, $1\le l<e$ and $\bi\in I^{\d}$, we have 
\begin{align*}
\Xi( e_{\d}^{\ot r-1} \ot e(\bi) \ot e_{\d}^{\ot d-r}) 
&=\sum e(\bj \bi^{(1)} \ldots \bi^{(r-1)} (\bi^{+\k}) \bi^{(r+1)} \ldots 
\bi^{(d)}), \\
\Xi(e_{\d}^{\ot r-1} \ot \psi_l e_{\d} \ot e_\d^{\ot d-r} ) &= 
\psi_{|\rho|+(r-1)e+l} 
\sum e(\bj \bi^{(1)} \ldots \bi^{(r-1)} (\bi^{+\k}) \bi^{(r+1)} \ldots 
\bi^{(d)}), \\
\Xi(e_{\d}^{\ot r-1} \ot (y_{l}-y_{l+1}) e_{\d} \ot e_\d^{\ot d-r} )
&= \\
& \hspace{-30mm}
(y_{|\rho|+(r-1)e+l} - y_{|\rho|+(r-1)e+l+1}) 
\sum e(\bj \bi^{(1)} \ldots \bi^{(r-1)} (\bi^{+\k}) \bi^{(r+1)} \ldots 
\bi^{(d)}),
\end{align*}
where each sum is over all $\bj\in I^{\rho,0}$ and $\bi^{(1)},\ldots,\bi^{(r-1)}, \bi^{(r+1)}, \ldots,\bi^{(d)}\in I^{\vn,1}_{+\k}$. 
Hence, by Theorem~\ref{thm:BKiso_detailed}\eqref{BKid1}\eqref{BKid3}, 
we have 
$\Xi(b_{\vn,1}^{\ot (r-1)} \ot T_l b_{\vn,1} \ot b_{\vn,1}^{\ot (d-r)})
= T_{|\rho|+(r-1)e+l} f$. 
By Proposition~\ref{prop:Xires}, it follows that the composition of $\Xi$ 
with the automorphism of 
$\cl H_{\vn,1} \wr \fr S_d$ that sends $s_r$ to $(-1)^e s_r$ for all $r$ and 
is the identity on $\cl H_{\vn,1}^{\ot d}$ is given by~\eqref{Hecke1}--\eqref{Hecke2}. 
\end{proof}

\end{document}